\documentclass[aos,preprint]{imsart} 
\usepackage[total={5in, 7.9in}]{geometry}
\usepackage[authoryear]{natbib}
\usepackage[utf8]{inputenc}
\usepackage{bbm}
\usepackage{dutchcal}
\usepackage{paralist}
\usepackage{mathrsfs}
\usepackage{verbatim}
\usepackage{amsmath}
\usepackage{amsthm}
\usepackage{amssymb}
\usepackage{mathtools}
\usepackage{bm}
\usepackage{multirow}
\usepackage{graphicx}
\usepackage{float}
\usepackage{comment}
\usepackage{amsfonts}

\usepackage{color}
\usepackage[dvipsnames]{xcolor}
\usepackage[colorlinks=true,linkcolor=blue,citecolor=blue,pdfborder={0 0 0}]{hyperref}

\allowdisplaybreaks

\numberwithin{equation}{section}

\theoremstyle{plain}
\newtheorem{Assump}{Assumption}[section]

\newtheorem{Theo}[Assump]{Theorem}
\newtheorem{Prop}[Assump]{Proposition}
\newtheorem{Lemma}[Assump]{Lemma}

\theoremstyle{remark}

\newtheorem{Remark}[Assump]{Remark}
\newtheorem{example}[Assump]{Example}

\everymath{ % Adjust maths subscript height
   \fontdimen16\textfont2=1.0pt
   \fontdimen17\textfont2=1.0pt
}

\newcommand{\ca}{\varphi}
\newcommand{\cB}{\bm{B}}
\newcommand{\cA}{\bm{A}}
\newcommand{\cb}{\psi}

\newcommand{\bbl}{\mathbbm{l}}
\newcommand{\dto}{\Rightarrow}
\newcommand{\pto}{\xrightarrow{p}}
\newcommand{\mathph}[1]{\mathrel{\phantom{#1}}}
\newcommand{\ang}[1]{\langle{#1}\rangle}
\newcommand{\floor}[1]{\lfloor{#1}\rfloor}
\newcommand{\ceil}[1]{\lceil{#1}\rceil}

\newcommand{\op}[1]{o_{\Prob}({#1})}

\newcommand*\diff{\mathop{}\!\mathrm{d}}
\newcommand{\ind}{\mathbbm{1}}
\newcommand{\ep}{\epsilon}

\newcommand{\ta}{\widetilde{a}}
\newcommand{\ba}{\bm{a}}

\newcommand{\bbu}{\mathbbm{u}}

\newcommand{\tc}{\widetilde{c}}
\newcommand{\cC}{\mathcal{C}}

\newcommand{\hC}{\hat{C}}
\newcommand{\dC}{\dot{C}}
\newcommand{\tbbC}{\tilde{\mathbbm{C}}}
\newcommand{\bbC}{\mathbbm{C}}
\newcommand{\cD}{\mathcal{D}}

\newcommand{\be}{\bm{e}}

\newcommand{\cF}{\mathcal{F}}
\newcommand{\bF}{\bm{F}}

\newcommand{\hF}{\hat{F}}
\newcommand{\bh}{\bm{h}}
\newcommand{\bI}{\bm{I}}

\newcommand{\ck}{\mathcal{k}}
\newcommand{\cK}{\mathcal{K}}
\newcommand{\cl}{\mathcal{l}}
\newcommand{\bM}{\bm{M}}

\newcommand{\tM}{\tilde{M}}

\newcommand{\tbM}{\widetilde{\bm{M}}}
\newcommand{\utbM}{\underline{\widetilde{\bm{M}}}}

\newcommand{\ubM}{\underline{\bm{M}}}
\newcommand{\uutbM}{\underline{\underline{\widetilde{\bm{M}}}}}

\newcommand{\bbR}{\mathbbm{R}}
\newcommand{\cT}{\mathcal{T}}
\newcommand{\bu}{\bm{u}}
\newcommand{\bU}{\bm{U}}
\newcommand{\cU}{\mathcal{U}}

\newcommand{\hU}{\hat{U}}
\newcommand{\tU}{\tilde{U}}
\newcommand{\tbU}{\tilde{\bm{U}}}
\newcommand{\hbU}{\hat{\bm{U}}}
\newcommand{\bv}{\bm{v}}
\newcommand{\cV}{\mathcal{V}}
\newcommand{\cW}{\mathcal{W}}
\newcommand{\bx}{\bm{x}}
\newcommand{\bX}{\bm{X}}

\newcommand{\uubx}{\underline{\underline{\bm{x}}}}

\newcommand{\cX}{\mathcal{X}}

\newcommand{\tZ}{\tilde{Z}}
\newcommand{\bbZ}{\mathbbm{Z}}

\newcommand{\bbW}{\mathbbm{W}}
\newcommand{\bbV}{\mathbbm{V}}
\newcommand{\bbG}{\mathbbm{G}}
\newcommand{\bbN}{\mathbbm{N}}

\newcommand{\SC}{{\mathrm{SC}}}

\newcommand{\sumb}{\sum_{i=1}^{b}}

\newcommand{\N}{\mathbbm{N}}
\newcommand{\R}{\mathbbm{R}}
\newcommand{\Z}{\mathbbm{Z}}

\DeclareMathOperator{\Var}{Var}
\DeclareMathOperator{\Cov}{Cov}
\DeclareMathOperator*{\argmin}{arg\,min}
\DeclareMathOperator*{\argmax}{arg\,max}

\newcommand{\Exp}{\mathbb{E}}
\newcommand{\Prob}{\mathbb{P}}

\newcommand{\ctheta}{\check{\theta}}
\newcommand{\cL}{\mathcal{L}}
\newcommand{\hcL}{\hat{\mathcal{L}}}
\newcommand{\tcL}{\tilde{\mathcal{L}}}
\newcommand{\ccL}{\check{\mathcal{L}}}

\newcommand{\trho}{\tilde{\rho}}

\newcommand{\tb}{\tilde{b}}
\newcommand{\bw}{\bm{w}}

\newcommand{\cbu}{\check{\bm{u}}}
\newcommand{\tRSS}{\widetilde{RSS}}
\newcommand{\hRSS}{\widehat{RSS}}

\begin{document}
\begin{frontmatter}

\title{Multiple block sizes and overlapping blocks for multivariate time series extremes}
\runtitle{Multiple block sizes and overlapping blocks}

\author{\fnms{Nan} \snm{Zou,}\corref{}\ead[label=e1]{nan.zou@utoronto.ca}}
\author{\fnms{Stanislav} \snm{Volgushev\thanksref{T2}}\corref{}\ead[label=e2]{stanislav.volgushev@utoronto.ca}}
\and
\author{\fnms{Axel} \snm{B\"ucher\thanksref{T1}}\corref{}\ead[label=e3]{axel.buecher@hhu.de}}

\runauthor{Zou, Volgushev and B\"ucher}

\address{
Heinrich-Heine-Universität D\"usseldorf \\  
Mathematisches Institut \\ 
Universit\"atsstr.~1 \\ 
40225 D\"usseldorf, Germany \\ 
\printead{e3}}

\address{
Department of Statistical Sciences \\
University of Toronto \\
100 St. George St. \\
Toronto, M5S 3G3, Ontario, Canada \\ 
\printead{e1} \\  \phantom{E-mail:\ }\printead*{e2}}

\affiliation{University of Toronto and Heinrich-Heine-Universität D\"usseldorf}

\thankstext{T1}{Supported by the collaborative research center ``Statistical modeling of nonlinear dynamic processes" (SFB 823) of the German Research Foundation (DFG).}
\thankstext{T2}{Supported by a Discovery Grant from the Natural
Sciences and Engineering Research Council of Canada.}

\begin{abstract}
Block maxima methods constitute a fundamental part of the statistical toolbox in extreme value analysis. However, most of the corresponding theory is derived under the simplifying assumption that block maxima are independent observations from a genuine extreme value distribution. In practice however, block sizes are finite and observations from different blocks are dependent. Theory respecting the latter complications is not well developed, and, in the multivariate case, has only recently been established for disjoint blocks of a single block size. We show that using overlapping blocks instead of disjoint blocks leads to a uniform improvement in the asymptotic variance of the multivariate empirical distribution function of rescaled block maxima and any smooth functionals thereof (such as the empirical copula), without any sacrifice in the asymptotic bias. We further derive functional central limit theorems for multivariate empirical distribution functions and empirical copulas that are uniform in the block size parameter, which seems to be the first result of this kind for estimators based on block maxima in general. The theory allows for various aggregation schemes over multiple block sizes, leading to substantial improvements over the single block length case and opens the door to further methodology developments. In particular, we consider bias correction procedures that can improve the convergence rates of extreme-value estimators and shed some new light on estimation of the second-order parameter when the main purpose is bias correction.
\end{abstract}

\begin{keyword}[class=AMS]
\kwd[Primary ]{62G32} 
\kwd{62E20} 
\kwd[; secondary ]{60G70}	
\kwd{62G20.} 
\end{keyword}

\begin{keyword}
\kwd{extreme-value copula}
\kwd{second order condition}
\kwd{bias correction}
\kwd{mixing coefficients}
\kwd{block maxima}
\kwd{empirical process}
\end{keyword}

\end{frontmatter}

\section{Introduction} 
\label{sec:intro}

Extreme-value theory provides a central statistical ingredient in various fields like hydrology, meteorology and financial risk management, which all have to deal with highly unlikely but important events, see, e.g., \cite{BeiGoeSegTeu04} for an overview. Mathematically, the properties of such events can be understood by studying the (multivariate) tail of probability distributions and the potential temporal dependence of tail events. Respective statistical methodology typically relies on some version of one of two fundamental approaches: the peaks-over-threshold (POT) method which considers only observations that exceed a certain high threshold, or the block maxima (BM) method which is based on taking maxima of observed values over consecutive blocks of observations and treating those maxima as (approximate) data from an extreme value distribution.

While historically the BM approach was the first to be invented \citep{Gum58}, the mathematical interest soon shifted towards the POT approach. POT methods are by now well understood, and there is a rich and mature literature on various theoretical and practical aspects of such methods, see \cite{DehFer06} for a review of many classical results and \citealp{DreRoo10, CanEinKhaLae15, FouDehMer15, EinDehZho16}  for recent developments. In the last couple of years, there has been an increased interest in the theoretical aspects of the BM approach for univariate observations, and recent work in this direction includes \cite{Dom15, FerDeh15, DomFer17, BucSeg18, BucSeg18b}. The case of multivariate observations has received much less attention, and the only theoretical analysis of (component-wise) block maxima in the multivariate setting that we are aware of is due to \cite{BucSeg14}. The present paper is motivated by this apparent imbalance of theoretical developments for BM methods as compared to POT methods in the multivariate case. 

It is well known that the analysis of multivariate distributions can be decomposed into two distinct parts: the analysis of marginal distributions and the analysis of the dependence structure as described by the associated copula. Classical results from extreme-value theory further show that the possible dependence structures of extremes have to satisfy certain constraints, but do not constitute a parametric family. In fact, the possible dependence structures  may be described in various equivalent ways (see, e.g., \citealp{Res87, BeiGoeSegTeu04, DehFer06}): by the exponent measure $\mu$ \citep{BalRes77}, by the spectral measure $\Phi$ \citep{DehRes77}, by the Pickands dependence function $A$ \citep{Pic81}, by the stable tail dependence function $L$ \citep{Hua92}, by the tail copula $\Lambda$ \citep{SchSta06}, by the madogram~$\nu$ \citep{NavGuiCooDie09}, by the extreme-value copula~$C_\infty$ (see \cite{GudSeg10} for an overview), or by other less popular objects. 

Since statistical theory for estimators of, e.g., the Pickands dependence function, the stable tail dependence function, or the madogram may be derived from corresponding results for the empirical copula process (see, e.g., \citealp{GenSeg09}), we focus on constructing estimators for the extreme-value copula $C_\infty$, which can in turn serve as a fundamental building block for subsequent developments. This approach was also taken in the above-mentioned reference \cite{BucSeg14}, who analyse the empirical copula process based on (disjoint) block maxima, and then apply the results to obtain the asymptotic behavior of estimators for the Pickands dependence function. 

The basic observational setting that we consider is the same as in \cite{BucSeg14}: data are assumed to come from a strictly stationary multivariate time series, and we assume that the copula of the random vector of component-wise block-maxima converges, as the block length tends to infinity, to a copula $C_\infty$ which is our main object of interest. However, in contrast to \cite{BucSeg14}, we base our estimators on overlapping instead of disjoint blocks. While the corresponding theoretical analysis is more involved due to the additional dependence introduced by overlaps in the blocks, we show that this always leads to a reduction in the asymptotic variance of the resulting empirical copula process and smooth functionals thereof. Another major difference with \cite{BucSeg14} is that we consider functional central limit theorems which explicitly involve the block size as a parameter. This generalization is crucial for various applications, some of which are considered in Section~\ref{sec:app}. 

As a first simple but useful application, we consider estimators for $C_\infty$ which are based on aggregating over various block length parameters, thereby providing estimators which are less sensitive to the choice of a single block length parameter. The corresponding asymptotic theory is a straightforward consequence of the asymptotic theory mentioned before. A Monte Carlo simulation study reveals the superiority of the aggregated estimators over their non-aggregated versions in typical finite-sample situations.

A second more involved application concerns the construction of bias-reduced estimators for $C_\infty$ (see \citealp{FouDehMer15, BeiEscGoeGui16} for recent proposals in the multivariate POT approach for i.i.d.\ observations). As is typically done when tackling the problem of bias reduction in extreme value statistics, the estimators are obtained by explicitly taking into account the second order structure of the extreme value model in the estimation step. We are not aware of any results on bias-reduced estimators within the block maxima framework in general. In fact, even for POT methods such results do not seem to exist in the multivariate time series setting (some results on the univariate time series case can be found in \citealp{DehMerZho16}). As a necessary intermediate step for bias correction, we need to consider estimation of a second order parameter which naturally shows up in the second order condition. We show that special care needs to be taken when estimating this parameter for its use in bias correction, and propose a penalized estimator which explicitly takes this specific aim into account. 

The improvement in both variance and bias of one of the estimators for $C_\infty$ proposed in this paper over the disjoint blocks estimator from \cite{BucSeg14} is illustrated in Figure~\ref{fig:fincomp}.

\begin{figure}[t!]
\begin{center}
%\vspace{-.1cm}
\includegraphics[width=0.99\textwidth]{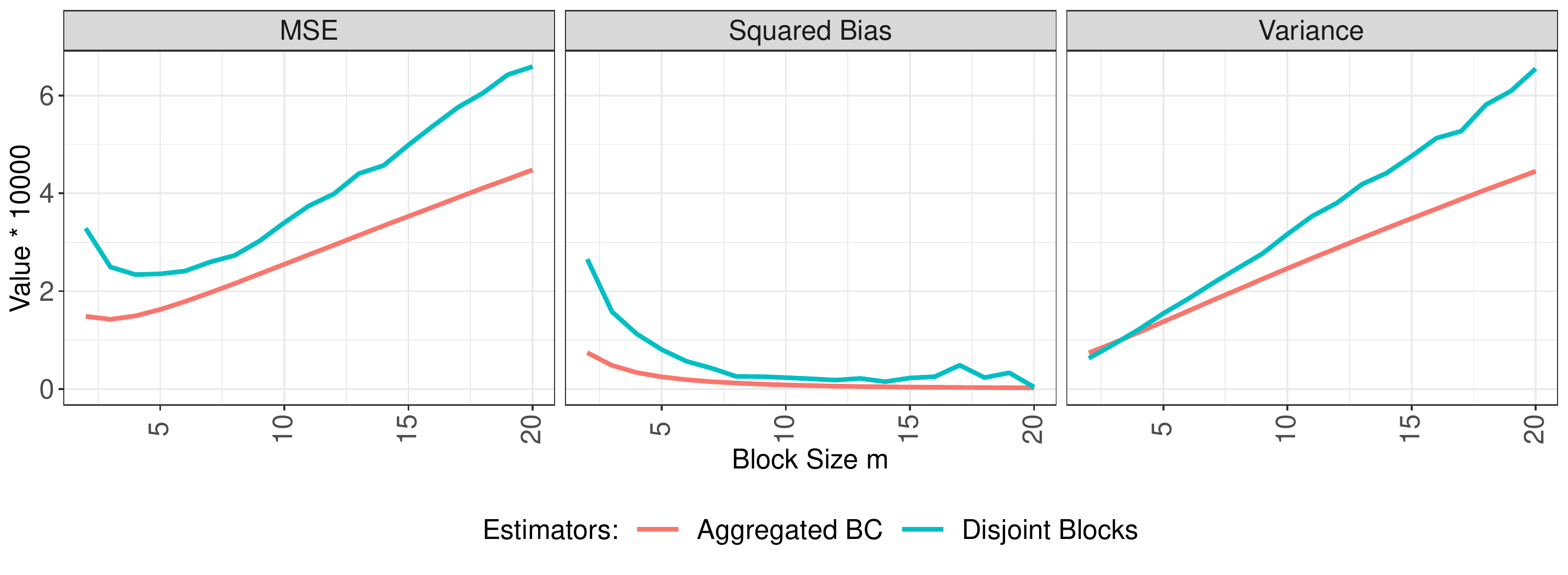}
\vspace{-.6cm}
\end{center}
\caption{\label{fig:fincomp} 
$10^{4} \times$average MSE, squared bias and variance of the disjoint blocks estimator from \cite{BucSeg14} and the aggregated bias corrected estimator proposed in this paper. Data generating process and estimators are as described in Section~\ref{sec:fin}, Model (M2).}
\vspace{-.3cm}
\end{figure}

The idea of using sliding/overlapping block maxima for statistical inference appears to be quite new to the extreme value community, whence similar results in the literature actually are rare, even in univariate situations. To the best of our knowledge, the idea first appeared in the context of estimating the extremal index of a univariate stationary time series, see \cite{RobSegFer09,Nor15,BerBuc18}. The only paper we are aware of in the classical univariate case is \cite{BucSeg18b}, which is restricted to the heavy tailed case. The idea of basing inference on multiple block sizes seems to be new, and is possibly transferable to the univariate case as well.

We further remark that there is a rich and mature literature that deals with estimation of extreme-value copulas and related objects when observations from an extreme-value copula are available (see, among many others, \citealp{Pic81}, \citealp{CapFouGen1997} for early contributions and \citealp{GenSeg09},\citealp{GudSeg10} for rank-based methods). However, the setting in that literature is different from ours since we do not assume that data from the extreme value copula are available directly.

The remaining parts of this paper are organized as follows: the sliding block maxima (empirical) copula process, including the block length as an argument of the process, is considered in Section~\ref{sec:master}. The applications on aggregated estimators, bias-reduced estimators and estimators of second order parameters are worked out in Section~\ref{sec:app}. Some theoretical examples, as well as a detailed Monte Carlo simulation study are presented in Section~\ref{sec:fin}. {All proofs are deferred to a supplementary material (\citealp{ZouVolBuc19supp}).} 

Throughout, for $\xi\in\R$, let $\ceil{\xi}$ be the smallest integer greater or equal to $\xi$. Let $\ang{\xi}$ be the largest integer smaller or equal to $\xi$ if $\xi
\geq0$ and the smallest integer greater or equal to $\xi$ if $\xi<0$. For $\bm u, \bm v \in \R^d$, write $\bu\leq \bv$ if $u_{j}\leq v_{j}$ for all $j$, and $\bu\nleq \bv$ if there exists $j$ such that $u_{j}>v_{j}$. Let $\bu\wedge\bv=(\min(u_{1},v_{1}),\dots,\min(u_{d},v_{d})).$ All convergences will be for $n\to\infty$, if not mentioned otherwise. The arrow $\dto$ denotes weak convergence in the sense of Hoffman-J{\o}rgensen, see \cite{VanWel96}.

%%%%%%%%%%%%%%%%%%%%%%%%%%%%%%%%%%%%%%%%%%%%%%%%%%%%%
\section{Functional weak convergence of empirical copula processes based on sliding block maxima} 
\label{sec:master}
%%%%%%%%%%%%%%%%%%%%%%%%%%%%%%%%%%%%%%%%%%%%%%%%%%%%%

Suppose $(\bX_{t})_{t\in \bbZ}=(X_{t,1},\dots,X_{t,d})_{t\in \bbZ}$ is a multivariate strictly stationary process, and that $(\bX_{t})_{t=1}^{n}$ is observable data. Let $m\in\{1, \dots, n\}$ be a block size parameter and, for $i=1, \dots,n-m+1$ and $j=1, \dots, d$, let $M_{m,i,j} = \max \{X_{t,j}:t\in [i,i+m)\cap \bbZ\}$ be the maximum of the $i$th sliding block of observations in the $j$th coordinate. For $\bm x=(x_1, \dots, x_d) \in \R^d$, let 
\begin{align*}
    &\bM_{m,i}=(M_{m,i,1},\dots,M_{m,i,d})
		\\
    &F_{m,j}(x)=\Prob(M_{m,1,j}\leq x)
		\\
    &F_{m}(\bx)=\Prob(\bM_{m,1}\leq \bx)
		\\
    &\bF_{m}(\bx)=(F_{m,1}(x_{1}),\dots,F_{m,d}(x_{d}))
		\\
    &\bF_{m}^{\leftarrow}(\bx)=(F_{m,1}^{\leftarrow}(x_{1}),\dots,F_{m,d}^{\leftarrow}(x_{d}))
		\\
    &U_{m,i,j}=F_{m,j}(M_{m,i,j})
		\\
    &\bU_{m,i}=(U_{m,i,1},\dots,U_{m,i,d}),
\end{align*}

where $G^{\leftarrow}$ denotes the left-continuous generalized inverse of a c.d.f.\ $G$. Subsequently, we assume that the marginal c.d.f.s of $X_{1,1}, \dots, X_{1,d}$ are continuous. 
In that case, the marginal c.d.f.s of $\bm M_{m,1}$ are continuous as well and
\[
C_{m}(\bu)=\Prob(\bU_{m,1}\leq \bu), \quad \bu \in [0,1]^{d},
\]
is the unique copula associated with $\bm M_{m,1}$. Throughout, we shall work under the following fundamental domain-of-attraction condition.
\begin{Assump}\label{assump:copula_convergence}
There exists a copula $C_{\infty}$ such that %, for any $\bu\in[0,1]^d$,
\[
\lim_{m\to\infty}C_{m}(\bu)= C_{\infty}(\bu), \qquad \bm u \in [0,1]^d.
\]
\end{Assump}

Typically, the limit $C_\infty$ will be an extreme value copula \citep{Hsi89, Hus90}, that is, {$C_\infty(\bm u^{1/s})^s=C_\infty(\bm u)$ for all $s>0$ and $\bm u \in [0,1]^d$ and}
\[
C_\infty(\bm u) = \exp\{- L(-\log u_1, \dots, -\log u_d) \}, \quad \bm u \in [0,1]^d,
\]
for some \textit{stable tail dependence function} $L:[0,\infty]^d \to [0,\infty]$ satisfying
\smallskip
\begin{compactenum}
\item[(i)] $L$ is homogeneous: $L(s\, \cdot) = sL(\cdot)$ for all  $s>0$;
\item[(ii)] $L(\bm e_j)=1$ for $j=1, \dots, d$, where $\bm e_j$ denotes the $j$th unit vector;
\item[(iii)] $\max(x_1, \dots, x_d) \le L(\bm x) \le x_1+ \dots + x_d$ for all $\bm x\in [0,\infty)^d$;
\item[(iv)] $L$ is convex;
\end{compactenum} \smallskip
see, e.g., \cite{BeiGoeSegTeu04}. By Theorem 4.2 in \cite{Hsi89}, this is for instance the case if the time series $(\bm X_t)_t$ is beta-mixing. However, $C_\infty$ is in general different from the extreme value attractor, say $C_\infty^{\rm{iid}}$, in case the observations are i.i.d.\ from the stationary distribution of the time series, see for instance Section 4.1 in \cite{BucSeg14}. In fact, (block) maxima calculated from time series naturally incorporate information about the serial dependence (as, e.g., measured by the multivariate extremal index, see Section 10.5.2. in \citealp{BeiGoeSegTeu04}), whence the BM approach is typically more suitable when it comes to, e.g., assessing return levels or periods. 
In the i.i.d.\ case, Assumption~\ref{assump:copula_convergence} is equivalent to the existence of a stable tail dependence function $L$ such that
\begin{align*} %\label{eq:dompot}
\lim_{t \to \infty} t \{ 1- C_1(\bm 1- \bm x / t)\} = L(\bm x), \qquad \bm x \in [0,\infty)^d,
\end{align*}
where the copula $C_1$ is naturally extended to a c.d.f.\ on $\R^d$.

Assumption~\ref{assump:copula_convergence} does not contain any information about the rate of convergence of $C_m$ to $C_\infty$. In many cases, more precise statements about this rate can be made, and it is even possible to write down higher order expansions for the difference $C_m - C_\infty$. For some of the material in the paper, we will assume the validity of such expansions. Recall that a function $\ca$ defined on the integers is regularly varying if $t\mapsto \ca(\ang{t})$ is regularly varying as a function $(0,\infty) \to \R$.

\begin{Assump}[Second order condition]\label{assump:high_order}~
There exists a regularly varying function
 $\ca: \N \to (0,\infty)$ with coefficient of regular variation $\rho_\ca < 0$ and a (necessarily continuous) non-null function $S$ on $[0,1]^d$ such that %, as $m \to \infty$,
\begin{align*} %\label{eqn:secor}
C_{m}(\bu) - C_\infty(\bu) = \ca(m)S(\bu) + o(\ca(m)) \qquad (m \to \infty),
\end{align*}
uniformly in $\bu \in [0,1]^d$.
\end{Assump}

We refer to the accompanying paper \cite{BucVolZou19} for a detailed account on second order conditions in the i.i.d.\ case. In particular, the latter paper shows that the block maxima second order condition above follows from the more common second order condition imposed on a POT-type convergence to $L$ under fairly general assumptions, see also Equation (6) in \cite{FouDehMer15}. 
It was further shown in \cite{BucVolZou19}  that, in the i.i.d.\ case,  the function $\ca$ in the condition above must be regularly varying (the part can hence be removed from the assumption), that the function $S$ has certain homogeneity properties and that local uniform convergence on $[\delta,1]^d$ is sufficient for uniform convergence on $[0,1]^d$. Specific examples in the i.i.d. and time series case are discussed in more detail in Section~\ref{subsec:examples}.

\subsection{Estimation in the case of known marginal distributions.}\label{sec:fix}

We begin by estimating $C_\infty$ in the case of known marginal c.d.f.s $F_{1,1}, \dots, F_{1,d}$, which, on the level of proofs, is a necessary intermediate step when considering the realistic case of unknown marginal c.d.f.s in the subsequent section. For block size $m' \in \{1, \dots, n\}$, let
\begin{align} \label{eq:cnknown}
\hC_{n,m'}^{\circ}(\bu) &=\frac{1}{n-m'+1}\sum_{i=1}^{n-m'+1} \ind(\bU_{m',i}\leq \bu), \quad \bu \in [0,1]^d,
\end{align}
denote the empirical c.d.f.\ of the  sample of standardized sliding block maxima $\bm U_{m',1}, \dots, \bm U_{m',n-m'+1}$.
Subsequently, we will consider block sizes of the form $m'=\ang{ma}$ with scaling parameter $a>0$. The respective centred empirical process we are interested in is
\begin{align*}
\bbC_{n,m}^{\lozenge}(\bu,a)
&=
{\sqrt{n/m} \{ \hC_{n,\ang{ma}}^{\circ}(\bu) - C_{\ang{ma}}(\bu) \} } \\
&=
{\sqrt{n/m}\frac{1}{b_{a}}\sum_{i=1}^{b_{a}}\Big\{  \ind(\bU_{\ang{ma},i}\leq \bu)-\Prob(\bU_{\ang{ma},i}\leq \bu)\Big \},}
\end{align*}
where $b_a= n-\ang{ma}+1$. For the functional weak convergence results to follow, we consider $\bbC_{n,m}^{\lozenge}$ as an element of $(\ell^\infty([0,1]^d\times A), \| \cdot \|_\infty)$, the space of bounded function on $[0,1]^d\times A$ equipped with the supremum norm, where $A = [a_\wedge,a_\vee] \subset (0,\infty)$ is a fixed interval, {the case $a_\wedge=a_\vee$ being explicitly allowed}.  We impose the following assumptions on the block length parameter $m=m_n$ and the serial dependence of the time series.

\begin{Assump}\label{assump:mixing}
Denote by $\alpha(\cdot)$ and $\beta(\cdot)$ the $\alpha$ and $\beta$ mixing coefficients of the process $\{\bX_{t}\}_{t\in \bbZ}$, respectively. Assume \smallskip
\begin{compactenum}[\rm (i)] 
\item $m=m_{n}\to \infty, \ n/m\to \infty,$
\item $ \alpha(h)=o(h^{-(1+\varrho)})$ as $h\to\infty$, for some $\varrho>0$,
\item $ \beta(m)(n/m)^{1/2}\to 0,$
\item $ \alpha(m)(n/m)^{1/2+\zeta}\to 0,
\ \text{for some} \ \zeta \in (0,1/2).$
\end{compactenum}
\end{Assump}

Condition (i) is a typical condition in extreme value statistics, and in fact a necessary condition to allow for consistent estimation of $C_\infty$.
Condition (ii) is a short-range dependence condition that we introduce merely for technical reasons associated with our method of proof. At the cost of more sophisticated proofs, the condition may possibly be relaxed. However, since the condition is known to be satisfied for many common time series model, we feel that such a relaxation is not necessarily needed. Assumptions (iii) and (iv) relate the block length parameter to the serial dependence and allow for obtaining central limit theorems (alpha-mixing) and proofs of tightness based on coupling arguments (beta-mixing).

\begin{Theo}\label{theo:known}
 Under Assumptions \ref{assump:copula_convergence} and \ref{assump:mixing},
\begin{equation*}
    \bbC_{n,m}^{\lozenge}\dto\bbC^{\lozenge} \quad \text{ in } \quad  \ell^\infty([0,1]^d\times A),
\end{equation*}
where $\bbC^{\lozenge}$ denotes a tight centred Gaussian process on $[0,1]^{d}\times A$ with continuous sample paths and covariance function 
\begin{align*}
&\Cov\big(\bbC^{\lozenge}(\bu,a),\bbC^{\lozenge}(\bv,c)\big) \\
&= 
\int_{-a}^{0}(C_{\infty}(\bu^{1/a}))^{-\xi}(C_{\infty}(\bv^{1/c}\wedge\bu^{1/a}))^{\xi+a}(C_{\infty}(\bv^{1/c}))^{c-\xi-a} \diff \xi 
\\
&\mathph{\to a^{-1/2}\bigg[}+\int_{0}^{c-a}(C_{\infty}(\bv^{1/c}))^{c-a}(C_{\infty}(\bv^{1/c}\wedge\bu^{1/a}))^{a} \diff \xi
\\
&\mathph{\to a^{-1/2}\bigg[}+\int_{c-a}^{c}(C_{\infty}(\bv^{1/c}))^{\xi}(C_{\infty}(\bv^{1/c}\wedge\bu^{1/a}))^{c-\xi}(C_{\infty}(\bu^{1/a}))^{\xi+a-c} \diff \xi
\\
&\mathph{\to a^{-1/2}\bigg[}-(c+a)C_{\infty}(\bv)C_{\infty}(\bu)  
\\
&=:\gamma(\bv,\bu,c,a), \qquad \qquad (a_\wedge\leq a\leq c\leq a_\vee, \bm u, \bm v \in [0,1]^d).
\end{align*}
\end{Theo}

Perhaps surprisingly, the limiting covariance does not depend on the serial dependence of the original time series, except through $C_\infty$ itself. In the univariate case this was also observed in \cite{BucSeg18b}.

\begin{Remark}  \label{rem:known}
Under a slightly weaker version of Assumption~\ref{assump:mixing}, \cite{BucSeg14}, Theorem 3.1, investigated the corresponding empirical process based on disjoint block maxima with $a=c=1$, that is, the process in $\ell^\infty([0,1]^d)$ defined by
\[
\bm u \mapsto \sqrt{n/m} \Big\{ \frac{1}{\ang{m/n}} \sum_{i=1}^{\ang{m/n}} \ind(\bm U_{m,1+m(i-1)} \le \bm u) - C_m(\bm u) \Big\},
\]
and with tight centred Gaussian limit denoted by $\bbC^{D}(\bu)$.
The covariance function of the limiting process is given by
\[
\gamma^{D}(\bu,\bv)= \Cov(\bbC^{D}(\bu), \bbC^{D}(\bv)) = C_{\infty}(\bu\wedge\bv)-C_{\infty}(\bu)C_{\infty}(\bv).
\]
A comparison between the covariance functionals $\gamma$ and $\gamma^D$ is worked out in Section~\ref{sec:comp} below, c.f. Section~\ref{sec:asyvar} in the supplementary material~\cite{ZouVolBuc19supp} for an alternative expression for $\gamma$.
\end{Remark}

\begin{proof}[Proof of Theorem~\ref{theo:known}]
Recall $b_{a}=n-\ang{ma}+1$, let $b=b_1=n-m+1$ and define 
\begin{equation*}
\bbC_{n,m}^{\lozenge,b}(\bu,a)
=
\sqrt{n/m}\frac{1}{b}\sum_{i=1}^{b} \Big(\ind(\bU_{\ang{ma},i}\leq \bu)-\Prob(\bU_{\ang{ma},i}\leq \bu)\Big).
\end{equation*}
The proof consists of several steps, which are explicitly taken care of in the supplementary material \cite{ZouVolBuc19supp}:
\begin{compactenum}
\item[(i)] In Lemma~\ref{lemma:number_of_block} we prove that $\|\bbC_{n,m}^{\lozenge}-\bbC_{n,m}^{\lozenge,b}\|_\infty \pto 0$. Hence it suffices to prove weak convergence of $\bbC_{n,m}^{\lozenge,b}$.
\item[(ii)] In Lemma~\ref{lem:tight} we show that $\bbC_{n,m}^{\lozenge,b}$ is asymptotically uniformly equicontinuous in probability with respect to the $\|\cdot\|_\infty$-norm on $[0,1]^d\times A$.
\item[(iii)] In Lemma~\ref{lemma:fidi_convergence} we  prove that the finite-dimensional distributions of $\bbC_{n,m}^{\lozenge,b}$ converge weakly to those of $\bbC^{\lozenge}$.
\end{compactenum}
Weak convergence of $\bbC_{n,m}^{\lozenge}$ then follows by combining (i)-(iii). 
\end{proof}

The proofs of Step~(ii) and Step~(iii) are quite lengthy and technical, but it is instructive to present the main ideas within the next two remarks.

\begin{Remark}[Proving fidi-convergence]\label{rem:fidi} 
The main steps for proving weak convergence of the finite-dimensional distributions (see Lemma~\ref{lemma:fidi_convergence} for details) are as follows:
\begin{asparaenum}[(i)]
\item \textbf{Calculation of the limiting covariance functional $\gamma$.}  This is treated in Lemma~\ref{lemma:cov}, and bears similarities with common long run variance calculations in classical time series analysis. The integrals in $\gamma$ are due to the fact that some of the sliding blocks are overlapping, with the integration variable $\xi$ controlling the relative position of two overlapping blocks, and with each of the three integrals corresponding to one of three possibilities for two blocks to overlap: (1) a block of length $a$ starts before a block of length $c$ and ends inside, (2) a block of length $a$ lies completely within a block of length $c$, or (3) a block of length $a$ starts inside a block of length $c$ and ends outside. Consider for instance the latter case, which would correspond to $0 < c-a < \xi < c$ and amounts to consideration of the event
$
\{ \bm M_{1:\ang{mc}} \le \bm x, \bm M_{\ang{m\xi}+1: \ang{m\xi}+\ang{ma}} \le \bm y\}.
$
The main idea consist of rewriting this event as
\[
\{\bm M_{1:\ang{m\xi}} \le \bm x\} \cap \{ \bm M_{\ang{m\xi}+1: \ang{mc}} \le \bm x \wedge \bm y \} \cap \{\bm M_{\ang{mc}+1: \ang{m\xi}+\ang{ma}} \le \bm y\}.
\]
We then use alpha mixing to show that the three events are asymptotically independent; this eventually gives rise to the three-fold product in the third integral in the definition of $\gamma$ with each of the factors corresponding to the probability of one of the events above.
\item \textbf{Big-Blocks-Small-Blocks technique.} The summands of the estimator of interest are collected in successive blocks of (block maxima) observations, with a `big block' followed by a `small block'  followed by a `big block' etc.  The small blocks are then shown to be negligible, while the big blocks are shown to be asymptotically independent (via alpha mixing). Weak convergence of the sum corresponding to big blocks can finally be shown by an application of the Lyapunov Central Limit Theorem.
\end{asparaenum}
\end{Remark}

\begin{Remark}[Proving asymptotic tightness] The main steps for proving the tightness part (see Lem\-ma~\ref{lem:tight}) are  as follows:
\begin{asparaenum}[(i)]
\item \textbf{Getting rid of serial dependence.} 
Based on a coupling lemma for beta mixing sequences by \cite{Ber79} and a blocking argument, proving tightness of $\bbC_{n,m}^{\lozenge,b}$ may be reduced to proving tightness of two empirical processes based on row-wise i.i.d.\ observations. In contrast to classical time series settings where blocks are based on the original observations, we consider blocks of collections of block maxima corresponding to all block sizes considered. Blocking vectors of block maxima is needed to deal with the additional block length parameter in our setting.
\item \textbf{Proving tightness via a moment bound.}
After the reduction in step (i), we now deal with row-wise i.i.d. observations and the results in \cite{VanWel96} can be applied. Here, each `observation' corresponds to a block of collections of block maxima mentioned in the previous step. The moment bound in Theorem 2.14.2 in the latter book allows to deduce tightness of the corresponding processes from controlling the bracketing numbers of certain function classes which map collections of block maxima to pieces in the sum defining $\bbC_{n,m}^{\lozenge}(\bu,a)$. 
\item 
\textbf{Bounding a certain bracketing number.} The last step is based on some explicit lengthy calculations, which take the precise definition of the triangular arrays into account, and in particular the fact that the `observations' are block maxima (with arguments similar to the one given in Remark~\ref{rem:fidi} for the calculation of the limiting covariance).
\end{asparaenum}
\end{Remark}

\subsection{Estimation in the case of unknown marginal c.d.f.s} \label{sec:estmar}

The results in Section~\ref{sec:fix} are based on the assumption that the marginal c.d.f.s are known. In practice, this is not realistic and marginals are typically standardized by taking component-wise ranks of observed block maxima.  For $x\in\R, j=1,\dots, d$ and block size $m'$, let 
\begin{align*}
&\hF_{n,m',j}(x)=\frac{1}{n-m'+1}\sum_{i=1}^{n-m'+1}\ind(M_{m',i,j}\leq x)
\end{align*}
and consider observable pseudo-observations from $C_{m'}$ defined as
\begin{align*}
\hbU_{n,m',i}=(\hU_{n,m',i,1},\dots,\hU_{n,m',i,d}), \qquad \hU_{n,m',i,j}=\hF_{n,m',j}(M_{m',i,j})
\end{align*}
The observable analog of the estimator $\hC_{n,m'}^{\circ}$ in \eqref{eq:cnknown} is then given by 
\[
\hC_{n,m'}(\bu)=\frac{1}{n-m'+1}\sum_{i=1}^{n-m'+1}\ind(\hbU_{n,m',i}\leq \bu),
\]
and we are interested in the asymptotic behavior of the associated empirical copula process, indexed by $\bm u \in [0,1]^d$ and  block length scaling parameter $a \in A$, defined as
\begin{equation*}
\widehat \bbC_{n,m}^{\lozenge}(\bu,a)
= \sqrt{n/m}\{ \hC_{n,\ang{ma}}(\bu)-C_{\ang{ma}}(\bu) \}.
\end{equation*}
Subsequently, the process will be called \textit{extended empirical copula process based on sliding block maxima}. Additional assumptions are needed for a corresponding weak convergence result.

\begin{Assump}\label{assump:pd}
For any $j=1,\dots,d$, the $j$th first order partial derivative $\dC_{\infty,j}(\bu)=\partial C_{\infty}(\bu)/\partial u_{j}$ of $C_\infty$ exists and is continuous on $\{\bu\in [0,1]^{d}: u_{j}\in (0,1)\}$. 
\end{Assump}

Recall that such an assumption is even needed for weak convergence of the classical empirical copula process based on i.i.d.\ observations from $C_\infty$ \citep{Seg12}. For completeness, define $\dC_{\infty,j}(\bu)=0$ if $u_j \in \{0,1\}$.
Following~\cite{BucSeg14}, we do not need differentiability of $C_m$ for finite $m$. Instead, we will work  with the functions
\[
\dot{C}_{m,j}(\bv) := \limsup_{h \downarrow 0} h^{-1} \{C_m(\bv + h\be_j) - C_m(\bv)\},
\]
where $j=1,\dots ,d, m \in \N, \bv \in [0,1]^d$ and $\be_{j}$ denotes the $j$th canonical unit vector in $\bbR^{d}$. Note that $\dot{C}_{m,j}$ is always defined and satisfies $0 \leq \dot{C}_{m,j} \leq 1$. 

For the upcoming main theorem of this paper, we will need an additional assumption on the quality of convergence of $C_m$ to $C_\infty$, which will eventually allow us to move from the known margins to the unknown margins case. Any of the following three conditions will be sufficient; the first two assumptions have also been considered in \cite{BucSeg14} (with $k_n=m_n$ in Part (a)), while the third part (a more refined version of (a)) is included specifically for the bias corrections worked out in Section~\ref{sec:bcknown}, where (a) is typically not met.

\begin{Assump}[Quality of convergence of $C_m$ to $C_\infty$]\label{assump:BuSe3.4_new}~ \smallskip
\begin{compactenum}
\item[(a)] 
A sequence $(k_n)_{n\in \N}$ of natural numbers with $k_n \to \infty$ is said to satisfy $\SC_1(k_n)$ if $\sqrt{n/k_n}(C_{k_n} - C_\infty)$ is relatively compact in $\cC([0,1]^d)$ (the space of continuous, real-valued functions on $[0,1]^d$). 
\item[(b)] 
For every $\delta \in (0,1/2)$, letting $S_{j,\delta} := [0,1]^{j-1}\times [\delta,1-\delta] \times [0,1]^{d-j}$,
\[
\lim_{m\to\infty} \max_{j=1,\dots,d}\sup_{\bu \in S_{j,\delta}} |\dot{C}_{m,j}(\bu) - \dot{C}_{\infty,j}(\bu)|=0.
\]

\item[(c)] A sequence $(k_n)_{n\in  \N}$ of natural numbers with $k_n \to \infty$ is said to satisfy $\SC_2(k_n)$ if Assumption~\ref{assump:high_order} holds, $S$ is uniformly H{\"o}lder-continuous of order $\delta \in (0,1]$, $(n/k_n)^{(1-\delta)/2}\ca(k_n) = o(1)$ as $n\to\infty$ and $n \mapsto \sqrt{n/k_n}\{ C_{k_n}-C_{\infty}-\ca(k_n)S]$ is relatively compact in $\cC([0,1]^d)$.
\end{compactenum}
\end{Assump}  

We are now ready to state the main result of this section.

\begin{Theo}[Functional weak convergence of the extended empirical copula process based on sliding block maxima]\label{theo:estmar}
Let Assumptions \ref{assump:copula_convergence}, \ref{assump:mixing} and \ref{assump:pd} hold.
If either $\SC_1(\ang{m_n a_n})$ from Assumption~\ref{assump:BuSe3.4_new}(a) holds for every converging sequence $(a_n)_{n\in\N}$ in $A$, or if Assumption~\ref{assump:BuSe3.4_new}(b) holds, or if $\SC_2(\ang{m_n a_n})$ from Assumption~\ref{assump:BuSe3.4_new}(c) holds for every converging sequence $(a_n)_{n\in\N}$ in $A$, then
\[
\widehat \bbC^{\lozenge}_{n,m} \dto \widehat \bbC^{\lozenge} \quad \mbox{ in } \ell^\infty([0,1]^d\times A),
\]
where, letting $\bu^{(j)}=(1,\dots,1,u_{j},1,\dots,1)$ with $u_{j}$ at the $j$th coordinate,
\begin{align} \label{eq:estmarlim}
\widehat\bbC^{\lozenge}(\bu,a)=  \bbC^{\lozenge}(\bu,a)-\sum_{j=1}^{d}\dC_{\infty,j}(\bu)\bbC^{\lozenge}(\bu^{(j)},a).
\end{align}
\end{Theo}

If additionally Assumption~\ref{assump:high_order} is met, then Theorem~\ref{theo:estmar} shows that the uniform convergence rate of $\hat C_{n,m}$ to $C_\infty$ is given by $O_\Prob(\sqrt{m/n} + \ca(m))$, where $\sqrt{m/n}$ corresponds to the stochastic part, while $\ca(m)$ is due to the deterministic difference between $C_m$ and $C_\infty$. Assuming for simplicity that Assumption~\ref{assump:high_order} holds with $\ca(m) = m^{\rho_\ca}$ we find that the best possible convergence rate of $\hat C_{n,m}$ is obtained by setting $m \asymp n^{1/(1-2\rho_\ca)}$. In Section~\ref{sec:app} we will show that this rate can in fact be improved by combining estimators $\hat C_{n,\ang{ma}}$ for several values of $a$. Establishing the asymptotic properties of those estimators will require the full power of Theorem~\ref{theo:estmar}, including the process convergence uniformly over the block length parameter $a$.

\begin{Remark} 
If Assumption~\ref{assump:high_order} is met and if $\sqrt{n/m} \ca(m) =O(1)$, then it is easy to show (using regular variation of $\ca$) that  $\SC_1(\ang{m_n a_n})$ from Assumption~\ref{assump:BuSe3.4_new}(a) holds for every converging sequence $(a_n)_{n\in\N}$ in $A$. Similarly, under Assumption~\ref{assump:high_order3}  below and if $\sqrt{n/m} \ca(m)\cb(m) =O(1)$, then
$\SC_2(\ang{m_n a_n})$ holds for every converging sequence $(a_n)_{n\in\N}$ in $A$. \end{Remark}

\subsection{A comparison of the asymptotic variances 
based on disjoint and sliding block maxima} \label{sec:comp}

The asymptotic variance of the sliding blocks version of the empirical copula with known and estimated margins will be shown to be less than or equal to the asymptotic variance of the corresponding disjoint blocks versions. Since the asymptotic bias of both approaches is the same, this suggests that the sliding blocks estimator, when available, should always be used instead of the disjoint blocks estimator. 

\begin{Theo}\label{theo:varcomp}
Suppose $C_\infty$ is an extreme value copula satisfying Assumption~\ref{assump:pd}. Let $\widehat\bbC^{\lozenge}(\bu,1)$ denote the weak limit of the empirical copula process based on sliding block maxima defined in \eqref{eq:estmarlim}. Similarly, recall $\bbC^{D}(\bu)$ as defined in Remark~\ref{rem:known} and let
\[
 \textstyle \widehat\bbC^{D}(\bu)= \bbC^{D}(\bu)-\sum_{j=1}^{d}\dC_{\infty,j}(\bu)\bbC^{D}(\bu^{(j)})
\]
denote the weak limit of the corresponding disjoint blocks version (Theorem 3.1 in \citealp{BucSeg14}). Then, for any $\bm u_1, \dots, \bm u_k \in [0,1]^d,k\in\N$,
\[
\Cov\Big( \widehat\bbC^{\lozenge}(\bu_1,1), \dots, \widehat\bbC^{\lozenge}(\bu_k,1) \Big) \le_L \Cov\Big(\widehat\bbC^{D}(\bu_1), \dots, \widehat\bbC^{D}(\bu_k)\Big)
\]
and
\begin{equation}\label{eq:comknownmarg}
\Cov\Big(\bbC^{\lozenge}(\bu_1,1), \dots, \bbC^{\lozenge}(\bu_k,1) \Big) \le_L \Cov\Big(\bbC^{D}(\bu_1), \dots, \bbC^{D}(\bu_k)\Big),
\end{equation}
where $\le_L$ denotes the Loewner-ordering between symmetric matrices.
\end{Theo}

The proof is given in Section~\ref{sec:proofcomp}. In Figure~\ref{fig:compvarth} we depict $\Var(\widehat\bbC^{\lozenge}(\bu,1))$ and $\Var(\widehat\bbC^{D}(\bu))$, for $\bu = (u,u)$ with $u \in [0,1]$, 
for the Gumbel--Hougaard Copula in \eqref{eqn:Gumbel-Hougaard} with shape parameter $\beta=1$ and $\beta=\ln 2/\ln(3/2)$. Note that when $\beta=1$, the Gumbel--Hougaard Copula degenerates to the independence copula on $[0,1]^2$, i.e., $C_\infty(u_1,u_2)=u_1u_2$ while $\beta = \ln 2/\ln(3/2)$ results in a tail dependence coefficient of $1/2$. The analytical expressions of $\Var(\widehat\bbC^{\lozenge}(\bu,1))$ and $\Var(\widehat\bbC^{D}(\bu))$ for the Gumbel--Hougaard Copula are presented in Section~\ref{sec:asyvar} in the supplementary material \cite{ZouVolBuc19supp}. The difference between $\Var(\widehat\bbC^{\lozenge}(\bu,1))$ and $\Var(\widehat\bbC^{D}(\bu))$ is seen to be substantial, in particular for small values of $u$.

\begin{figure}[h]
\begin{center}
\vspace{-.3cm}
\mbox{\includegraphics[width=4.3cm]{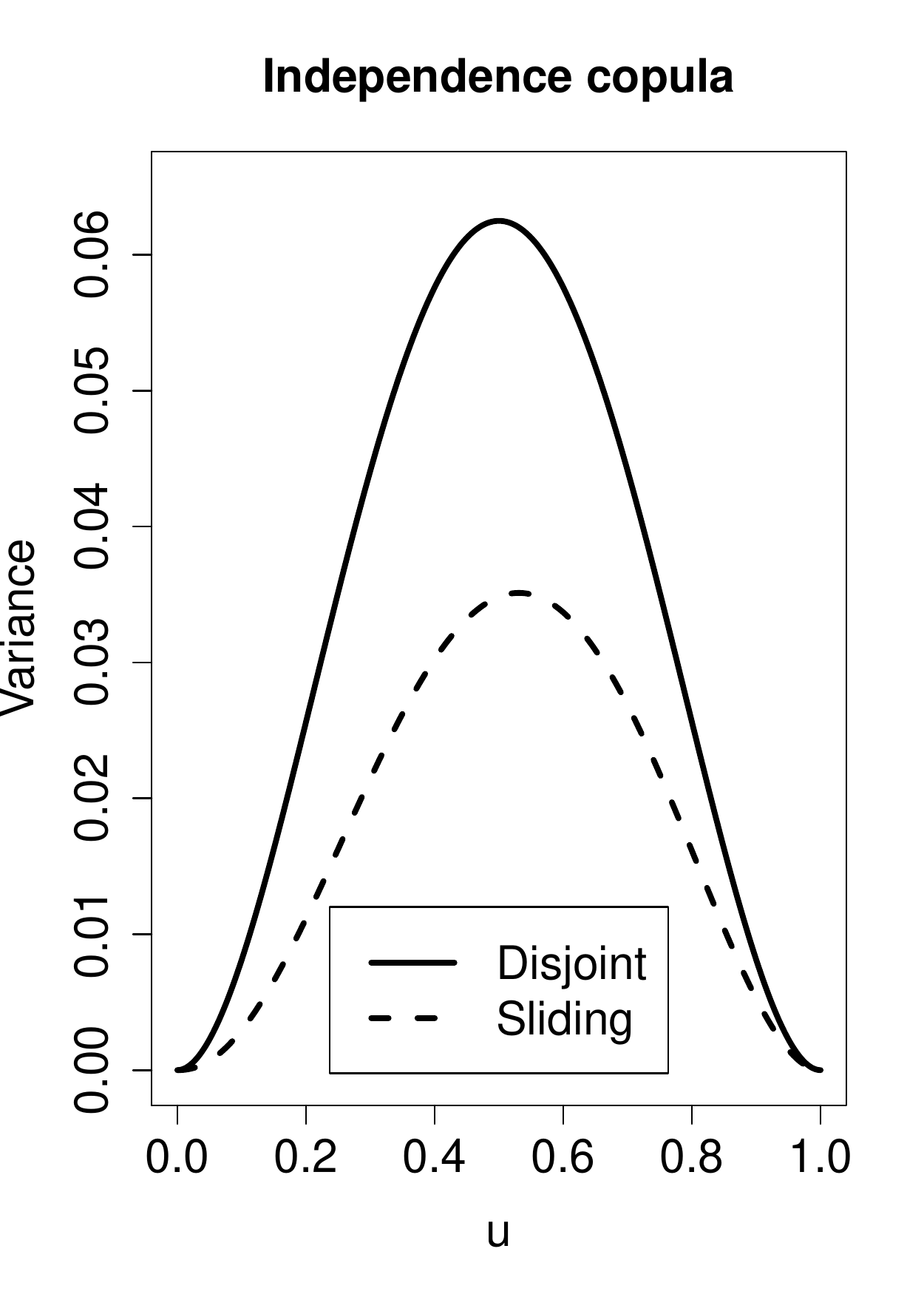} \hspace{-.2cm}
\includegraphics[width=4.3cm]{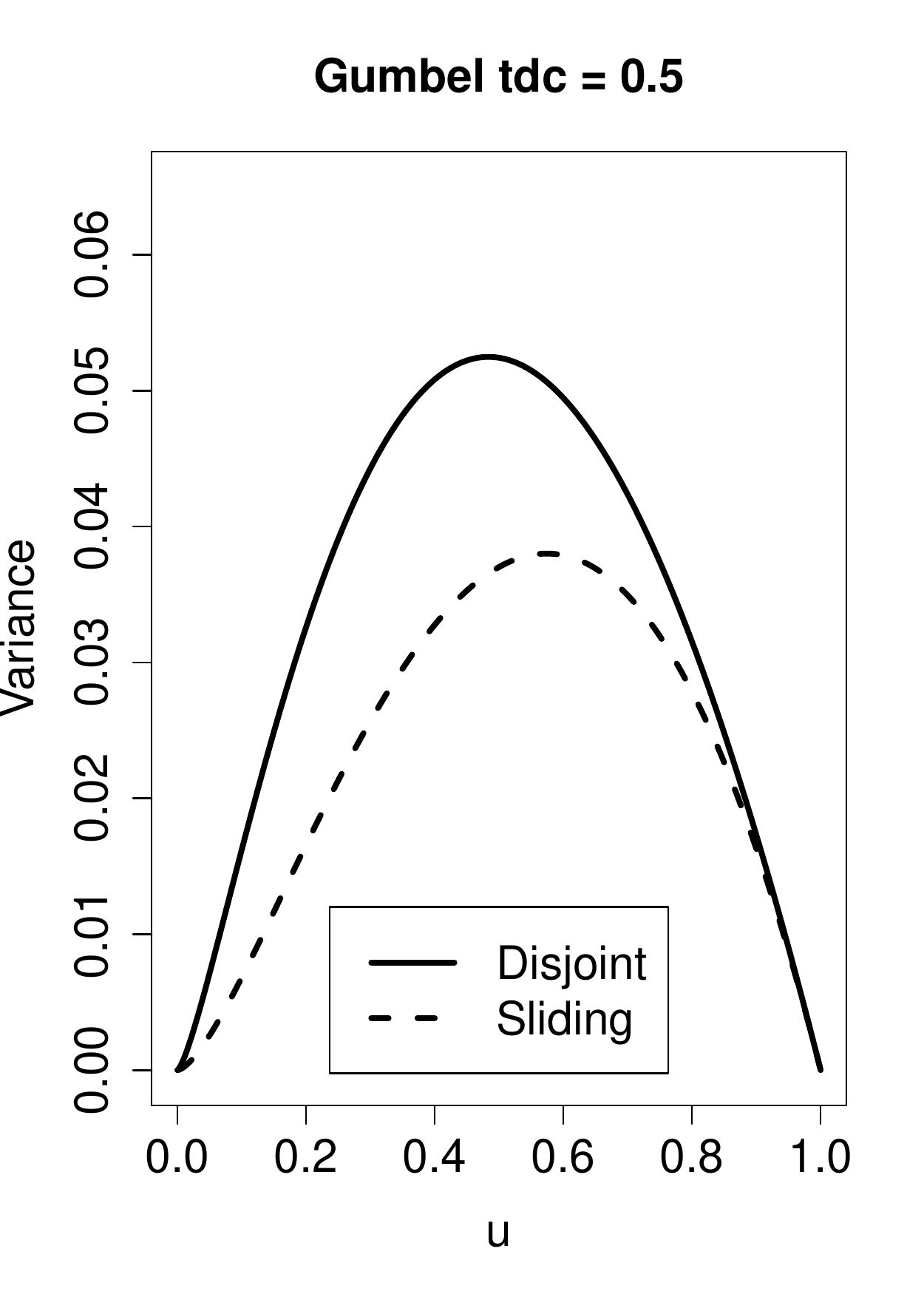}\hspace{-.2cm}
\includegraphics[width=4.3cm]{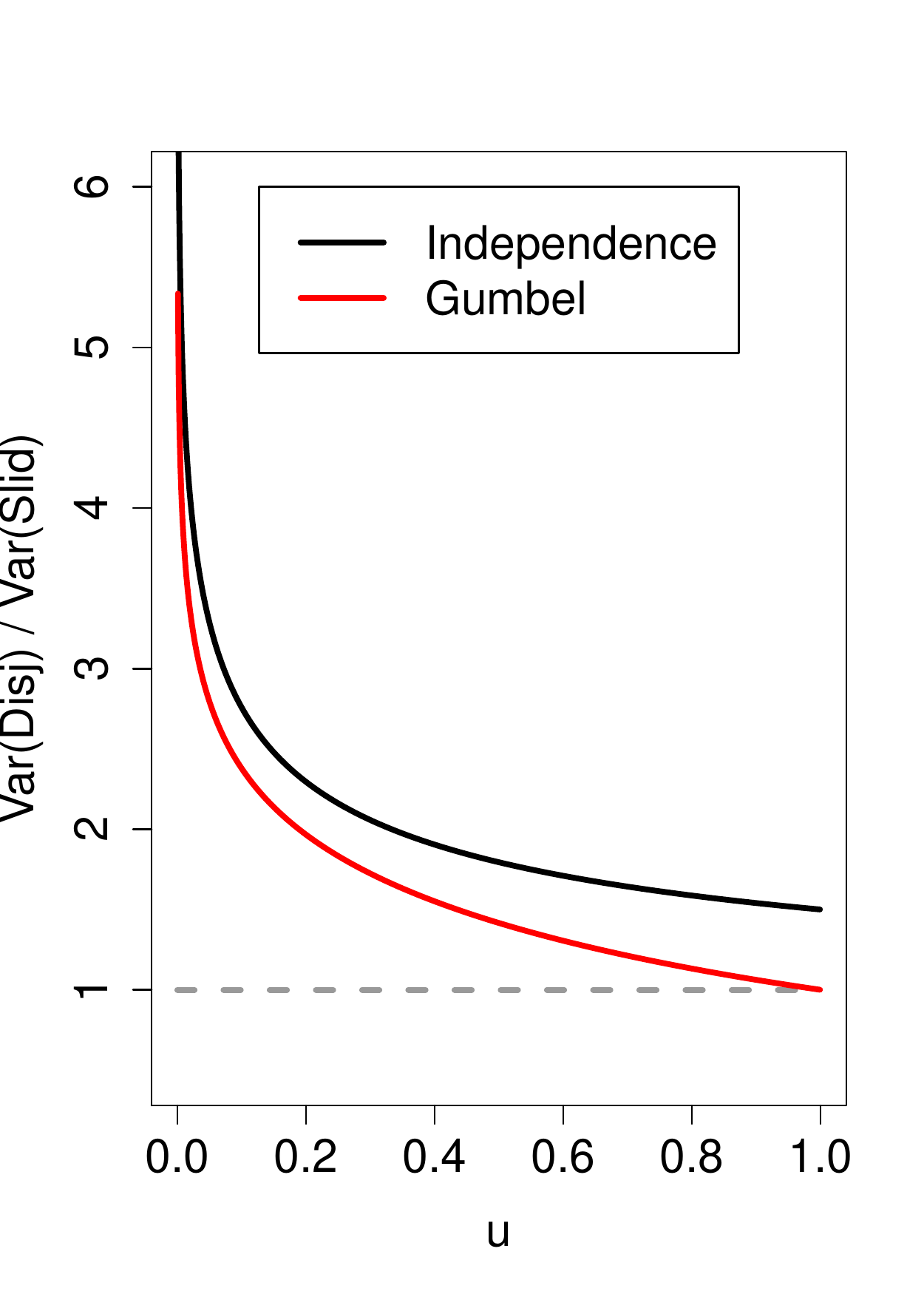}}
\vspace{-.8cm}
\end{center}
\caption{Left plot: $\Var(\widehat\bbC^{\lozenge}(u,u,1))$ (dashed line) and $\Var(\widehat\bbC^{D}(u,u))$ (solid line) as a function of $u \in [0,1]$ for $C_\infty(u,v) = uv$. Middle plot: same with Gumbel--Hougaard copula with tail dependence coefficient $1/2$. Right plot: $\Var(\widehat\bbC^{D}(u,u))/\Var(\widehat\bbC^{\lozenge}(u,u,1))$ as a function of $u \in [0,1]$.} \label{fig:compvarth}
\vspace{-0.3cm}
\end{figure}

As a consequence of the previous result, whenever $T$ is a continuous and linear (real-valued) functional on the space of continuous functions on $[0,1]^d$ (e.g., the Hadamard derivative of a functional $\Phi:\ell^\infty(T) \to \R$ at $C_\infty$, tangentially to the subspace of continuous functions), then 
\[
\Var(T(\widehat\bbC^{\lozenge}(\cdot, 1))) \le  \Var(T(\widehat\bbC^{D})).
\]
Indeed, by the Riesz representation theorem (\citealp{Dud02}, Theorem 7.4.1), $T(\bbC) = \int_{[0,1]^d} \bbC \diff \mu$ for some finite signed Borel measure $\mu$ on $[0,1]^d$, whence 
\[
\Var(T(\bbC))= \int_{[0,1]^d}  \int_{[0,1]^d} \Cov(\bbC(\bm u), \bbC(\bm v)) \diff \mu(\bm u) \diff \mu(\bm v).
\]
The claim then follows by measure-theoretic induction. Examples of interesting functionals $T$ can for instance be found in \cite{GenSeg10}, Section 3, which comprise Blomqvist's beta, Spearman's footrule, Spearman's  rho and Gini's gamma.

%%%%%%%%%%%%%%%%%%%%%%%%%%%%%%%%%%%%%%%%%%%%%%%%%%%%%
\section{Applications of the functional weak convergence} 
\label{sec:app}
%%%%%%%%%%%%%%%%%%%%%%%%%%%%%%%%%%%%%%%%%%%%%%%%%%%%%

The functional weak convergence result in Theorem~\ref{theo:estmar} can be applied to large variety of statistical problems. Classical applications include the derivation of the asymptotic behavior of estimators for the Pickands dependence function, see, e.g., Section 3.3 in \cite{BucSeg14}. Throughout this section, we discuss applications that explicitly make use of the fact that we allow for various block sizes, allowing one to aggregate over those block sizes, to derive bias reduced estimators or to even estimate second order characteristics. 

Despite not being necessary for the bias correction to \textit{work}, many of the results in this section can be formulated in a convenient explicit way under the assumption of a third order condition.

\begin{Assump}[Third order condition]\label{assump:high_order3}~
Assumption~\ref{assump:high_order} holds and there exists a regularly varying function $\cb: \N \to (0,\infty)$ with coefficient of regular variation $\rho_\cb < 0$ and a (necessarily continuous) non-null function $T$ on $[0,1]^d$, not a multiple of $S$, such that, uniformly in $\bu \in [0,1]^d$,
\begin{align} \label{eqn:thirdor}
\lim_{k\to\infty} \frac{1}{\cb(k)} \bigg\{\frac{C_{k}(\bu)-C_\infty(\bu)}{\ca(k)}  - S(\bu) \bigg\} = T(\bu).
\end{align}
\end{Assump}

Under the additional assumption that $(\bX_t)_{t \in \Z}$  is an i.i.d.\ sequence, it can be proved that $\cb$ in the above condition must be regularly varying under mild additional assumptions (it can hence be removed from the assumption).

\begin{Lemma} \label{lem:regvarab} Assume that the time series $(\bX_t)_{t \in \Z}$ is an i.i.d.\ sequence. If Assumption~\ref{assump:high_order} holds and additionally there exists a function $\cb: \N \to (0,\infty)$ with $\cb(k) = o(1)~(k \to \infty)$ and a non-null function $T$ such that~\eqref{eqn:thirdor} holds uniformly in $\bu \in [0,1]^d$, and if the functions $S, S^2/C_\infty$ and $T$ are linearly independent, then $\cb$ is regularly varying of order $\rho_\cb\le 0$.
\end{Lemma}

Next we discuss an additional property of the function $\ca$ from Assumption~\ref{assump:high_order} which allows to quantify the speed of convergence of 
\begin{align} \label{eq:ra}
r_x(k) =  \Big(\frac{\ang{xk}}k\Big)^{\rho_\ca} -  \frac{\ca(\ang{xk})}{\ca(k)} 
\end{align} 
(note that convergence to zero of this difference follows from regular variation of $\ca$). This difference will be important in later parts of the manuscript as it will appear in several bounds that are related to bias correction.

\begin{Lemma} \label{lem:bounddiffa} Assume that $\cX \subset (0,\infty)$ is compact and that there exists a non-negative function $\delta:\N \to [0,\infty)$ with $\lim_{k\to\infty} \delta(k) = 0$ such that, uniformly in $x \in \cX$,
\begin{equation}\label{eq:unifCk}
C_{\ang{xk}}(\bu) = C_k(\bu^{1/x_k})^{x_k} + O(\delta(k)), \qquad (k \to \infty),
\end{equation}
for any $\bu \in (0,1)^d$, where  $x_k := \ang{xk}/k$. Under  Assumption~\ref{assump:high_order3} we have, uniformly in $x \in \cX$, 
\[
r_x(k) =  \Big(\frac{\ang{xk}}{k}\Big)^{\rho_{\ca}}  - \frac{\ca(\ang{xk})}{\ca(k)} = O\Big(\ca(k) + \cb(k) + \delta(k)/\ca(k) \Big) \quad (k \to \infty).
\]  
\end{Lemma}

In the iid case, Equation~\eqref{eq:unifCk} obviously holds with $\delta \equiv 0$. The next result provides a bound on the difference in~\eqref{eq:unifCk} under mixing conditions.

\begin{Lemma}\label{rem:raiid}
Let  Assumption~\ref{assump:copula_convergence} hold with an extreme-value copula $C_\infty$. Further, let $(\bX_t)_{t \in \Z}$ be $\alpha$-mixing with mixing coefficients $\alpha(k) = O(k^{-(1+\varrho)})$ for some $\varrho>0$. Then~\eqref{eq:unifCk} holds with $\delta(k) = O(k^{-(1+\varrho)/(2+\varrho)}\log k )$.  
\end{Lemma}

\subsection{Improved estimation by aggregation over block lengths}\label{sec:simpagg}

Since the functional weak convergence result in Theorem~\ref{theo:estmar} involves a scaling parameter for the block length, we may easily analyse estimators for $C_\infty$ which are based on aggregating over several blocks. 
More formally, we consider the following general construction: for a set $M=M_n \subset \{1,\dots,n\}$ of block length parameters and a set $w=\{w_{n,k}: k \in M_n\}$  of weights satisfying $\sum_{k \in M} w_{n,k} = 1$ for all $n \in \N$, let
\[
\hat C_{n,(M,w)}^{\mathrm{agg}}(\bu) = \sum_{k \in M_n} w_{n,k}\hat C_{n,k}(\bu), \qquad \bm u \in [0,1]^d.
%\frac{\sum_{k \in M_n} w_{n,k}\hat C_{n,k}(\bu)}{\sum_{k \in M_n} w_{n,k}}, \qquad \bm u \in [0,1]^d.
\]
To derive the asymptotic distribution of this weighted aggregated estimator, we make the following assumption on the tuple $(M,w)$.

\begin{Assump}\label{assump:wM}
Let $m=m_n$ denote the sequence from Assumption~\ref{assump:mixing}.
For some closed interval $A = [a_\wedge, a_\vee] \subset (0,\infty)$ of positive length, we have 
\[
M=M_n = \{k \in \N : k / m \in A \}
\] 
and the weights $w_{n,k}$ satisfy $\lim_{n\to\infty} m w_{n,\ang{ma}} = f(a)$ uniformly over $A$ for some continuous $f$ on $A$ with $\int_A f(a) \diff a = 1$. %$\int_A f(a) \diff a \neq 0$
\end{Assump}

For instance, given a continuous function $f$ on $A$ that integrates to unity, we may choose the weights $w_{n,k}= f(k/m) / \{ \sum_{\ell \in M_n} f(\ell/m) \}$.

\begin{Prop}\label{prop:simpleagg} Let any of the sufficient conditions in Theorem~\ref{theo:estmar} be met and assume that additionally Assumption~\ref{assump:wM} is true. Then, in $\ell^\infty([0,1]^d)$, 
\[
\sqrt{\frac{n}{m}}\Big( \hat C_{n,(M,w)}^{\mathrm{agg}}(\cdot) - C_\infty(\cdot) - \sum_{k \in M_n} w_{n,k} \{C_{k}(\cdot) - C_\infty(\cdot)\}\Big) 
\dto 
\int_A f(a) \widehat\bbC^{\lozenge}(\cdot,a) \diff a.  
\]
\end{Prop}

Note that the asymptotic results in Theorem~\ref{theo:estmar} imply that the asymptotic variance of $\hat C_{n,\ang{ma}}(\bu)$ is proportional to $ma/n$. For simplicity ignoring the dependence between $\hat C_{n,k}(\bu)$ for different $k$, this motivates the choice $w_{n,k} = k^{-1} / ( \sum_{\ell \in M_n} \ell^{-1}) $, which is in fact the solution to the minimization problem `minimize $\sum_k (k/n) w_{n,k}^2$ over $w_{n,k}$ with $\sum_k w_{n,k} = 1$'. The corresponding function $f$ is  $f(a) = c/a$, with $c$ a normalizing constant such that the integral over $f$ is one. Despite this being a crude approximation since $\hat C_{n,k}(\bu)$ will be strongly dependent for different values of $k$, it performs reasonably well in simulations where we will see that in many cases it leads to an improvement in MSE. An alternative approach to choosing $w_{n,k}$ would consist of estimating the entire variance-covariance matrix of $\{\hat C_{n,k}(\bu): k \in M_n\}$ and minimize a corresponding quadratic form of $w_{n,k}$. We leave a detailed investigation of this question to future research.

Finally, note that if the second order condition from Assumption~\ref{assump:high_order} holds, then the deterministic bias term (see also the discussion in the next section) in Proposition~\ref{prop:simpleagg} can be further decomposed as
\begin{align*}
B_{n,(M,w)}^{\mathrm{agg}}(\bm u) 
&\equiv 
\sum_{k \in M_n} w_{n,k} \{C_{k}(\bu) - C_\infty(\bu)\} \\
&= 
\ca(m) S(\bu)\sum_{k \in M_n} w_{n,m(k/m)} \{ (k/m)^{\rho_{\ca}}+ o(1)\} 
\\
&= \ca(m) S(\bu) \int_A f(a)a^{\rho_\ca}\diff a + o(\ca(m)).
\end{align*}  
Note in particular that the asymptotic bias vanishes if $\ca(m) \sqrt{n/m} = o(1)$. 

\subsection{Bias correction} 
\label{sec:bcknown}

Before discussing the general methodology in this section, we comment on the notion of bias of $\hat C_{n,\ang{ma}}(\bu)$ as an estimator for the attractor copula $C_\infty(\bm u)$. The difference $\hat C_{n,\ang{ma}}(\bu) - C_\infty(\bm u)$ can be naturally decomposed into two terms
\[
D_{n,m}^\lozenge(\bu,a) = \hat C_{n,\ang{ma}}(\bu) - C_{\ang{ma}}(\bu), 
\quad 
B_{n,m}^{\lozenge}(\bm u,a) = C_{\ang{ma}}(\bu)-C_\infty(\bu).
\]

The first term captures the stochastic part of $\hat C_{n,\ang{ma}}(\bu) - C_\infty(\bm u)$  and may be rewritten as
\[
D_{n,m}^\lozenge(\bu,a) = \sqrt{\frac{m}n} \widehat\bbC^{\lozenge}_{n,m}(\bu,a) = O_\Prob\Big( \sqrt{\frac{m}n}\Big).
\]
Recall that, by Theorem~\ref{theo:estmar}, $\widehat\bbC^{\lozenge}_{n,m}(\bu,a)$ converges to a \textit{centered} Gaussian process. For this reason, throughout the remaining part of this paper, when discussing the bias of an estimator, we mostly concentrate on (versions of) the deterministic sequence $B_{n,m}^{\lozenge}$, which might in fact be of larger order than $O((m/n)^{1/2})$ and which we will call the approximation part of the bias. Note that this is a slight abuse of terminology as we never prove results about $\Exp[D_{n,m}^\lozenge(\bu,a)]$; however, a similar approach has also been taken in \cite{FouDehMer15}.

Regarding the approximation part of the bias, note that the fundamental Assumption~\ref{assump:copula_convergence} only guarantees that $B_{n,m}^{\lozenge} = o(1)$. Under the second order condition from Assumption~\ref{assump:high_order} however, we obtain a hold on both the size and the direction of the bias:
\begin{align} \label{eq:pb}
B_{n,m}^{\lozenge}(\bm u,a)  &= \ca(\ang{ma}) S(\bm u) + o(\ca(\ang{ma})) = \ca(m) a^{\rho_\ca} S(\bm u) + o(\ca(m)) \\
&= O(\ca(m)) \nonumber.
\end{align}
It is the main purpose of this section to exploit the generality of Theorem~\ref{theo:estmar} to construct estimators for $C_\infty$ with a smaller order approximation bias.

More precisely, in the current Section~\ref{sec:bcknown}, we present three approaches on how to reduce the bias under either the preliminary assumption that the second order coefficient $\rho_{\ca}$ is known, or that an estimate $\hat \rho_\ca$ is available. In the next section, we will then discuss how to obtain such an estimate. 
For the remaining parts of Section~\ref{sec:app}, suppose that the third order condition from Assumption~\ref{assump:high_order3} is met, which implies the expansion
\begin{equation}\label{eq:cmexp}
%B_{n,m}^{\lozenge}(\bm u,1) = 
C_m(\bu) - C_\infty(\bu) = \ca(m)S(\bu) + \ca(m)\cb(m)T(\bu) + o(\ca(m)\cb(m)), 
\end{equation} 
$m\to\infty,$ for the approximation part of the bias of $\hat C_{n,m} - C_\infty$.

\subsubsection{Naive bias-corrected estimator} \label{sec:bcnai}
The expansion in \eqref{eq:cmexp} implies that, assuming $ma \in \N$ for simplicity for the moment,
\begin{align*}
C_{ma}(\bu)-C_{m}(\bu) 
&=  \{\ca(ma) - \ca(m)  \} S(\bu) + O(\ca(m)\cb(m))  \\
&= 
(a^{\rho_\ca}-1) \ca(m) S(\bu) + O(\ca(m)\cb(m)).
\end{align*}
This suggests that the leading bias term $\ca(m)S(\bu)$ in Expansion~\eqref{eq:cmexp} can be estimated by the plug-in version $\{\hat C_{m',n}(\bu) - \hat C_{m,n}(\bu)\}/ \{ (m'/m)^{\rho_\ca}-1 \}$ where $m' \neq m$ is an integer and we set $a = m'/m$ in the expansion above. Subtracting this estimated bias from the estimator $\hat C_{n,m}$ naturally leads to the following \textit{naive bias-corrected estimator} 
\[
\hat C_{n,(m,m')}^{\mathrm{bc,nai}}(\bu) = \hat C_{n,m}(\bu) - \frac{\hat C_{n,m'}(\bu) - \hat C_{n,m}(\bu)}{(m'/m)^{\rho_\ca} - 1}.
\]
Note that this estimator is infeasible in practice since $\rho_\ca$ is unknown. A feasible estimator  denoted by $\check C_{n,(m,m')}^{\mathrm{bc,nai}}$, can be obtained by replacing $\rho_\ca$ with an estimator $\hat \rho_\ca$. In the result below we quantify the impact of such a replacement under the mild condition $\hat \rho_\ca = \rho_\ca + \op{1}$, estimators satisfying this assumption will be presented in Section~\ref{sec:estrho} below. Furthermore, it is worthwhile to mention that $\hat C_{n,(m,m')}^{\mathrm{bc,nai}} = \hat C_{n,(m',m)}^{\mathrm{bc,nai}}$ as can be verified by a simple calculation.

Assuming that $m' = \ang{ma}$ for some fixed value $a \in (0,\infty), a \neq 1$, the asymptotic distribution of this estimator is as follows. 

\begin{Prop}\label{prop:bcnai} Let any of the sufficient conditions in Theorem~\ref{theo:estmar} be met. Additionally, suppose that Assumption~\ref{assump:high_order3} is met and assume that $m' = \ang{ma}$ for some fixed constant $0 < a \neq 1$. Then, in $\ell^\infty([0,1]^d)$, 
\begin{multline*}
\sqrt{\frac{n}{m}}\Big( \hat C_{n,(m,m')}^{\mathrm{bc,nai}}(\cdot) - C_\infty(\cdot) - B_{n,(m,m')}^{\mathrm{bc,nai}}(\cdot) \Big) \\
\dto \widehat\bbC^{\lozenge}_{\mathrm{bc,nai}}(\cdot,a) := \widehat\bbC^{\lozenge}(\cdot,1) - \frac{\widehat\bbC^{\lozenge}(\cdot,a) - \widehat\bbC^{\lozenge}(\cdot,1)}{a^{\rho_\ca}-1} ,
\end{multline*}
where the bias term $B_{n,(m,m')}^{\mathrm{bc,nai}}$ admits the expansion
\begin{multline*}
B_{n,(m,m')}^{\mathrm{bc,nai}}(\bu) 
= \Big\{\ca(m) r_a(m) \frac{S(\bm u)}{a^{\rho_\ca}-1} + \ca(m) \cb(m) \frac{1-a^{\rho_\cb}}{1-a^{-\rho_\ca}} T(\bm u) \Big\} \\
+ \ca(m)o\Big(\cb(m)+|r_a(m)|\Big).
\end{multline*}
with $r_a(m) =  (\ang{ma}/m)^{\rho_\ca} - \ca(\ang{ma})/\ca(m) =o(1)$ as in \eqref{eq:ra}.
In particular, we have
\[
\sup_{\bu \in [0,1]^d} |B_{n,(m,m')}^{\mathrm{bc,nai}}(\bu)| = \ca(m)O\Big(\cb(m) + |r_a(m)|\Big).  
\]
If moreover $\hat \rho_\ca$ satisfies $\hat \rho_\ca = \rho_\ca + \op{1}$, then, uniformly in $\bu \in [0,1]^d$ 
\[
\check{C}_{n,(m,m')}^{\mathrm{bc,nai}}(\bu) = \hat C_{n,(m,m')}^{\mathrm{bc,nai}}(\bu) + O_\Prob\Big(|\hat\rho_\ca - \rho_\ca|\{\ca(m) + \sqrt{m/n}\} \Big).
\]
\end{Prop} 

Note that the bias term $B_{n,(m,m')}^{\mathrm{bc,nai}}$ is of smaller order than the bias term $B_{n,m}^{\lozenge}$ of the plain empirical copula based on sliding block maxima, see \eqref{eq:pb}. Moreover, in the i.i.d.\ case, we can further bound $|r_a(m)|$ by $O(\ca(m) + \cb(m))$, see Lemma~\ref{lem:bounddiffa} and Remark~\ref{rem:raiid}.
 
\subsubsection{Improving the naive bias-corrected estimator by aggregation}
The naive bias-corrected estimator is fairly simple since it only considers two block length parameters $m$ and $m'=\ang{am}$. One way to improve this estimator is to consider aggregation over different block lengths; an approach that was shown to work well in~\cite{FouDehMer15} for estimating the stable tail dependence function. Many kinds of aggregation are possible, but for the sake of brevity we will restrict our attention to the following version inspired by Section~\ref{sec:simpagg} (which works well in finite-sample settings as demonstrated in Section~\ref{sec:fin})
\[
\hat C_{n,(m,M,w)}^{\mathrm{bc,agg}}(\bu)  = \sum_{k \in M_n} w_{n,k}\hat C_{n,(m,k)}^{\mathrm{bc,nai}}(\bu).
\]
Here $M=M_n\subset\{1, \dots n\} \setminus \{m_n\}$ and $\{w_{n,k}: k \in M_n\}$ are assumed to satisfy Assumption~\ref{assump:wM}. Similarly to the discussion in Section~\ref{sec:bcnai}, let $\check C_{n,(m,M,w)}^{\mathrm{bc,agg}}$ denote a feasible version of $\hat C_{n,(m,M,w)}^{\mathrm{bc,agg}}$, with $\rho_\ca$ replaced  by $\hat \rho_\ca$.

\begin{Prop}\label{prop:bcagg}Let any of the sufficient conditions in Theorem~\ref{theo:estmar} be met. Additionally, {suppose that Assumption~\ref{assump:high_order3} is met} and that $(M_n,\{w_{n,k}: k \in M_n\})$ satisfies Assumption~\ref{assump:wM} and $1 \notin A$. Then, in $\ell^\infty([0,1]^d)$, 
\[
\sqrt{\frac{n}{m}}\Big( \hat C_{n,(m,M,w)}^{\mathrm{bc,agg}}(\cdot) - C_\infty(\cdot) - B_{n,(m,M,w)}^{\mathrm{bc,agg}}(\cdot) \Big) 
\dto 
\int_A f(a) \widehat\bbC^{\lozenge}_{\mathrm{bc,nai}}(\cdot,a) \diff a,
\]
where the bias term $B_{n,(m,M,w)}^{\mathrm{bc,agg}}$ satisfies 
\begin{multline*}
B_{n,(m,M,w)}^{\mathrm{bc,agg}}(\bu) 
\\
= \int_A f(a)\Big\{\ca(m) r_a(m) \frac{S(\bm u)}{a^{\rho_\ca}-1} + \ca(m) \cb(m) \frac{(1-a^{\rho_\cb})T(\bm u)}{1-a^{-\rho_\ca}} \Big\} \diff a + o(r(m)),
\end{multline*}
where, recalling $r_a(m)$ from \eqref{eq:ra},
\begin{align} \label{eq:rmn}
r(m) = \ca(m)\Big(\cb(m) + \sup_{a \in A}\big|r_a{(m)}\big|\Big).
\end{align}
In particular
\[
\sup_{\bu \in [0,1]^d} |B_{n,(m,M,w)}^{\mathrm{bc,agg}}(\bu)| = O(r({m})).
\]
If moreover $\hat \rho_\ca = \rho_\ca + \op{1}$ then we have, uniformly in $\bu \in [0,1]^d$ 
\[
\check C_{n,(m,M,w)}^{\mathrm{bc,agg}}(\bu) = \hat C_{n,(m,M,w)}^{\mathrm{bc,agg}}(\bu) + O_\Prob\Big(|\hat\rho_\ca - \rho_\ca|\{\ca(m) + \sqrt{m/n}\} \Big).
\] 
\end{Prop}

\subsubsection{Regression-based bias correction}  \label{sec:regbias}

A more sophisticated, regression-based estimator (inspired by \citealp{BeiEscGoeGui16}, where the POT-case is tackled) can be motivated by the following consequence of the expansion in~\eqref{eq:cmexp} and the regular variation of $\ca(\cdot)$: 
\begin{equation}\label{eq:cmexpalt}
C_{\ang{ma}}(\bu) = C_\infty(\bu) + a^{\rho_\ca} \ca(m) S(\bu) + r_{m}(\bu), \qquad m \to \infty,
\end{equation}
for all $a>0$, where $r_{m}(\bu) = o(\ca(m))$. Letting $y_{i,n} := \hat C_{n, k_{i}}(\bu)$ for suitable values $k_{i}$ (to be determined below) we find that
\begin{equation}\label{eq:cmexpreg}
y_{i,n} = C_\infty(\bu) + (k_{i}/m)^{\rho_\ca} \ca(m) S(\bu) + \varepsilon_{i,n}
\end{equation}
where the remainder $\varepsilon_{i,n}$ contains both the stochastic error {$\hat C_{n,k_{i}}(\bu) - C_{k_{i}}(\bu)$} and the deterministic error from expansion~\eqref{eq:cmexpalt}. This motivates the following weighted least square estimator for $C_\infty(\bu)$ and {$B_m(\bm u)= \ca(m)S(\bu)$}: 
\begin{multline}
(\hat C^{\mathrm{bc,reg}}_{n,(M,w)}(\bu), \hat B^{\mathrm{bc,reg}}_{n,(m,M,w)}(\bu))  \\
\in
\argmin_{(b,c)\in\R^2} \sum_{k \in M_n} w_{n,k} \{ \hat C_{n,k}(\bu) - b - (k/m)^{\rho_\ca} c \}^2, \label{eq:defbcreg}
\end{multline}
where $w_{n,k}$ and $M = M_n \subset \{1,\dots ,n\}$ are as in Section~\ref{sec:simpagg} with the additional assumption that the weights $w_{n,k}$ are non-negative. Note that, since the parameter $\rho_\ca$ is fixed in the above minimization problem, the value of $\hat C^{\mathrm{bc,reg}}_{n,(M,w)}(\bu)$ does in fact not depend on $m$ and hence we do not need to consider $m$ as an index in $\hat C^{\mathrm{bc,reg}}_{n,(M,w)}(\bu)$.
Similarly to the discussion in Section~\ref{sec:bcnai}, let $(\check C^{\mathrm{bc,reg}}_{n,(M,w)}(\bu), {\check B^{\mathrm{bc,reg}}_{n,(m,M,w)}(\bu)})$ denote a feasible version of $(\hat C^{\mathrm{bc,reg}}_{n,(M,w)}(\bu), {\hat B^{\mathrm{bc,reg}}_{n,(m,M,w)}(\bu)})$, where $\rho_\ca$ is replaced  by $\hat \rho_\ca$.

Assuming that $M_n$ contains sufficiently many elements so that the inverse matrix in the next display exists, the minimization problem above has the unique closed-form solution 
\begin{multline*}
\left(
\begin{array}{c} 
\hat C^{\mathrm{bc,reg}}_{n,(M,w)}(\bu) \\ 
\hat B^{\mathrm{bc, reg}}_{n,(m,M,w)}(\bu)
\end{array}
\right) 
= \left(
\begin{array}{cc} 
\mu_{0,n} & \mu_{1,n} \\ 
\mu_{1,n} & \mu_{2,n}  
\end{array}
\right)^{-1}
\left(
\begin{array}{c} 
\sum_{k \in M_n} w_{n,k}\hat C_{n,k}(\bu) \\ 
\sum_{k \in M_n} w_{n,k}(k/m)^{\rho_\ca}\hat C_{n,k}(\bu)  
\end{array}
\right) 
\end{multline*}
where we defined $\mu_{v,n} := \sum_{k \in M_n} w_{n,k} (k/m)^{v\rho_\ca}, v= 0,1,2$. 
To state the asymptotics of this estimator, define
\[
\kappa_{v} := \int_A f(a) a^{v \rho_{\ca}} \diff a, \qquad T_v(\bu) := \int_A f(a)a^{v \rho_{\ca}} \widehat\bbC^{\lozenge}(\bu,a) \diff a,
\]
and
\[
\cT_{m,v}(\bu) := \int_{A}f(a)a^{v \rho_{\ca}}\Big\{
a^{\rho_{\ca}+\rho_{\cb}}\ca(m)\cb(m)T(\bu) - \ca(m)r_{a}(m)S(\bu)\Big\}\diff a.
\]
\begin{Prop}\label{prop:bcreg}
Let any of the sufficient conditions in Theorem~\ref{theo:estmar} be met. Additionally, {suppose that Assumption~\ref{assump:high_order3} is met} and that $(M_n,\{w_{n,k}: k \in M_n\})$ satisfies Assumption~\ref{assump:wM}. 
Then, in $\ell^\infty([0,1]^d)$, 
\[
\sqrt{\frac{n}{m}}\Big( \hat C_{n,(M,w)}^{\mathrm{bc,reg}}(\cdot) - C_\infty(\cdot) - B_{n,(m,M,w)}^{\mathrm{bc,reg}}(\cdot) \Big) 
\dto 
\frac{\kappa_2 T_0(\cdot) - \kappa_1 T_1(\cdot)}{\kappa_2 \kappa_0 - \kappa_1^2},
\]
where the bias term $B_{n,(m,M,w)}^{\mathrm{bc,reg}}$ satisfies 

\[
B_{n,(m,M,w)}^{\mathrm{bc,reg}}(\bu) = \frac{\kappa_2 \cT_{m,0}(\bu) - \kappa_1 \cT_{m,1}(\bu)}{\kappa_2 \kappa_0 - \kappa_1^2}+o(r(m)) = O(r(m)),
\]
with $r(m)$ as defined in \eqref{eq:rmn}. 
Moreover,
\[
\sqrt{\frac{n}{m}}\Big(\hat B^{\mathrm{bc, reg}}_{n,(m,M,w)}(\cdot) - \ca(m)S(\cdot) - \Gamma_{n,(m,M,w)}^{B}(\cdot) \Big) 
\dto 
\frac{\kappa_0 T_1(\cdot) - \kappa_1 T_0(\cdot)}{\kappa_2 \kappa_0 - \kappa_1^2}
\]
in $\ell^\infty([0,1]^d)$, 
where the bias term $\Gamma_{n,(m,M,w)}^{B}$ satisfies 
\[
\Gamma_{n,(m,M,w)}^{B}(\bu)=\frac{\kappa_0 \cT_{m,1}(\bu) - \kappa_1 \cT_{m,0}(\bu)}{\kappa_2 \kappa_0 - \kappa_1^2}+o(r(m)) = O(r(m)),
\]
and the processes involving $\hat C_{n,(M,w)}^{\mathrm{bc,reg}},\hat B_{n,(m,M,w)}^{\mathrm{bc,reg}}$ converge jointly. 
If moreover $\hat \rho_\ca = \rho_\ca + \op{1}$, then  we have, uniformly in $\bu \in [0,1]^d$ 
\[
\check C_{n,(M,w)}^{\mathrm{bc,reg}}(\bu) 
= \hat C_{n,(M,w)}^{\mathrm{bc,reg}}(\bu) + O_\Prob\Big(r(m) + |\hat\rho_\ca - \rho_\ca|\{\ca(m) + \sqrt{m/n}\} \Big).
\]  
\end{Prop}

\subsection{Estimating the second order parameter} \label{sec:estrho}
Estimators for $\rho_{\ca}$ can be obtained by considering the expansion in~\eqref{eq:cmexpalt}. A simple estimator  can be based on the observation that, for any $\bu$ with $S(\bu) \neq 0$ and any $a\ne 1$, 
\[
\frac{C_{\ang{m a^2}}(\bu) - C_{m}(\bu)}{C_{\ang{m a}}(\bu) - C_{m}(\bu)} = \frac{a^{2\rho_\ca}-1}{a^{\rho_\ca}-1 } + o(1) = a^{\rho_\ca} + 1 + o(1) , \qquad m \to \infty.
\]
Letting $m_\rho=m_\rho(n)$ denote a block length parameter (typically chosen of smaller order than the block length $m$ used for estimating $C_\infty$, whence the different notation here), this suggests the following naive estimator for $\rho_\ca$:
\[
\hat\rho_{\ca}^{\mathrm{nai}}(a,\bu) = \log_a\Big(\frac{\hat C_{n,\ang{m_\rho a^2}}(\bu) - \hat C_{n,m_\rho}(\bu)}{\hat C_{n,\ang{m_\rho a}}(\bu) - \hat C_{n,m_\rho}(\bu)} - 1\Big).
\]

\begin{Prop} \label{prop:hatrhonai} Let Assumption~\ref{assump:high_order3} be met and let $m_\rho=m_\rho(n)$ be an increasing sequence of integers such that any of the sufficient conditions in Theorem~\ref{theo:estmar} is met for that sequence. Further assume that $(m_\rho/n)^{1/2} = o(\ca(m_\rho))$. Then, for any $\bu \in [0,1]^d$ with $S(\bu) \neq 0$ and any $a\ne 1$, we have
\begin{multline*}
\ca(m_\rho) \sqrt{\frac{n}{m_\rho}} \Big( \hat\rho_{\ca}^{\mathrm{nai}}(a,\bu) - \rho_\ca - \Gamma_{n, m_\rho}^{\rho, \mathrm{nai}}(\bu, a) \Big)  \\
\dto  
{\frac{\widehat \bbC^{\lozenge}(\bu,a^2) - \widehat \bbC^{\lozenge}(\bu,1) - (a^{\rho_\ca}+1)\{\widehat \bbC^{\lozenge}(\bu,a) - \widehat \bbC^{\lozenge}(\bu,1)\}}{S(\bu)a^{\rho_\ca}(a^{\rho_\ca}-1)\log a} },
\end{multline*}
where
\begin{multline*}
\Gamma_{n, m_\rho}^{\rho, \mathrm{nai}}(\bu, a) = \cb(m_\rho)\frac{T(\bu)}{S(\bu)}\frac{(a^{\rho_\ca + \rho_\cb}-1)(a^{\rho_\cb}-1)}{(a^{\rho_\ca}-1)\log a} \\
+ O(r_{a^2}(m_\rho) + r_{a}(m_\rho)+m_\rho^{-1}) + o(\cb(m_\rho)).
\end{multline*}
In particular, we have
\[
\hat\rho_{\ca}^{\mathrm{nai}}(a,\bu) - \rho_{\ca} = O_\Prob\Big(\frac{1}{\ca(m_\rho)}\sqrt{\frac{m_\rho}{n}}\Big) + O(m_\rho^{-1} + r_{a^2}(m_\rho) + r_{a}(m_\rho) + \cb(m_\rho)).
\]
\end{Prop}

While the estimator $\hat\rho_{\ca}^{\mathrm{nai}}(a,\bu)$ defined above is easy to motivate and analyze theoretically, we found in simulations that it does not work well when the sample size $n$ is small or even moderate (up to $n = 5000$). This motivated us to consider alternative estimators by treating $\rho_\ca$ in equation~\eqref{eq:cmexpreg} as unknown. Specifically, we considered estimators of the form
\begin{equation}\label{eq:rhoreg}
(\hat b_0, \hat b_1, \hat \rho_\ca^{\mathrm{reg}}) 
\in
\argmin_{b_0, b_1, \rho < 0} \sum_{k \in M_n} w_{n,k}\Big(\hat C_{n,k}(\bu) - b_0 - b_1 (k/m_\rho)^{\rho}\Big)^2,
\end{equation}
where $w_{n,k}$ and $M = M_n \subset \{1,\dots ,n\}$ are as in Section~\ref{sec:simpagg} with the additional assumption that the weights $w_{n,k}$ are non-negative. This lead to some improvement in performance compared to using $\hat\rho_{\ca}^{\mathrm{nai}}$, but still did not lead to very satisfactory results, prompting us to refine the estimator even further.

To gain an intuitive understanding of the shortcomings of $\hat \rho_\ca^{\mathrm{nai}},\hat \rho_\ca^{\mathrm{reg}}$ as plug-in estimators for bias correction, we take a closer look  at the properties of the quantity
\[
\widetilde C_{n,(m,\ang{ma})}^{\mathrm{nai}}(\bu;\gamma) := \hat C_{n,m}(\bu) - \frac{\hat C_{n,\ang{ma}}(\bu) - \hat C_{n,m}(\bu)}{(\ang{ma}/m)^{\gamma} - 1} ,
\]
which is simply the naive bias-corrected estimator from Section~\ref{sec:bcnai} but with $\gamma<0$ plugged in instead of the true $\rho_\ca$. We next take a close look at the bias and variance of this `estimator' as a function of $\gamma$ under the third order condition from Assumption~\ref{assump:high_order3}. The leading part of the bias is approximately given by
\[
\ca(m)S(\bu)\Big(1 - \frac{a^{\rho_\ca} - 1}{a^\gamma-1} \Big) = \ca(m)S(\bu)\frac{a^\gamma - a^{\rho_\ca}}{a^\gamma-1}.
\] 
A close analysis reveals that $\gamma \mapsto g(\gamma) := |a^\gamma - a^{\rho_\ca}|/|a^\gamma-1|$ is decreasing on $(-\infty,\rho_\ca)$ with $\lim_{\gamma\to - \infty} g(\gamma) = a^{\rho_\ca}$ if $a > 1$ and $\lim_{\gamma\to - \infty} g(\gamma) = 1$ if $a < 1$ and increasing on $(\rho_\ca,0)$ with $\lim_{\gamma\uparrow 0} g(\gamma) = \infty$ for $a \in (0,\infty)\backslash\{1\}$, see Figure~\ref{fig:func_g} for a picture of the graph for two specific choices of $a,\rho_\ca$. Hence the leading bias will never be increased compared to the original estimator if $\gamma$ is smaller than $\rho_\ca$, but can increase dramatically if $\gamma > \rho_\ca$, especially if $\gamma$ gets close to zero. Similarly, the asymptotic variance of the `bias correction part' $\{ \hat C_{n,\ang{ma}}(\bu) - \hat C_{n,m}(\bu) \} / \{ (\ang{ma}/m)^{\gamma} - 1\}$ can be found to be a strictly increasing function of $\gamma$.

\begin{figure}[t]
\begin{center}
\vspace{-.7cm}
\includegraphics[width=0.69\textwidth]{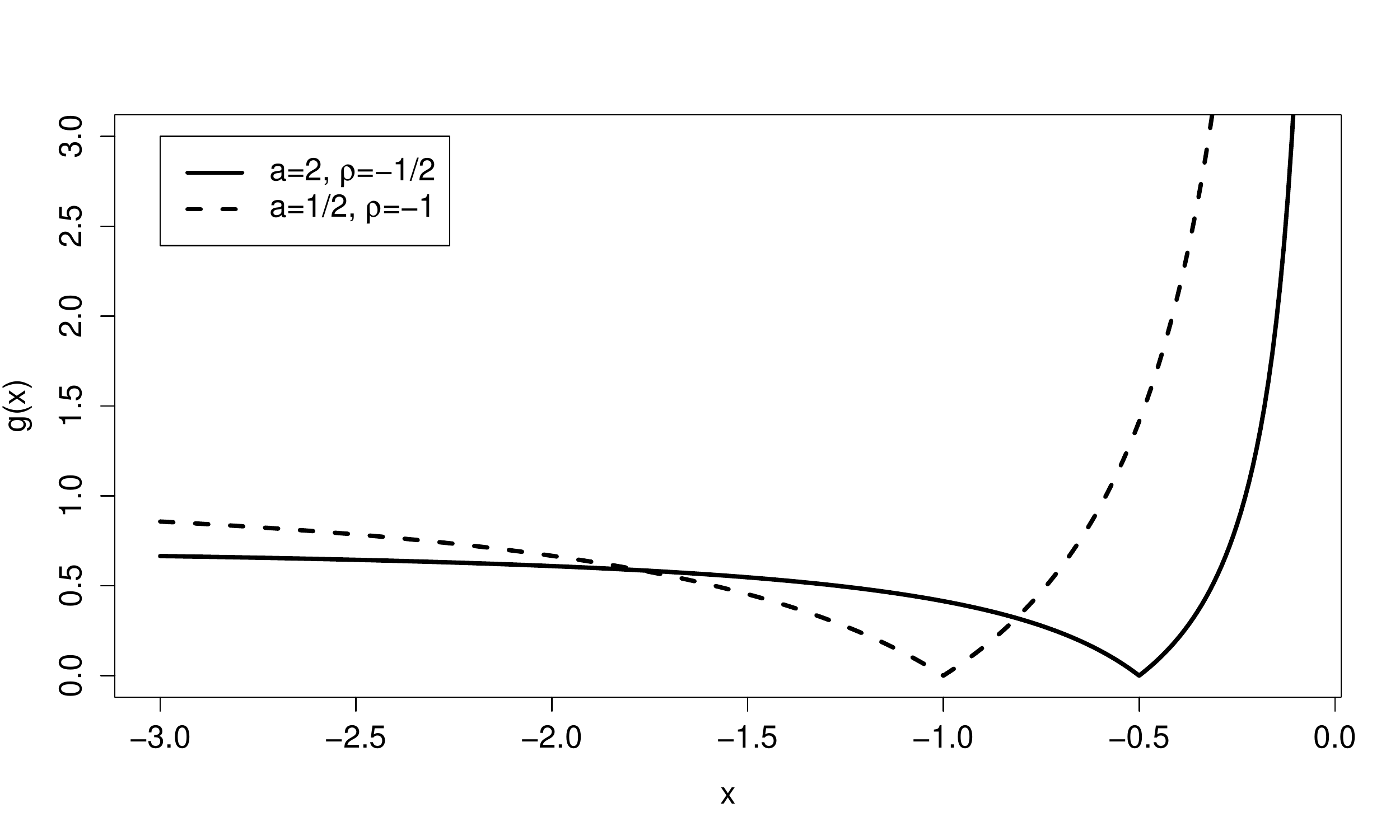}
\vspace{-.7cm}
\end{center}
\caption{Function $g$ for two choices of $(a,\rho_\ca)$.} \label{fig:func_g}
\vspace{-.3cm}
\end{figure}

In summary, the above findings suggest a very asymmetric behavior in the performance of the naive bias corrected estimator with respect to values of $\gamma$ that are too large or too small relative to the true parameter $\rho_\ca$. This apparent asymmetry is not taken into account in the minimization problem~\eqref{eq:rhoreg}. It thus seems natural to introduce an additional penalty term which discourages the estimator of $\rho_\ca$ from being too close to $0$. We hence consider the estimator 
\[
(\hat b_0(\bu), \hat b_1(\bu), \hat \rho_\ca^{\mathrm{pen}}(\bu)) 
\in \argmin_{\rho \in [K',K''],b_0,b_1 \in \R} \hRSS_\eta(b_{0},b_{1},\rho;\bu),
\] 
where $K'<K''<0$ are fixed constants (in the simulations, we choose $K' = -2$ and $K''= -0.1$), $\eta\ge 0$ denotes a penalty parameter, and 
\begin{align*}
\hRSS_\eta(b_{0},b_{1},\rho;\bu)
&=
\tRSS(b_{0},b_{1},\rho;\bu) + \frac{\eta}{|\rho|} {\min_{a_{0}, a_{1} \in \R, K'\leq \kappa\leq K''}\tRSS(a_{0},a_{1},\kappa;\bu)}, \\
\tRSS(b_{0},b_{1},\rho;\bu)
&=
\sum_{k \in M_{m}}w_{n,k} \{ \hC_{n,k}(\bu)-b_{0}-b_{1}(k/m_\rho)^{\rho} \}^{2}.
\end{align*} 
To motivate the factor ${\min_{a_{0}, a_{1} \in \R, K'\leq \kappa\leq K''}\tRSS(a_{0},a_{1},\kappa;\bu)}$ in the penalty, note that, provided this factor is non-zero, an equivalent representation for the corresponding minimization problem is to minimize
\[
\frac{\tRSS(b_{0},b_{1},\rho;\bu)}{\min_{a_{0}, a_{1} \in \R, K'\leq \kappa\leq K''}\tRSS(a_{0},a_{1},\kappa;\bu)} + \frac{\eta}{|\rho|}.
\]
Since the minimal achievable value of the ratio equals $1$, this automatically provides a scaling for the penalty part $\frac{\eta}{|\rho|}$ and makes this choice attractive in practice. 
Finally, observe that the procedure described above produces an estimator of $\rho_\ca$ for each value of $\bu$. We hence propose to further aggregate estimators $\hat \rho_\ca^{\mathrm{pen}}(\bu)$ across different values of $\bu \in U$ for some finite set $U \subset (0,1)^d$ to obtain the aggregated estimator
\[
\hat \rho_{\ca,U}^{\mathrm{pen,agg}} := \frac{1}{|U|} \sum_{\bu \in U} \hat \rho_\ca^{\mathrm{pen}}(\bu).
\]  
Next we prove consistency of the estimators defined above.

\begin{Prop}\label{prop:reg_est_rho} 
Suppose that Assumption~\ref{assump:high_order} is met with $\rho_\ca \in [K', K'']$ and let $m_\rho=m_\rho(n)$ be an increasing sequence of integers such that any of the sufficient conditions in Theorem~\ref{theo:estmar} is met for that sequence. Further, assume that $\sqrt{n/m_\rho}\ca(m_\rho)\to \infty$, that Assumption~\ref{assump:wM} is met with $m_\rho$ instead of $m$, and that $w_{n,k} > 0$ for all $k,n$. Then, for any compact $U \subset \{\bu\in [0,1]^{d}:S(\bu)\neq 0\}$ and any fixed $\eta \geq 0$,
\begin{align*}
\sup_{\bu\in U}|\rho_\ca^{\mathrm{pen}}(\bm u)-\rho_\ca|=\op{1}.
\end{align*}
Also, $\hat \rho_{\ca,U}^{\mathrm{pen,agg}} = \rho_\ca + \op{1}$ for any finite set $U \subset \{\bu\in [0,1]^{d}:S(\bu)\neq 0\}$.\end{Prop}

\section{Examples and finite-sample properties} \label{sec:fin}

The proposed estimators will be compared in a simulation study. We begin by providing some details on several examples that will be used in the simulations. For the sake of simplicity, we only consider the case $d=2$ below. For a generic $d\geq2$, see Examples \ref{example:thighdim} and \ref{example:outerhighdim} in the supplementary material \cite{ZouVolBuc19supp}.

\subsection{Examples} \label{subsec:examples}

\begin{example}[$t$-Copula, iid case]
For degrees of freedom $\nu\in \N$ and correlation $\theta\in(-1,1)$, the $t$-copula is defined, for $(u,v)\in [0,1]^{2}$, as
\[
{\small D(u,v;\nu, \theta) 
=\int_{-\infty}^{t_{\nu}^{-1}(u)} \int_{-\infty}^{t_{\nu}^{-1}(v)}\frac{\Gamma\Big(\frac{\nu+2}{2}\Big)}{\Gamma\big(\frac{\nu}{2}\big)\pi\nu |P|^{1/2}}\bigg(1+\frac{\bx'P^{-1}\bx}{\nu}\bigg)^{-\frac{\nu+2}{2}}\diff x_2 \diff x_1, }
%\quad (u,v)\in [0,1]^{2},
\]
where $\bx=(x_{1},x_{2})'$, $P$ is a $2\times 2$ correlation matrix with off-diagonal element $\theta$, and $t_{\nu}$ is the cumulative distribution function of a standard univariate $t$-distribution with degrees of freedom $\nu$. Let $L$ and $M$ be the first-order and the second-order POT-type limits associated to $D$. More specifically, 
\[
L(x,y)=yt_{v+1}\bigg(\frac{(y/x)^{1/\nu}-\theta}{\sqrt{1-\theta^2}}\sqrt{\nu+1}\bigg)+
xt_{v+1}\bigg(\frac{(x/y)^{1/\nu}-\theta}{\sqrt{1-\theta^2}}\sqrt{\nu+1}\bigg),
\]
and $M=M(x,y)$ is defined in Section 4 and 4.1 of \cite{FouDehMer15}. Recall that 
$
D_{\infty}(e^{-x},e^{-y})=e^{-L( x,y)}.
$
Let 
\[\Gamma_{2}(x,y)=x^2(\partial L/\partial x)(x,y)+y^2(\partial L/\partial y)(x,y).
\]
By Theorem 2.6 of \cite{BucVolZou19}, Assumption \ref{assump:high_order} holds for $(D_{m})_{m\in\N}$ with $D_m(u,v)=D(u^{1/m},v^{1/m})^m$. Specifically, when $\nu=1$, we have $\rho_\ca=-1$, $\ca(m)=(2m)^{-1}$, and 
\[
S(e^{-x},e^{-y})=D_{\infty}(e^{-x},e^{-y})(\Gamma_{2}(x,y)-L^{2}(x,y));
\]
when $\nu=2$, we have $\rho_\ca =-1$, $\ca(m)=(2m/3)^{-1}$, and 
\[
S(e^{-x},e^{-y})=D_{\infty}(e^{-x},e^{-y})\big[(1/3)(\Gamma_{2}(x,y)-L^{2}(x,y))-(2/3)M(x,y)\big];
\]
when $\nu=3,4,\dots$, we have $\rho_\ca=-2\nu^{-1}$, $\ca(m)=m^{\rho_\ca}$, and
\[
S(e^{-x},e^{-y})=-D_{\infty}(e^{-x},e^{-y})M(x,y).
\]
\end{example}
 
\begin{example}[Outer-power transformation of Clayton Copula, iid case]
For $\theta>0$ and $\beta\geq 1$, the outer-power transformation of a Clayton Copula is defined as 
\[
D(u,v;\theta,\beta)=\Big[1+\big\{(u^{-\theta}-1)^{\beta}+(v^{-\theta}-1)^{\beta}\big\}^{1/\beta}\Big]^{-1/\theta},\quad (u,v)\in [0,1]^{2}
\]
which is to be interpreted as zero if $\min(u,v)=0$. By Theorem 4.1 in \cite{Cha09}, $D$ is in the copula domain of attraction of the Gumbel--Hougaard Copula with shape parameter $\beta$, defined by
\begin{equation}\label{eqn:Gumbel-Hougaard}
D_{\infty}(u,v)=D(u,v;\beta)\coloneqq \exp\Big[-\big\{(-\log u)^{\beta}+(-\log v )^{\beta}\big\}^{1/\beta}\Big], \quad (u,v)\in [0,1]^{2},    
\end{equation}
which is again to be interpreted as zero if $\min(u,v)=0$. Further, by Proposition 4.3 of \cite{BucSeg14}, Assumption \ref{assump:high_order} is met with $\rho_\ca=-1$, $\ca(m)=(2m)^{-1}$, and 
\[
S(u,v)=\theta \Lambda(u,v;\beta),
\]
where, letting $x=-\log u$ and $y=-\log v$,
\[
\Lambda(u,v;\beta)=D(u,v;\beta)\Big\{\big(x^{\beta}+y^{\beta}\big)^{2/\beta}-\big(x^{\beta}+y^{\beta}\big)^{1/\beta-1}\big(x^{\beta+1}+y^{\beta+1}\big)\Big\}.
\]
\end{example} 

\begin{example}[Moving-Maximum-Process]
\label{ex:movmax}
Let $D$ denote a copula and let $(\bm W_t)_{t\in\Z}$ denote an iid sequence from $D$. Fix $p\in\N$ and let $a_{ij}$ $(i=0, \dots, p; j=1, \dots, d)$ denote nonnegative constants satisfying 
\[
\sum_{i=0}^p a_{ij} = 1 \quad (j=1, \dots, d).
\]
The moving maximum process $(\bm U_t)_{t \in \Z}$ of order $p$ is defined as
\[
U_{tj} = \max_{i=0, \dots, p} W_{t-i,j}^{1/a_{ij}}, \qquad (t \in \Z; j=1, \dots, d),
\]
with the convention that $w^{1/0} =0$ for $w\in(0,1)$. As suggested by the notation, the random variables $U_{tj}$ are uniformly distributed on $(0,1)$, whence a model with arbitrary continuous margins can easily be obtained by considering $X_{tj}=\eta_j(U_{tj})$ for some strictly increasing (quantile) function $\eta_j:(0,1) \to \R$.

Assume that the copula $D$ is in the (iid) copula domain of attraction of an extreme-value copula $D_\infty$, that is, for any $\bu \in [0,1]^d$,
\[
D_m(\bu) = \{ D(\bu^{1/m}) \}^m \longrightarrow D_\infty(\bm u) \qquad (k\to\infty).
\]
Note that $D_m$ is the copula of the componentwise block maximum of size $m$, based on the sequence $(\bm W_t)_{t\in\N}$. 

As a consequence of Proposition 4.1 in \cite{BucSeg14}, if $C_m$ denotes the copula of the componentwise block maximum of size $m$ based on the sequence $(\bm U_t)_{t \in \N}$, then
\[
\lim_{m\to\infty} C_m(\bu) = D_\infty(u), \quad \bu\in[0,1]^d
\]
as well, i.e., Assumption~\ref{assump:copula_convergence} is met. We prove in the Appendix that if Assumption~\ref{assump:high_order} is met for $(D_m)_m$ (denote the auxiliary function by $\ca_D$ and $S_D$), then it is also met for $(C_m)_m$ provided that $1/m = o(\ca_D(m))$, with the same auxiliary functions. In case $1/m \not= o(\ca_D(m))$ additional technical assumptions are needed and the functions $\ca_D,S_D$ and $\ca,S$ might differ. Details in the general case are omitted for the sake of brevity.
\end{example}

\subsection{Finite-sample properties}\label{sec:finiteprop}

In this section we compare the estimators for $C_\infty$ introduced in the previous section by means of Monte-Carlo simulations. We focus on the case $d=2$ below; respective results in higher dimensions are quite similar and do not reveal additional deep insights, see the cases $d = 4,8$ treated in Section \ref{sec:highdim} in the supplementary material \cite{ZouVolBuc19supp}. Results for all estimators are reported as follows: each estimator is computed for all values $\bu \in \mathcal{U} := \{.1,.2,\dots,.9\}^2$ and block size $m \in \{1,\dots,20\}$ (except for the aggregated versions, for which we specify the set of block length parameters below). Squared bias, variance and MSE of each estimator and in each point $\bu \in \mathcal{U}$ for sample size $n = 1000$ was estimated based on $1000$ Monte Carlo replications. For the sake of brevity we only report summary results which correspond to taking averages of the squared bias, MSE and variance over all values $\bu \in \mathcal{U}$. We present results on the following models.
\begin{enumerate}
\item[(M1)] iid realizations from an Outer Power Clayton Copula with $d = 2, \theta = 1, \beta = \log(2)/\log(2-0.25)$. 
\item[(M2)] A moving maximum process based on the outer Power Clayton Copula with $d=2, \theta = 1, \beta = \log(2)/\log(2-0.25)$ and $a_{11} = 0.25, a_{12} = 0.5$. 
\item[(M3)] iid realizations from a $t$-Copula with $d = 2, \nu = 5, \theta = 0.5$.  
\item[(M4)] A moving maximum process based on a $t$-Copula with $d = 2, \nu = 5, \theta = 0.5$ and $a_{11} = 0.25, a_{12} = 0.5$. 
\item[(M5)] A moving maximum process based on a $t$-Copula with $d = 2, \nu = 3, \theta = 0.25$ and $a_{11} = 0.25, a_{12} = 0.5$. 
\end{enumerate}
For the sake of brevity, we do not include an iid version of Model (M5) because the findings are very similar to the time series case. Further note that we also investigated other parameter combinations, but chose to only present results for the above models as they provide, to a large extent, a representative subset of the results.

Following the heuristics after Proposition~\ref{prop:simpleagg}, weights $w=\{w_{n,k}: k \in M\}$ are always chosen as
\begin{align} \label{eq:wc}
\textstyle w_{n,k} = k^{-1}\big(\sum_{\ell \in M} \ell^{-1}\big)^{-1},
\end{align}
with block length sets $M=M_n$ as specified below, possibly depending on the specific estimator.

\subsubsection{Comparison of estimators without bias correction}

We first focus on the performance of three estimators that do not involve any bias correction: 
\smallskip
\begin{compactitem}
\item the disjoint blocks estimator $\hat C_{n,m}^D$ from~\cite{BucSeg14}, see also Section~\ref{sec:comp};
\item the sliding blocks estimator $\hat C_{n,m}$ from Section~\ref{sec:estmar};
\item the aggregated sliding blocks estimator $\hat C_{n,(M,w)}^{\mathrm{agg}}$ from Section~\ref{sec:simpagg}, with block length set $M= \{m,m+1,\dots,m+9\}$ and weights as in \eqref{eq:wc}. 
\end{compactitem} 
\smallskip
The respective results corresponding to Models (M1)-(M5) are shown in Figure~\ref{fig:nbc2}. As predicted by the theory, the variance curves are linear in $m$, with the disjoints blocks estimator always exhibiting the largest variance, while the variances of the aggregated and vanilla version of the sliding blocks estimator are both smaller and similar to each other. In terms of bias, the disjoint and vanilla sliding blocks estimators $\hat C_{n,m}^{D}$ and $\hat C_{n,m}$ show a very similar behavior, with only some smaller deviations (in particular visible for larger block sizes) which may possibly be explained by the fact that the disjoint blocks estimator does not make use of all observations in case the block length $m$ is not a divisor of the sample size $n=1000$. The aggregated sliding blocks estimator typically has the smallest bias among the three competitors. Finally, in terms of MSE, the aggregated sliding blocks estimator again shows the uniformly best performance. Except for Model (M5), the global minimum of the MSE-curve for $\hat C_{n,(M,w)}^{\mathrm{agg}}$ is substantially smaller than the minima for the other two estimators.

When comparing the five models, we observe a qualitatively similar behavior for models (M1)-(M4), with the bias typically being larger in the iid case than in the time series setting.
Model (M5) however exhibits little to no bias for all block sizes under consideration, even for $m=1$. As a consequence, at their minimal MSE, the three estimators yield comparably good results. The observant reader might also note that the bias in the serially dependent models seems to be smaller than in the iid case. Intuitively, this is due to the fact that realizations from moving maximum processes are already based on maxima and thus it can be expected that their dependence structure is closer to that of a `limiting' max-stable model described by $C_\infty$.

\begin{figure}[t!]
\begin{center}
%\vspace{-.1cm}
\includegraphics[width=0.98\textwidth]{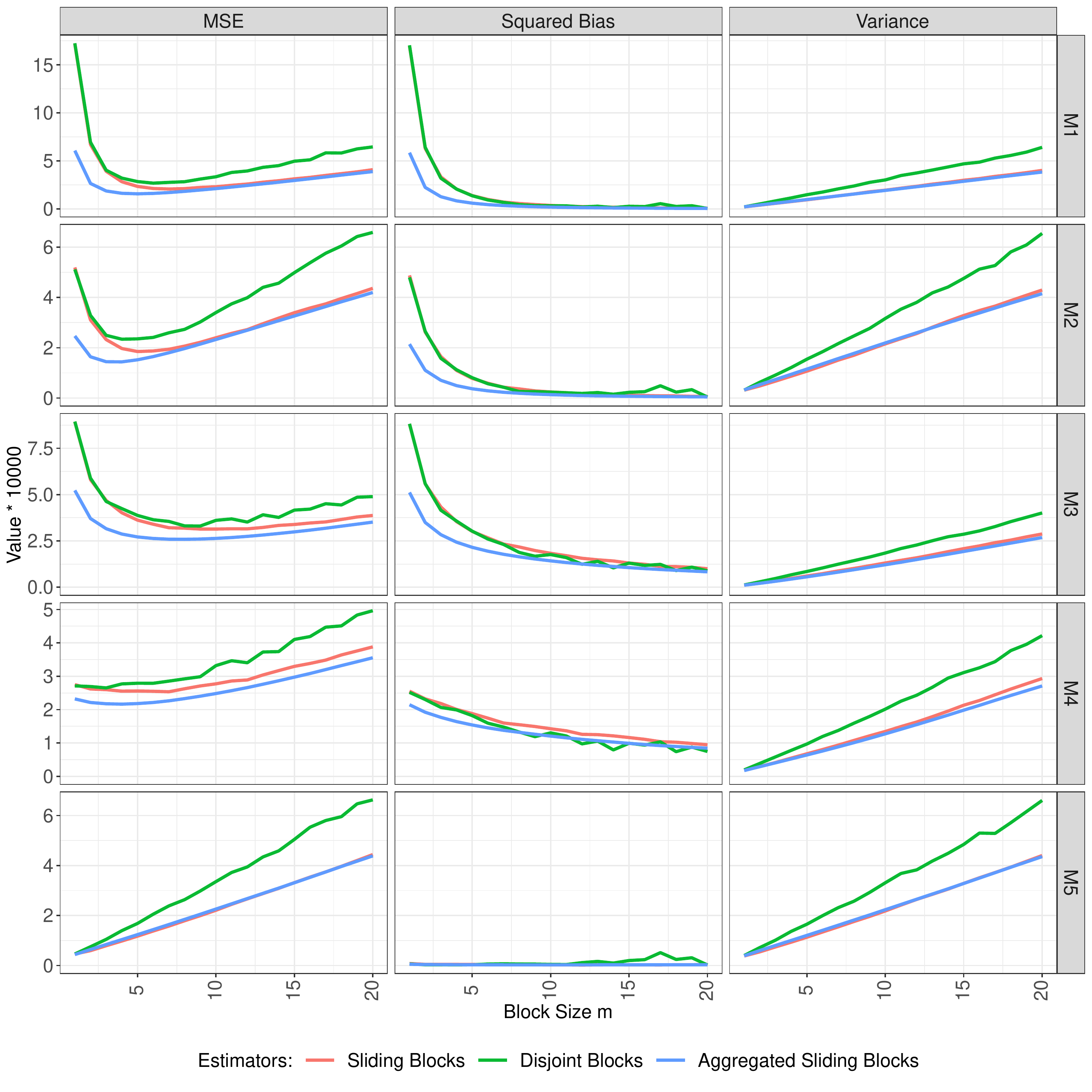}
\vspace{-.6cm}
\end{center}
\caption{\label{fig:nbc2} 
$10^{4} \times $ average MSE, average squared bias and average variance of sliding blocks estimator,  disjoint blocks estimator, and aggregated sliding blocks estimator.}
\vspace{-.3cm}
\end{figure}

\subsubsection{Comparison of bias corrected estimators}

In this section, three bias corrected estimators for the vanilla sliding blocks estimator $\hat C_{n,m}$  are compared  with  $\hat C_{n,m}$ itself. In all cases, the second order parameter $\rho_\ca$ is estimated through $\hat \rho_{\ca} = \hat \rho_{\ca,U}^{\rm pen,agg}$, with the  parameters of that estimator set to $K' = -2, K'' = -0.1, \eta = 1/2, U = \{(.1,.1),(.11,.11),\ldots,(.5,.5)\}, M = \{2,\ldots,50\}$ and weights as in \eqref{eq:wc}.
We consider the following estimators:
\smallskip
\begin{compactitem}
\item The naive bias corrected estimator $\check C_{n,(m,m')}^{\mathrm{bc,nai}}$ with $m' = 1$ and $m \geq 2$. 
\item The aggregated naive bias corrected estimator $\check C^{\mathrm{bc, agg}}_{n,(m',M,w)}$ with $(m',M)=(1,\{m,\dots,m+9\})$ (where $m \geq 2$ is on the x-axis) and with weights as in \eqref{eq:wc}. 
\item The regression-based bias corrected estimator $\check C_{n,(M,w)}^{\mathrm{bc,reg}}$ with $M=\{1,m,m+1,\dots,m+9\}$ (where $m \geq 2$ is on the x-axis) and with weights as in \eqref{eq:wc} (recall from the discussion right after~\eqref{eq:defbcreg} that $\check C_{n,(M,w)}^{\mathrm{bc,reg}}$ does not depend on the parameter $m$ in that equation). 
\end{compactitem} 
\smallskip

The choice of small block sizes for the bias correction, in particular $m'=1$, is motivated by the fact that this choice leads to the best performance in the simulations we tried. Similar observations were made in~\cite{FouDehMer15} who recommend using a very large value for the threshold $k$ in the POT setting.

\begin{figure}[t!]
\begin{center}
\includegraphics[width=0.98\textwidth]{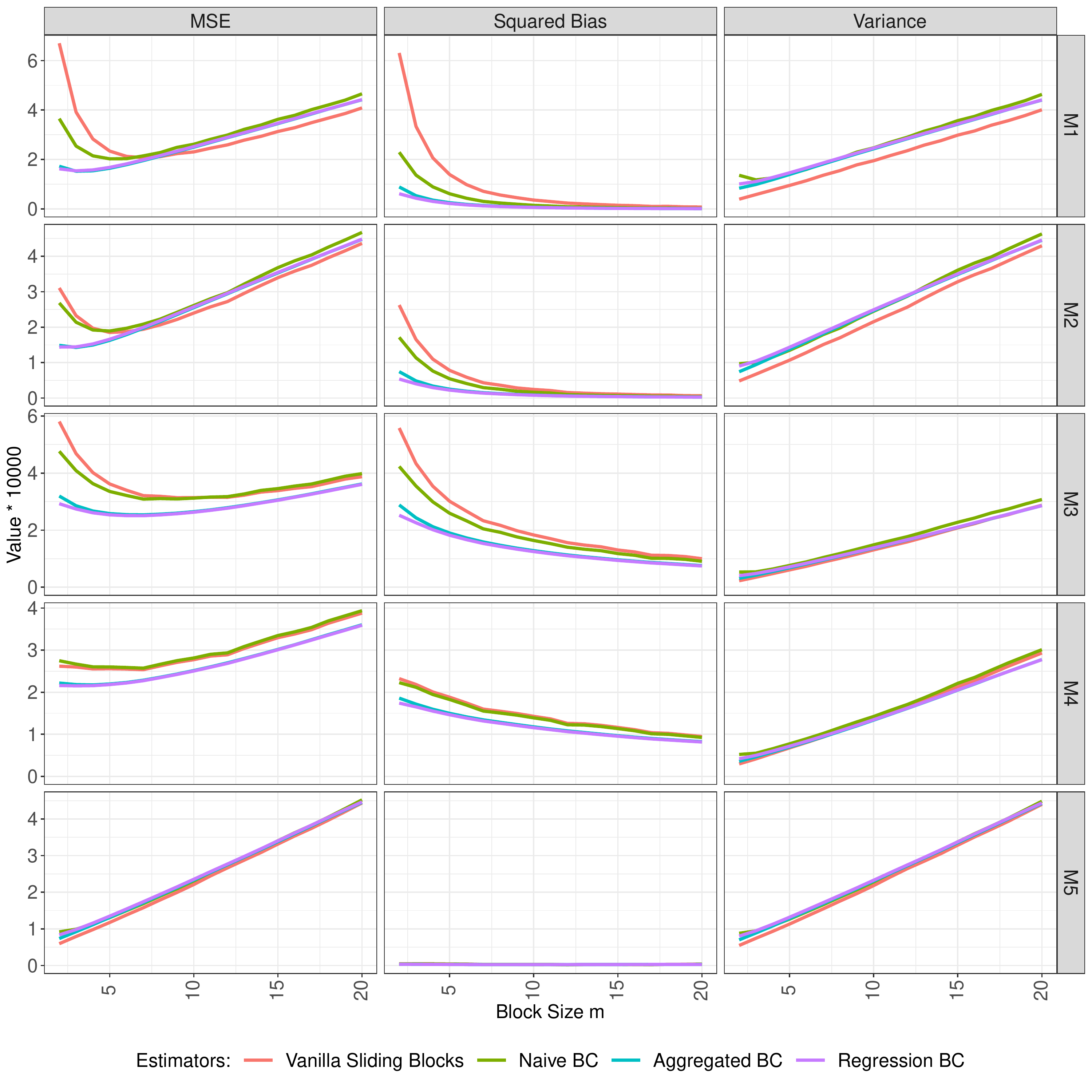}
\vspace{-.6cm}
\end{center}
\caption{\label{fig:bc} 
$10^{4} \times $ average MSE, average squared bias and average variance of sliding blocks estimator,  naive bias corrected estimator, and aggregated naive bias corrected estimator.} 
\vspace{-.3cm}
\end{figure}
The results are presented in Figure~\ref{fig:bc}. We observe that the naive bias corrected estimator exhibits, at each fixed block size, a slightly larger variance and a slightly smaller squared bias than the plain sliding blocks empirical copula. In terms of MSE, no universal statement regarding the ordering between the two estimators can be made. Their minimal MSEs (for each separate model, over all block length parameters) are however quite similar. We further find that aggregating the naive bias-corrected estimator leads to substantial improvements for small values of $m$ and no major impact for larger values of $m$. This is similar to the findings in the previous section. Compared with the vanilla sliding block estimator, the aggregated bias corrected estimator shows much less sensitivity to the parameter $m$ in Model~(M1)-(M3) where there is a substantial bias. In Model~(M5), where the bias is negligible compared to the variance, attempts to correct the bias introduce a bit of variance leading to a slight increase in MSE for all block sizes. Finally, the aggregated naive and regression-based bias corrected estimators show very similar performance.

Based on the simulation results, we would recommend using the aggregated bias corrected estimator among all bias corrected estimators since it leads to better results than the naive estimator, is reasonably fast to compute (see Section \ref{sec:comptime}), and is simpler to implement than the regression-based estimator. At the same time, it is less sensitive to the choice of the block size parameter compared to the estimator without bias correction. 

\section*{Acknowledgments}
The authors would like to thank Sebastian Engelke and Chen Zhou for fruitful discussions. We are also grateful to the Associate Editor and three anonymous Referees for detailed feedback which helped to improve the presentation of our results.

This research has been supported by the Collaborative Research Center ``Statistical modeling of nonlinear dynamic processes'' (SFB 823) of the German Research Foundation and by a Discovery Grant from the Natural
Sciences and Engineering Research Council of Canada, which is gratefully acknowledged. Parts of this paper were written when A.\ B\"ucher was a post doctoral researcher at Ruhr-Universität Bochum.

\begin{supplement}[id=suppA]
\stitle{Supplement to: ``Multiple block sizes and overlapping blocks for multivariate time series extremes''}
\slink[doi]{COMPLETED BY THE TYPESETTER}
\sdatatype{.pdf}
  \sdescription{The supplement contains the proofs for the results in this paper. 
}
\end{supplement}

\bibliographystyle{apalike}
\bibliography{oevc}

\newpage

\thispagestyle{empty}

\newgeometry{total={6.8in, 8in}}

\begin{center}

{\bfseries SUPPLEMENT TO THE PAPER:  \\ [0mm] ``MULTIPLE BLOCK SIZES AND OVERLAPPING BLOCKS FOR \\ [0mm] MULTIVARIATE TIME SERIES EXTREMES''}
\vspace{.5cm}

{\textsc{By Nan Zou, Stanislav Volgushev and Axel Bücher}}

\vspace{.28cm}

{\textit{University of Toronto and Heinrich-Heine-Universit\"at D\"usseldorf}}

\vspace{.28cm}

\begin{center}
\begin{minipage}{.6\textwidth}
{\small \hspace{.5cm} Proofs from the main paper as well as additional simulations are provided. Appendix~\ref{sec:proof1} contains proofs for Section~\ref{sec:master}, Appendix~\ref{sec:proof2} those for Section~\ref{sec:app}, and Appendix~\ref{sec:proof3} those for Section~\ref{subsec:examples}. Additional simulation results are presented in Appendix~\ref{sec:addsim}.}
\end{minipage}
\end{center}
\end{center}

\vspace{.5cm}

\newcommand{\cS}{\mathcal{S}}

\appendix

\section{Proofs for Section~2} 
\label{sec:proof1}

\subsection{Proofs for Section~\ref{sec:fix}}

To keep things self-contained, we begin by repeating the proof of Theorem~\ref{theo:known} from the main text. 

\begin{proof}[Proof of Theorem~\ref{theo:known}]
Recall $b_{a}=n-\ang{ma}+1$, $b=b_1=n-m+1$ and
\begin{align*}
\bbC_{n,m}^{\lozenge}(\bu,a)
&=
\sqrt{n/m}\frac{1}{b_{a}}\sum_{i=1}^{b_{a}}\Big( \ind(\bU_{\ang{ma},i}\leq \bu)-\Prob(\bU_{\ang{ma},i}\leq \bu)\Big), \\
\bbC_{n,m}^{\lozenge,b}(\bu,a)
&=
\sqrt{n/m}\frac{1}{b}\sum_{i=1}^{b} \Big(\ind(\bU_{\ang{ma},i}\leq \bu)-\Prob(\bU_{\ang{ma},i}\leq \bu)\Big).
\end{align*}
where  $(\bu,a) \in [0,1]^{d}\times A]$ with $A=[a_\wedge, a_\vee]$. 
Below we provide technical details for the following steps:
\begin{compactenum}
\item[(i)] In Lemma~\ref{lemma:number_of_block}, we will prove that $\|\bbC_{n,m}^{\lozenge}-\bbC_{n,m}^{\lozenge,b}\|_\infty \pto 0$. Hence it suffices to prove weak convergence of $\bbC_{n,m}^{\lozenge,b}$.
\item[(ii)] In Lemma~\ref{lem:tight} we will show that $\bbC_{n,m}^{\lozenge,b}$ is asymptotically uniformly equicontinuous in probability with respect to the $\|\cdot\|_\infty$-norm on $[0,1]^d\times A$.
\item[(iii)] We will prove in Lemma~\ref{lemma:fidi_convergence} that the finite-dimensional distributions of $\bbC_{n,m}^{\lozenge,b}$ converge weakly to those of $\bbC^{\lozenge}$.
\end{compactenum}
Weak convergence of $\bbC_{n,m}^{\lozenge}$ and hence the theorem then follows by combining (i)-(iii). 
\end{proof}

\subsubsection{Proof of Step (i)}

\begin{Lemma}\label{lemma:number_of_block}
Under Assumption \ref{assump:mixing},
\begin{equation*}
    \sup_{(\bu,a)\in[0,1]^{d} \times A}|\bbC_{n,m}^{\lozenge}(\bu,a)-\bbC_{n,m}^{\lozenge,b}(\bu,a)|\pto 0. 
\end{equation*}
\end{Lemma}

\begin{proof}[Proof of Lemma \ref{lemma:number_of_block}]
Decompose
$\bbC_{n,m}^{\lozenge}-\bbC_{n,m}^{\lozenge,b}=B_{1}+B_{2},$
where
\begin{align*}
B_{1}(\bu,a)
&=
\sqrt{n/m}\frac{1}{b_{a}}\sum_{i=1}^{b_{a}}\Big( \ind(\bU_{\ang{ma},i}\leq \bu)-\Prob(\bU_{\ang{ma},i}\leq \bu)\Big)\\
&\mathph{=} \qquad
-\sqrt{n/m}\frac{1}{b_{a}}\sumb\Big( \ind(\bU_{\ang{ma},i}\leq \bu)-\Prob(\bU_{\ang{ma},i}\leq \bu)\Big)\\
B_{2}(\bu,a)
&=
\sqrt{n/m} \Big( \frac{1}{b_{a}} - \frac{1}{b} \Big) \sumb\Big( \ind(\bU_{\ang{ma},i}\leq \bu)-\Prob(\bU_{\ang{ma},i}\leq \bu)\Big).
\end{align*}
By Assumption \ref{assump:mixing}, we have
\begin{align*}
    \sup_{(\bu,a)\in[0,1]^{d} \times A}|B_{1}(\bu,a)|&\leq 
    \sup_{a\in A }\sqrt{n/(mb_{a}^2)}|b_{a}-b|=O(\sqrt{m/n})=o(1).
\end{align*}
Similarly, 
\begin{align*}
    \sup_{(\bu,a)\in[0,1]^{d} \times A }|B_{2}(\bu,a)|&\leq 
    \sup_{a\in A} \sqrt{n/m} \Big| \frac{b-b_a}{bb_a} \Big| b  =O(\sqrt{m/n})=o(1). 
\end{align*}
which implies the assertion.
\end{proof}

\subsubsection{Proof of Step (ii): asymptotic equicontinuity.}

\begin{Lemma}~\label{lem:tight}
Under the conditions of Theorem~\ref{theo:known}, $\bbC_{n,m}^{\lozenge,b}$ is asymptotically uniformly equicontinuous in probability with respect to the $\|\cdot\|_\infty$-norm on $[0,1]^d\times A$.
\end{Lemma}

\noindent \textit{Proof.}
The proof is based on a  blocking technique. For $i\in \N$, let $\ubM_i \in \R^{(1+ \ang{m a_{\vee}} - \ang{ma_{\wedge}} ) \times d}$ be defined by its entries 
\[
(\ubM_i)_{s,j} = M_{\ang{ma_{\wedge}}+s-1,i,j}, \quad s=1,\dots ,1+\ang{m a_{\vee}} - \ang{ma_{\wedge}} , j=1,\dots,d.
\]
{In words, $(\ubM_i)_{s,j}$ denotes the block maximum of the observations starting at time $i$ with block length $\ang{ma_{\wedge}}+s-1$, in the $j$th coordinate.}
Let $\ell_m = (a_{\vee}+1)m$ (to be interpreted as a maximal block size) and let $\cK=b/(2\ell_m)=O(n/m)$. For simplicity, we shall assume that $\cK$ and $\ell_m$ are integers. For $k=1, \dots , \cK$, let
\begin{align*}
\cA_{k}&=\Big\{2(k-1)\ell_m+1, \dots ,2(k-1)\ell_m + \ell_m\Big\},
\\
\cB_{k}&=\Big\{(2k-1)\ell_m+1, \dots ,(2k-1)\ell_m + \ell_m\Big\},
\end{align*}
such that $\cA_{1}, \cB_{1} \dots , \cA_{\cK} , \cB_{\cK}$ is a partition of $\{1, \dots, b\}$. By the coupling lemma in \cite{Ber79} and \cite{Dou95}, we can construct inductively a triangular array $\{\utbM_{i}\}_{i = 1,\dots,b}$, such that 
\begin{equation}\label{eq:coupling}
\begin{aligned}
&\text{(i)} \ \big\{\utbM_{i}:i\in \cA_{k}\big\}\stackrel{d}{=}\big\{\ubM_{i}:i\in \cA_{k}\big\} \text{ and } \big\{\utbM_{i}:i\in \cB_{k}\big\}\stackrel{d}{=}\big\{\ubM_{i}:i\in \cB_{k}\big\} \text{ for any $k=1, \dots, \cK$} 
\\
&\text{(ii)} \ \Prob\Big(\big\{\utbM_{i}:i\in \cA_{k} \big\}\neq \big\{\ubM_{i}:i\in \cA_{k} \big\}\Big)\leq\beta(m) \text{ and }\Prob\Big(\big\{\utbM_{i}:i\in \cB_{k} \big\}\neq \big\{\ubM_{i}:i\in \cB_{k} \big\}\Big)\leq\beta(m)
\\
&\text{(iii)} \ \big\{\utbM_{i}:i\in \cA_{k}\big\}_{k=1,\dots,\cK}\ \text{and} \
\big\{\utbM_{i}:i\in \cB_{k}\big\}_{k=1,\dots,\cK} \text{ are row-wise independent triangular arrays.}
\end{aligned}
\end{equation}

For $ a\in A, i=1, \dots, b$ and $j=1, \dots, d$, let $\tM_{\ang{ma},i,j}$ denote the $(1+\ang{ma}-\ang{ma_{\wedge}},j)$'th entry of $\utbM_{i}$; note  that $\tilde M_{\ang{ma}, i, j} =_d M_{\ang{ma},i,j}$.
Further, let $\tbM_{\ang{ma},i} \in \R^d$ denote the $1+\ang{ma}-\ang{ma_{\wedge}}$'th column of $\utbM_{i}$, and let
\begin{align*}
%\tM_{\ang{ma},i,j}&=\max_{t=1,\dots,\ang{ma}-1}X_{i+t,j},\\
\tU_{\ang{ma},i,j}&=F_{\ang{ma},j}(\tM_{\ang{ma},i,j}),
\\
\tbU_{\ang{ma},i}&=(\tU_{\ang{ma},i,1},\dots,\tU_{\ang{ma},i,d})'.
\end{align*}
Finally, for $(\bu,a) \in [0,1]^d\times A$, let
\[
\tbbC_{n,m}^{\lozenge,b}(\bu,a) = \sqrt{n/m}\frac{1}{b}\sumb \Big(\ind(\tbU_{\ang{ma},i}\leq \bu)-\Prob(\tbU_{\ang{ma},i}\leq \bu)\Big).
\]
Later we will show that part (ii) of \eqref{eq:coupling} implies
\begin{equation}\label{eq:tightness_difference}
\sup_{(\bu,a)\in[0,1]^{d} \times A}|\tbbC_{n,m}^{\lozenge,b}(\bu,a)-\bbC_{n,m}^{\lozenge,b}(\bu,a)|=\op{1}.
\end{equation}
Hence it suffices to prove asymptotic equicontinuity of $\tbbC_{n,m}^{\lozenge,b}$. To this end observe the representation
\begin{align}\label{eq:tightness_decomposition}
\sqrt{a}\tbbC_{n,m}^{\lozenge,b}(\bu,a) = \bbW_{n,m}(\bu,a) + \bbV_{n,m}(\bu,a),
\end{align}
where $\bbW_{n,m}$ and $\bbV_{n,m}$ are stochastic processes on $[0,1]^{d}\times A$ defined by
\begin{align*}
\bbW_{n,m}(\bu,a)&=\frac{1}{\sqrt{\cK}}\sum_{k=1}^{\cK}(W_{n,k}(\bu,a)-\Exp W_{n,k}(\bu,a)),\\
\bbV_{n,m}(\bu,a)&=\frac{1}{\sqrt{\cK}} \sum_{k=1}^{\cK}(V_{n,k}(\bu,a)-\Exp V_{n,k}(\bu,a)),
\end{align*}
where
\begin{align*}
W_{n,k}(\bu,a)&= {\sqrt{\frac{n\cK}{b^2m}}} \sumb\ind
(\tbM_{\ang{ma},i}\leq\bF_{\ang{ma}}^{\leftarrow}(\bu))\ind(i\in \cA_{k}),\\
V_{n,k}(\bu,a)&= \sqrt{\frac{n\cK}{b^2m}}\sumb\ind
(\tbM_{\ang{ma},i}\leq\bF_{\ang{ma}}^{\leftarrow}(\bu))\ind(i\in \cB_{k}).
\end{align*}
To prove asymptotic equicontinuity of $\tbbC_{n,m}^{\lozenge,b}$ it suffices to prove asymptotic equicontinuity and boundedness in probability of $\bbW_{n,m}$ and $\bbV_{n,m}$ (note that by assumption the set $A$ is  bounded and bounded away from zero). Since the distribution of both terms is the same, we will focus on $\bbW_{n,m}$. For $k=1, \dots, \cK$ and $a\in A$, let $\uutbM^{(k)} \in \bbR^{\ell_m\times(1+\ang{ma_{\vee}}-\ang{ma_{\wedge}})\times d}$ be defined as
\[
(\uutbM^{(k)})_{i,s,j}= (\utbM_{2(k-1)\ell_m + i})_{s,j}, 
\]
By (i) and (iii) of \eqref{eq:coupling} and stationarity, $\{\uutbM^{(k)}\}_{k=1,\dots,\cK}$ is a row-wise i.i.d. triangular array. Let $\bbG_{\cK}$ denote the empirical process corresponding to those observations. Then
\[
\bbW_{n,m}(\bu,a) = \frac{1}{\sqrt{\cK}}\sum_{k=1}^{\cK}(W_{n,k}(\bu,a)-\Exp W_{n,k}(\bu,a)) =\frac{1}{\sqrt{\cK}}\sum_{k=1}^{\cK}  \Big(f_{\bu,a}(\uutbM^{(k)})-\Exp f_{\bu,a}(\uutbM^{(k)})\Big)=\bbG_{\cK}f_{u,a}
\]
where $f_{\bu,a}: \bbR^{\ell_m\times (1+\ang{ma_{\vee}}-\ang{ma_{\wedge}})  \times d} \to \bbR$ is defined by
\begin{align} \label{eq:fua}
f_{\bu,a}(\uubx) 
= 
\sqrt{\frac{n\cK}{b^2m}} \sum_{i=1}^{\ell_m} \ind\Big(x_{i,1+\ang{ma}-\ang{ma_{\wedge}},\cdot} \leq \bF_{\ang{ma}}^{\leftarrow}(\bu)\Big).
\end{align}

For $\delta>0$, consider the sequences of functions classes
\begin{align*}
\cF_m &= \Big\{ f_{\bu,a} \Big| \bu \in [0,1]^d, a \in (\bbZ/m)\cap [a_\wedge/2,a_\vee] \Big\}, \\
\cF_{m,\delta} &:= \Big\{ f_{\bu,a} - f_{\bv,c} \Big| \bu, \bv \in [0,1]^d, a,c \in (\bbZ/m)\cap [a_\wedge/2,a_\vee] , \|\bu-\bv\|_{\infty} \vee |c-a| \leq \delta \Big\}.
\end{align*}
Since for all $\bu,\bv\in [0,1]^{d}$ and $a,c\in A$, 
\begin{equation*}
\begin{aligned}
&\mathph{=}\bbW_{n,m}(\bu,a)-\bbW_{n,m}(\bv,c)=\bbW_{n,m}(\bu,\ang{ma}/m)-\bbW_{n,m}(\bv,\ang{mc}/m)\\
&=\frac{1}{\sqrt{\cK}}\sum_{k=1}^{\cK}  \Big((f_{\bu,\ang{ma}/m}-f_{\bv,\ang{mc}/m})(\uutbM^{(k)})-\Exp (f_{\bu,\ang{ma}/m}-f_{\bv,\ang{mc}/m})(\uutbM^{(k)})\Big)\\
&=\bbG_{\cK}(f_{\bu,\ang{ma}/m}-f_{\bv,\ang{mc}/m}),
\end{aligned}    
\end{equation*}
we have, for $n$ sufficiently large so that $\ang{m a_\wedge} > a_\wedge/2$,
\begin{equation*}%\label{eq:equicont_W}
\sup_{\|\bu-\bv\|_{\infty} \vee |c-a| \leq \delta} |\bbW_{n,m}(\bu,a)-\bbW_{n,m}(\bv,c)| \leq 
\|\bbG_{\cK}\|_{\cF_{m,\delta+1/m}}
\end{equation*}
and similarly
\[
\sup_{\bu \in [0,1]^d, a \in A} |\bbW_{n,m}(\bu,a)| \leq \|\bbG_{\cK}\|_{\cF_{m}}.
\]
Hence, the equicontinuity of $\bbW_{n,m}$ will follow if we can prove that 
\begin{equation}\label{eq:equicont_G}
\lim_{\delta \downarrow 0} \limsup_{n \to \infty} \Exp\Big[\|\bbG_{\cK}\|_{\cF_{m,\delta+1/m}} \Big] = 0,
\end{equation}
while the corresponding boundedness in probability will follow from
\begin{equation}\label{eq:unifbound_G}
\limsup_{n \to \infty} \Exp\Big[\|\bbG_{\cK}\|_{\cF_{m}} \Big] < \infty.
\end{equation}
To prove~\eqref{eq:equicont_G} and~\eqref{eq:unifbound_G} we shall apply Theorem 2.14.2 in \cite{VanWel96} to the function classes $\cF_{m,\delta+1/m}$, $\cF_m$ and the empirical process $\bbG_{\cK}$. Let $\|\cdot\|_{P_m,2}$ denote the norm
\[
\|f\|_{P_m,2} = \Big\{\Exp\Big[\Big(f(\uutbM^{(1)})\Big)^2\Big]\Big\}^{1/2}.
\]
Let $E$ be an envelope function for $\cF_{m}$ and note that $2E$ is an envelope function for $\cF_{m,\delta+1/m}$. Define
\[
\ba(x)=x\|2E\|_{P_m,2}/\sqrt{1+\log N_{[~]}(x\|2E\|_{P_m,2}, \cF_{m,\delta+1/m},\|\cdot\|_{P_m,2})}, \quad x > 0,
\]
where $N_{[~]}$ denotes the bracketing number, see Definition 2.1.6 in \cite{VanWel96}. Note that
\begin{equation}\label{eq:bracket_connect}
N_{[~]}(\ep,\cF_{m,\delta+1/m},\|\cdot\|_{P_m,2})\leq N_{[~]}(\ep,\cF_{m}-\cF_{m},\|\cdot\|_{P_m,2})\leq \Big(N_{[~]}(\ep/2,\cF_{m},\|\cdot\|_{P_m,2})\Big)^{2}.
\end{equation}
By the middle part of Theorem 2.14.2 in \cite{VanWel96}, 
if there exists $\kappa>0$ such that for every $f\in \cF_{m,\delta+1/m}$,
\begin{align*}
\|f\|_{P_m,2}< \kappa \|2E\|_{P_m,2}, 
\end{align*}
then 
\begin{multline}
\Exp\Big|\|\bbG_{\cK}\|_{\cF_{m,\delta+1/m}} \Big| %\nonumber
\\
\lesssim \|2E\|_{P_m,2}\int_{0}^{\kappa}\sqrt{1+\log N_{[~]}(\ep \|2E\|_{P_m,2}, \cF_{m,\delta+1/m}, \|\cdot\|_{P_m,2})} \diff \ep +\sqrt{\cK}\Exp\Big[2E\ind\{2E>\sqrt{\cK}\ba(\kappa)\}\Big]. \label{eq:Fmdelta_b1}
\end{multline}
We begin by observing that, for any $f_{\bu,a}$ as defined in \eqref{eq:fua}, {we have
\[
\|f_{\bu,a}\|_{\infty}
= \sup_{\uubx \in  \R^{\ell_m \times (1+\ang{ma_{\vee}}-\ang{ma_{\wedge}}) \times d}} |f_{\bu, a}(\uubx)|
 \leq \ell_m \sqrt{\frac{n\cK}{b^2m}} \le \sqrt{{a_{\vee}+1}}.
\]
for all sufficiently large $n$ (using that $b \ge n/2$, eventually). 
Hence, we can choose $E=\sqrt{a_{\vee}+1}$} as an envelope function of $\cF_{m}$. Later we shall prove that there exist $\eta_1>0,\eta_2>0,\ep_0 > 0$ such that 
\begin{equation}\label{eq:bracketing_number_bound}
N_{[~]}(\ep \|E\|_{P_m,2},\cF_{m},\|\cdot\|_{P_m,2}) \leq \ep^{-\eta_{1}} \qquad \forall \ep\in(0,\ep_0),
\end{equation} 
and, for all $\delta \in (0, \ep_0/2)$ and all sufficiently large $n$ (such that $\delta+1/m < \ep_0$), %$\ep_0>\delta+2/m>0$
\begin{equation}\label{eq:secmomWn}
\sup_{f \in \cF_{m,\delta+1/m}} \|f\|_{P_m,2} \leq (\delta+1/m)^{\eta_{2}} \|2E\|_{P_m,2}.
\end{equation}
Now~\eqref{eq:Fmdelta_b1} together with some simple computations utilizing~\eqref{eq:bracket_connect},~\eqref{eq:bracketing_number_bound} and~\eqref{eq:secmomWn} shows that for $\ep_0>\delta+1/m>0$
\begin{equation*}\label{eq:theorem_2_14_2}
\begin{aligned}
\Exp\Big[\|\bbG_{\cK}\|_{\cF_{m,\delta+1/m}}\Big] \lesssim &~ \int_{0}^{(\delta+1/m)^{\eta_{2}}} \sqrt{1+2\eta_{1}|\log \ep|} \diff \ep
+ \sqrt{\cK}\Prob\Big(2E \geq \sqrt{\cK}\ba((\delta+1/m)^{\eta_{2}})\Big). 
\end{aligned}
\end{equation*}
For fixed $\delta > 0$ the term $\ba((\delta+1/m)^{\eta_{2}})$ is bounded away from $0$ uniformly in $m$ while $\cK = \cK_n \to \infty$ as $n \to \infty$. This implies \eqref{eq:equicont_G}.  

The bound in~\eqref{eq:unifbound_G} follows by similar but simpler arguments utilizing the last part of Theorem 2.14.2 in \cite{VanWel96}.
\hfill $\Box$ 

\medskip

\noindent\textbf{Proof of (\ref{eq:tightness_difference}).} 
By (i) of \eqref{eq:coupling}, we have $\Prob(\tbU_{\ang{ma},i}\leq \bu)=\Prob(\bU_{\ang{ma},i}\leq \bu)$. Hence,
\begin{align*}
\tbbC_{n,m}^{\lozenge,b}(\bu,a)-\bbC_{n,m}^{\lozenge,b}(\bu,a)
&=
{\sqrt{\frac{n}{b^2m}}} \sumb(\ind(\tbU_{\ang{ma},i}\leq \bu)-\ind(\bU_{\ang{ma},i}\leq \bu))\\
&=
\sqrt{\frac{n}{b^2m}} \sum_{k=1}^{\cK}\sumb(\ind(\tbU_{\ang{ma},i}\leq \bu)-\ind(\bU_{\ang{ma},i}\leq \bu))(\ind(i\in \cA_{k})+\ind(i \in \cB_{k})).
\end{align*}
Now, for fixed $k\in\{1, \dots, \cK\}$, since $|\cA_k| = \ell_m = (a_{\vee}+1) m$, we have
\begin{align*}
    \mathph{\leq}\sup_{(\bu,a)\in[0,1]^{d} \times A}&\Big| \sumb(\ind(\tbU_{\ang{ma},i}\leq \bu)-\ind(\bU_{\ang{ma},i}\leq \bu))\ind(i\in \cA_{k})\Big|\\
    &\leq (a_{\vee}+1)m\sup_{a\in A} \ind \Big(\{\tbU_{\ang{ma},i}:i \in \cA_{k}\}\neq\{\bU_{\ang{ma},i}:i \in \cA_{k}\} \Big) \\
    &\leq (a_{\vee}+1)m\ind \Big(\{\utbM_{m,i}:i \in \cA_{k}\}\neq\{\ubM_{m,i}:i \in \cA_{k}\} \Big).
\end{align*}
Similarly,
\begin{align*}
    \mathph{\leq}\sup_{(\bu,a)\in[0,1]^{d} \times A}&\Big|\sumb(\ind(\tbU_{\ang{ma},i}\leq \bu)-\ind(\bU_{\ang{ma},i}\leq \bu))\ind(i\in \cB_{k})\Big|\\
    &\leq (a_{\vee}+1)m\ind \Big(\{\utbM_{m,i}:i \in \cB_{k}\}\neq\{\ubM_{m,i}:i \in \cB_{k}\} \Big).
\end{align*}
By (ii) of \eqref{eq:coupling} and Assumption \ref{assump:mixing}~(iii),
\[
\Exp\Big[\sup_{(\bu,a)\in[0,1]^{d} \times A}|\tbbC_{n,m}^{\lozenge,b}(\bu,a)-\bbC_{n,m}^{\lozenge,b}(\bu,a)|\Big]\leq \beta(m)\sqrt{\frac{n}{m}}\to 0.
\]
The result follows by Markov's Inequality. \hfill $\Box$ 

\medskip
\noindent\textbf{Proof of (\ref{eq:bracketing_number_bound})} 
 Consider the functions $\bbl_{\bu,a,c}, \bbu_{\bu,a,c}: \bbR^{\ell_m\times(1+\ang{ma_{\vee}}-\ang{ma_{\wedge}})\times d} \to \bbR$ defined by
\begin{align}
\bbl_{\bu,a,c}(\uubx)
&=
{\sqrt{\frac{n\cK}{b^2m}}} \sum_{i=1}^{\ell_m} \ind\Big(x_{i,1+\ang{mc}-\ang{ma_{\wedge}},\cdot} \leq \bF_{\ang{ma}}^{\leftarrow}(\bu)\Big), \label{eq:defl}
\\
\bbu_{\bu,a,c}(\uubx)
&=
\sqrt{\frac{n\cK}{b^2m}} \sum_{i=1}^{\ell_m} \ind\Big(x_{i,1+\ang{ma}-\ang{ma_{\wedge}},\cdot} \leq \bF_{\ang{mc}}^{\leftarrow}(\bu)\Big), \label{eq:defu}
\end{align}
and note that $\bbl_{\bu,a,c} = \bbu_{\bu,c,a}$.

Further,  $\bbl_{\bu,a,c}$ is increasing in $\bu$ (coordinate-wise) and $a$ and decreasing in $c$. Likewise,  $\bbu_{\bu,a,c}$ is increasing in $\bu$ (coordinate-wise) and $c$ and decreasing in $a$. {Subsequently, let $A'=[a_\wedge /2, a_\vee +2]$.}
Lemma~\ref{lemma:Lipschitz_1} implies that there exist $\eta, K \in (0,\infty)$ such that, for all {sufficiently large $n$}, all $\bu,\bv \in [0,1]^d$ and  all $a,c \in A'\cap (\bbZ/m)$ with $|c-a| \geq m^{-1/2}$,
\[
\Big\|\bbl_{\bu,a,c} - \bbu_{\bv,a,c}\Big\|_{P_m,2} \leq K(\|\bu-\bv\|_\infty^{1/2} \vee |c-a|^\eta),
\]
Let $\tilde \eta := \min\{\eta,1/2\}/2 \in (0,1/4]$ and $\epsilon_0 := 1 \wedge K^{-1/\tilde\eta} \in (0,1]$. Then 
we have, for all $\|\bu-\bv\|_\infty\leq \ep_0$ and $m^{-1/2} \leq |c-a| \leq \ep_0$, 
\begin{equation}\label{eq:luboundsimple}
\Big\|\bbl_{\bu,a,c} - \bbu_{\bv,a,c}\Big\|_{P_m,2} \leq K(\|\bu-\bv\|_\infty \vee |c-a|)^{2\tilde\eta} \leq (\|\bu-\bv\|_\infty \vee |c-a|)^{\tilde\eta}. 
\end{equation}
To simplify notation,  we subsequently write $\eta = \tilde \eta$. Begin by considering the case $\ep > 2 m^{-\eta/2}$. For $\bh=(h_{1},\dots,h_{d+1}) \in \bbN^{d+1}$ and $\kappa = \ang{m\ep^{1/\eta}}/m$ define
\begin{align*}
\cD_{\ep,\bh} 
=
 \Big[(h_{1}-1)\kappa,  {h_{1}\kappa\wedge 1} \Big]
\times \dots & \times 
\Big[(h_{d}-1)\kappa, { h_{d} \kappa \wedge 1} \Big]  \\
&  \times 
\Big[(h_{d+1}-1)\kappa +  {\frac{\ang{ma_\wedge}}{m} }, \big(h_{d+1} \kappa + {\frac{\ang{ma_\wedge}}{m}}\big){\wedge \big(a_{\vee}+2\big)} \Big] 
\end{align*}
Then we have, for sufficiently large $n$,
\begin{align}\label{eq:tightness_partition}
[0,1]^{d}\times[a_{\wedge},a_{\vee}] 
&\subset 
\cup_{\bh\in\{1,2, \dots ,\ceil{1/\kappa}\}^{d}
\times\{{  1,2,\dots,\ceil{(a_{\vee}-a_{\wedge} + 1)/\kappa} \}} } \cD_{\ep,\bh} \nonumber \\
&\subset {[0,1]^d \times [a_\wedge/2, a_\vee+2]}
\end{align}
Let $u_{\bh,j}=(h_{j}-1)\kappa $, ${v_{\bh,j}=h_{j}\kappa \wedge 1}$, $a_{\bh}=(h_{d+1}-1)\kappa +  {\ang{ma_\wedge}/{m} }$, $c_{\bh}=\big(h_{d+1}\kappa + {{\ang{ma_\wedge}}/{m}}\big)\wedge \big(a_{\vee}+2\big)$, and
\begin{equation*}
\bu_{\bh}=\Big(u_{\bh,1}, \dots ,u_{\bh,d}\Big), \quad \bv_{\bh}=\Big(v_{\bh,1}, \dots,v_{\bh,d}\Big).
\end{equation*}
Then $(\bu_{\bh},a_{\bh})$ and $(\bv_{\bh},c_{\bh})$ are  the corners of the cuboid $\cD_{\ep,\bh}$ in $[0,1]^{d}\times {[a_{\wedge}/2,a_{\vee}+2}]$. For all $f_{\bu,a}\in \cF_{m}$, by \eqref{eq:tightness_partition}, there exists $\bh\in\N^{d+1}$ such that $(\bu,a)\in \cD_{\ep,\bh}$. For such $f_{\bu,a}$ and $\bh$, we have 
$\bbl_{\bu_{\bh},a_{\bh},c_{\bh}} 
\le 
\bbl_{\bu,a,a} 
=  
f_{\bu,a}
=
\bbu_{\bu,a,a} 
\leq \bbu_{\bv_{\bh},a_{\bh},c_{\bh}}$ by the monotonicity properties of $\bbl$ and $\bbu$. Therefore, $\cF_{m}$ is covered by the collection of brackets 
\begin{equation}\label{eq:brackets_1}
\Big\{[\bbl_{\bu_{\bh},a_{\bh},c_{\bh}},\bbu_{\bv_{\bh},a_{\bh},c_{\bh}}]: \bh\in\{1,2, \dots ,\ceil{1/\kappa}\}^{d}\times \{{  1,2,\dots,\ceil{(a_{\vee}-a_{\wedge} + 1)/\kappa}} \} \Big\}.
\end{equation}
By construction and by~\eqref{eq:luboundsimple} {[note that $\cD_{\ep,\bh}\subset [0,1]^d \times A'$]} we have, for any $(\bu_{\bh},a_{\bh}), (\bv_{\bh},c_{\bh})$,
\[
\Big\|\bbl_{\bu_{\bh},a_{\bh},c_{\bh}} - \bbu_{\bv_{\bh},a_{\bh},c_{\bh}}\Big\|_{P_m,2} \leq \kappa^\eta \leq \ep, 
\]  
i.e., the collection in~\eqref{eq:brackets_1} provides a cover of $\cF_m$ by $\ep$ brackets. This implies
\begin{equation}\label{eq:bracket_bound_1}
N_{[~]}(\ep,\cF_{m},\|\cdot\|_{P_m,2}) \leq 2^{d+1}(a_{\vee} - a_{\wedge} + 1) \kappa^{-(d+1)} \leq 4^{d+1}(a_{\vee} - a_{\wedge} + 1) \ep^{-(d+1)/\eta}. 
\end{equation}

Next consider the case $\ep \leq 2 m^{-\eta/2}$. For the constant $K$ from Lemma~\ref{lemma:Lipschitz_2} let $\kappa := K^{-2}\ep^2$. For $\bh=(h_{1},\dots,h_{d+1})$, let $u_{\bh,j}=(h_{j}-1)\kappa$, $v_{\bh,j}={h_{j}\kappa \wedge 1}$, and $a_{\bh}=(h_{d+1}-1)/m + \floor{ma_{\wedge}}/m$. Then $\cF_{m}$ is covered by the collection of brackets 
\begin{equation}\label{eq:brackets_2}
\Big\{[\bbl_{\bu_{\bh},a_{\bh},a_{\bh}},\bbu_{\bv_{\bh},a_{\bh},a_{\bh}}]: \bh\in\{1,2, \dots , \ceil{1/\kappa} \}^{d}\times\{1,2,\dots, \ceil{m(a_{\vee}-a_{\wedge}){+2}}\}\Big\},
\end{equation}
By Lemma \ref{lemma:Lipschitz_2}, for sufficiently large $m$,
\begin{align*}
\Big\|\bbu_{\bv_{\bh},a_{\bh},a_{\bh}}-\bbl_{\bu_{\bh},a_{\bh},a_{\bh}}\Big\|_{P_m,2} \leq K\|\bu_{\bh}-\bv_{\bh}\|_{\infty}^{1/2} = \ep.
\end{align*}
Hence, \eqref{eq:brackets_2} is a collection of $\ep$-brackets that covers $\cF_{m}$. Notice the number of brackets in the collection \eqref{eq:brackets_2} is bounded by $2^{d+1} \kappa^{-d}m(a_{\vee}-a_{\wedge})\le 2^{d+1+\eta/2} K^{2d} \ep^{-2d-2/\eta}$, {for sufficiently large $m$}. Combining this with~\eqref{eq:bracket_bound_1} we have proved that for  constants $\xi, \ep_0$ depending on $A,\rho,\alpha(\cdot),d$ only we have, for any $0 \leq \ep \leq \ep_0$ (note that for $\ep_0 < 1$ constants can be absorbed into powers of $\ep^{-1}$ by changing $\xi$),
\begin{equation*}%\label{eq:bracket_bound_2}
N_{[~]}(\ep,\cF_{m},\|\cdot\|_{P_m,2})\leq \ep^{-\xi}.    
\end{equation*}
Combining this with~\eqref{eq:bracket_connect} completes the proof of~\eqref{eq:bracketing_number_bound}. \hfill $\Box$

\medskip
\noindent\textbf{Proof of (\ref{eq:secmomWn}).} 
Suppose $a,c \in (\bbZ/m)\cap [a_{\wedge},a_{\vee}]$ and $a\leq c$. Then either $|c-a|>m^{-1/2}$ or $m^{-1} \leq |c-a|\leq m^{-1/2}$ or $c-a=0$. Now we discuss case by case. Begin by observing that
\[
\bbl_{\bu\wedge\bv,a,c} \leq f_{\bu,a}\wedge f_{\bv,c} \leq f_{\bu,a}\vee f_{\bv,c} \leq \bbu_{\bu\vee\bv,a,c}.
\]
As a consequence, 
\begin{align} \label{eq:L_2all}
\|f_{\bu,a} - f_{\bv,c}\|_{P_m,2}
& = 
\|f_{\bu,a}\wedge f_{\bv,c} - f_{\bu,a}\vee f_{\bv,c}\|_{P_m,2} \nonumber \\
&\leq 
\|\bbu_{\bu\vee\bv,a,c} - \bbl_{\bu\wedge\bv,a,c}\|_{P_m,2}  
\\
&= \|\bbu_{\bu\vee\bv,\ang{ma}/m,\ang{mc}/m} - \bbl_{\bu\wedge\bv,\ang{ma}/m,\ang{mc}/m}\|_{P_m,2} \nonumber
\end{align}
By Lemma \ref{lemma:Lipschitz_1}, when $c-a>m^{-1/2}$, we obtain the upper bound
\begin{align}\label{eq:L_2_1} 
\|f_{\bu,a} - f_{\bv,c}\|_{P_m,2}
& \leq K\big(\|\bu-\bv\|_{\infty}^{1/2}\vee|c-a|^{\eta} \big) .
\end{align}
When $m^{-1}\leq c-a\leq m^{-1/2}$, select $\ta,\tc\in [a_{\wedge},a_{\vee}]$ such that $2 m^{-1/2} \geq \tc-\ta \geq m^{-1/2}, \tc,\ta \in \bbZ/m$ and $\ta\leq a\leq c\leq \tc$. By \eqref{eq:L_2all} and Lemma \ref{lemma:Lipschitz_1},
\begin{align}\label{eq:L_2_2}
\|f_{\bu,a} - f_{\bv,c}\|_{P_m,2}
&\leq 
\|\bbu_{\bu\vee\bv,a,c} - \bbl_{\bu\wedge\bv,a,c}\|_{P_m,2}  \nonumber\\
&\leq 
\|\bbu_{\bu\vee\bv,\ta,\tc} - \bbl_{\bu\wedge\bv,\ta,\tc}\|_{P_m,2}  \nonumber\\
&\leq K\big(\|\bu-\bv\|_{\infty}^{1/2}\vee|\tc-\ta|^{\eta} \big)  \nonumber\\ 
&\leq 2^\eta K\big(\|\bu-\bv\|_{\infty}^{1/2}\vee m^{-\eta/2} \big) 
\leq 2^\eta K\big(\|\bu-\bv\|_{\infty}^{1/2}\vee |c-a|^{\eta/2} \big) 
\end{align}
where the last inequality uses the fact that $|c-a| \geq 1/m$.
Finally, by Lemma~\ref{lemma:Lipschitz_2}, when $c-a=0$, 
\begin{equation}\label{eq:L_2_3}
\|f_{\bu,a} - f_{\bv,c}\|_{P_m,2} 
 =  \|\bbu_{\bv,a,a} - \bbl_{\bu,a,a}\|_{P_m,2} \leq K\|\bu-\bv\|_{\infty}^{1/2}. 
\end{equation}
A combination of \eqref{eq:L_2_1}, \eqref{eq:L_2_2}, and \eqref{eq:L_2_3} gives \eqref{eq:secmomWn} where the constant $K$ can be dropped at the cost of changing the power of $\|\bu-\bv\|\wedge |c-a|$
\hfill $\Box$

\subsubsection{Proof of Step (iii): fidi convergence}

We begin by stating and proving two technical results.

\begin{Lemma}\label{lemma:scale}
Suppose that Assumption~\ref{assump:mixing}(i) and (ii) and Assumption~\ref{assump:copula_convergence} are met. Further, let $\ck=\ck_n$ be a sequence of positive integers such that 
\begin{equation}\label{eqn:scale_assump}
\ck/m\to \Xi\in(0,1]. 
\end{equation}
Then, for all $\bm u \in [0,1]^d$ and as $n\to\infty$,
\begin{align*}
\text{(i)}\quad &\bF_{\ck}(\bF_{m}^{\leftarrow}(\bu))\to\bu^{\Xi},\\
\text{(ii)}\quad &F_{\ck}(\bF_{m}^{\leftarrow}(\bu))\to(C_{\infty}(\bu))^{\Xi},\\
\text{(iii)}\quad &(C_{\infty}(\bu))^{\Xi}=C_{\infty}(\bu^{\Xi}).
\end{align*}
\end{Lemma}

\begin{proof}[Proof of Lemma \ref{lemma:scale}]
Let $v_{m}$ be an arbitrary sequence in $[0,1]$, and $\bv_{m}$ be an arbitrary sequence in $[0,1]^{d}$. Lemma 2.1 of \cite{Lea83}, together with a straightforward extension to multivariate time series, shows that if there exists a sequence $\cl=\cl_{n}$ such that 
\begin{equation}\label{eqn:scale_assump_2}
\cl=o(m) \ \text{and} \ \alpha(\cl)\to 0,
\end{equation}
then we have
\begin{equation}\label{eqn:scale_lemma_0}
\big|\big(\Prob(M_{m/(m/\ck),0,j}\leq v_{m})\big)^{m/\ck}-\Prob(M_{m,0,j}\leq v_{m})\big|\to 0,
\end{equation}
and
\begin{equation}\label{eqn:scale_lemma}
\big|\big(\Prob(\bM_{m/(m/\ck),0}\leq \bv_{m})\big)^{m/\ck}-\Prob(\bM_{m,0}\leq\bv_{m})\big|\to 0.
\end{equation}
See also Lemma 4.1 of \cite{Hsi89}. Let $\cl=\cl_n=\ang{m^{2/3}}$. By \eqref{eqn:scale_assump} and Assumption \ref{assump:mixing}(i),(ii), \eqref{eqn:scale_assump_2} holds. 
Now plug in $v_{m}=F_{m,j}^{\leftarrow}(u_{j})$. 
By \eqref{eqn:scale_assump} and \eqref{eqn:scale_lemma_0},
\begin{align*}
\bF_{\ck}(\bF_{m}^{\leftarrow}(\bu))&=\Big\{\Prob(M_{\ck,0,1}\leq F^{\leftarrow}_{m,1}(u_{1}),\dots,\Prob(M_{\ck,0,d}\leq F^{\leftarrow}_{m,d}(u_{1}))\Big\}\\
&=\Big\{(\Prob(M_{m/(m/\ck),0,1}\leq F^{\leftarrow}_{m,1}(u_{1})))^{m/\ck},\dots,(\Prob(M_{m/(m/\ck),0,d}\leq F^{\leftarrow}_{m,d}(u_{1})))^{m/\ck}\Big\}^{\ck/m}\\
&=\Big\{\Prob(M_{m,0,1}\leq F^{\leftarrow}_{m,1}(u_{1})),\dots,\Prob(M_{m,0,d}\leq F^{\leftarrow}_{m,d}(u_{1}))\Big\}^{\ck/m}+o(1)\\
&=\Big\{\Prob(U_{m,0,1}\leq u_{1}),\dots,\Prob(U_{m,0,d}\leq u_{d})\Big\}^{\ck/m}+o(1)\\
&=\bu^{\Xi}+o(1).
\end{align*}
Then plug in $\bv_{m}=\bF_{m}^{\leftarrow}(\bu)$. 
By \eqref{eqn:scale_assump} and \eqref{eqn:scale_lemma},
\begin{align*}
\mathph{=}F_{\ck}(\bF_{m}^{\leftarrow}(\bu))
&=
\Prob(\bM_{\ck,0}\leq \bF_{m}^{\leftarrow}(\bu))=\Big\{\big(\Prob(\bM_{m/(m/\ck),0}\leq \bF_{m}^{\leftarrow}(\bu))\big)^{m/\ck}\Big \}^{\ck/m}\\
&=
\Big\{\Prob(\bM_{m,0}\leq \bF_{m}^{\leftarrow}(\bu))\Big\}^{\Xi} + o(1)
=
\Big\{\Prob(\bF_{m}(\bM_{m,0})\leq \bu)\Big\}^{\Xi} + o(1)\\
&=
\Big\{C_{\infty}(\bu)\Big\}^{\Xi} + o(1).
\end{align*}
Hence we have shown (i) and (ii) of Lemma \ref{lemma:scale}. Finally, by part (i) and Assumption \ref{assump:copula_convergence}, 
\begin{align} \label{eq:ffi}
F_{\ck}(\bF_{m}^{\leftarrow}(\bu)) 
&= 
\Prob(\bM_{\ck,0}\leq\bF_{m}^{\leftarrow}(\bu))
=
\Prob(\bF_{\ck}^{\leftarrow}(\bF_{\ck}(\bM_{\ck,0})) \leq\bF_{m}^{\leftarrow}(\bu))
\\
&= \nonumber
\Prob(\bU_{\ck,0}\leq \bF_{\ck}(\bF_{m}^{\leftarrow}(\bu)))=C_{\ck}(\bF_{\ck}(\bF_{m}^{\leftarrow}(\bu)))\\
&\to \nonumber
C_{\infty}(\bu^{\Xi}).
\end{align}
and, on the other hand, by part (ii),
\[
F_{\ck}(\bF_{m}^{\leftarrow}(\bu))\to(C_{\infty}(\bu))^{\Xi}.
\]
This implies (iii).
\end{proof}

\begin{Lemma}\label{lemma:cov}
Recall the definition of $\bbC_{n,m}^{\lozenge,b}$ in \eqref{eq:number_of_block}. If  Assumption~\ref{assump:copula_convergence} and Assumptions~\ref{assump:mixing}(i),(ii)  are met, then, for any $\bm u, \bm v \in [0,1]^d$ and any $a_{\wedge}\leq a\leq c\leq a_{\vee}$, we have
\[
\Cov(\bbC_{n,m}^{\lozenge,b}(\bu,a),\bbC_{n,m}^{\lozenge,b}(\bv,c))\to \gamma(\bv,\bu,c,a),
\]
where $\gamma(\bv,\bu,c,a)$ is defined in Theorem~\ref{theo:known}.
\end{Lemma}

\begin{proof}[Proof of Lemma \ref{lemma:cov}]
For $i\in\N$ let 
\[
D_i(\bm u, a) = \ind(\bU_{\ang{ma},i}\leq\bv)-\Prob(\bU_{\ang{ma},1}\leq \bv),
\]
such that
\begin{equation}\label{eqn:integral_decomposition}
    \Cov\big(\bbC_{n,m}^{\lozenge,b}(\bu,a),\bbC_{n,m}^{\lozenge,b}(\bv,c)\big))=A_{1}+A_{2}+A_{3},
\end{equation}
where  $A_j=A_j(\bm u,\bm v, a, c)$ is defined as
\begin{align*}
&A_{1}= {\frac{n}{mb^2}} \sum_{h=-\infty}^{\infty}\ind(-b<h<0)(b-|h|)\Cov(D_{0}(\bv,c),D_{h}(\bu,a)),\\
&A_{2}= {\frac{n}{mb^2}} \sum_{h=-\infty}^{\infty}\ind(0\leq h<\ang{mc}-\ang{ma})(b-|h|)\Cov(D_{0}(\bv,c),D_{h}(\bu,a)),\\
&A_{3}= {\frac{n}{mb^2}} \sum_{h=-\infty}^{\infty}\ind(\ang{mc}-\ang{ma}\leq h<b)(b-|h|)\Cov(D_{0}(\bv,c),D_{h}(\bu,a)).
\end{align*}
Now, for sufficiently large $n$,
\begin{align*}
A_{1}
&=
{\frac{n}{mb} } \sum_{h=-\infty}^{\infty}\ind(-\ang{ma}<h<0)(1-|h|b^{-1})\Cov(D_{0}(\bv,c),D_{h}(\bu,a))\\
&\mathph{=}  \qquad
+ {\frac{n}{mb} } \sum_{h=-\infty}^{\infty}\ind(-b<h\leq -\ang{ma})(1-|h|b^{-1})\Cov(D_{0}(\bv,c),D_{h}(\bu,a))\\
&=
 {\frac{n}{b} } \int_{-a}^{0}(1-|\ang{m\xi}|b^{-1})\Cov(D_{0}(\bv,c),D_{\ang{m\xi}}(\bu, a)) \diff \xi\\
&\mathph{=}\qquad
- {\frac{n}{mb} } \Cov(D_{0}(\bv,c),D_{0}(\bu, a)) \\
&\mathph{=}\qquad
-{\frac{n}{b} } \int_{-a}^{-\ang{ma}/m}(1-|\ang{m\xi}|b^{-1})\Cov(D_{0}(\bv,c),D_{\ang{m\xi}}(\bu, a)) \diff \xi \\
&\mathph{=}\qquad
+{\frac{n}{mb} } \sum_{h=-\infty}^{\infty}\ind(-b<h\leq-\ang{ma})(1-|h|b^{-1})\Cov(D_{0}(\bv,c),D_{h}(\bu,a)).
\end{align*}
All but the first term on the right hand side of the previous equation vanish. Indeed
\begin{align*}
& |m^{-1}\Cov(D_{0}(\bv,c),D_{0}(\bu, a))|\leq m^{-1} \to 0, \\
& \Big|\int_{-a}^{-\ang{ma}/m}(1-|\ang{m\xi}|b^{-1})\Cov(D_{0}(\bv,c),D_{\ang{m\xi}}(\bu, a)) \diff \xi\Big|\leq a-\ang{ma}/m\to0,
\end{align*}
and, by Assumption \ref{assump:mixing}(ii), 
,
\begin{multline*}
\mathph{\leq}\Big|m^{-1}\sum_{h=-\infty}^{\infty}\ind(-b< h\leq-\ang{ma})(1-|h|b^{-1})\Cov(D_{0}(\bv,c),D_{h}(\bu,a))\Big|\\
\leq m^{-1}\Big(1+\sum_{h>\ang{ma}}\alpha(h-\ang{ma})\Big)\leq m^{-1}\Big(1+\sum_{h=1}^{\infty}\alpha(h)\Big) \to 0.
\end{multline*}
As a result, 
\begin{equation}\label{eqn:integral_1}
A_{1}={(1+o(1))} \int_{-a}^{0}(1-|\ang{m\xi}|b^{-1})\Cov(D_{0}(\bv,c),D_{\ang{m\xi}}(\bu, a)) \diff \xi.
\end{equation}
Similarly, 
\begin{equation}\label{eqn:integral_2}
A_{2}={(1+o(1))} \int_{0}^{c-a}(1-|\ang{m\xi}|b^{-1})\Cov(D_{0}(\bv,c),D_{\ang{m\xi}}(\bu, a)) \diff \xi,
\end{equation}
and
\begin{equation}\label{eqn:integral_3}
A_{3}={(1+o(1))} \int_{c-a}^{c}(1-|\ang{m\xi}|b^{-1})\Cov(D_{0}(\bv,c),D_{\ang{m\xi}}(\bu, a)) \diff \xi.
\end{equation}
Suppose $\xi \in (-a,0)$. Notice 
\begin{equation}\label{eqn:Cov_D}
\Cov(D_{0}(\bv,c),D_{\ang{m\xi}}(\bu, a))=\Prob(\bU_{\ang{mc},0}\leq \bv, \bU_{\ang{ma},\ang{m\xi}}\leq \bu)-\Prob(\bU_{\ang{mc},0}\leq \bv)\Prob(\bU_{\ang{ma},\ang{m\xi}}\leq \bu).
\end{equation}
For $k\in\N$ and $i\in\Z$, let $\bM_{i:(i+k-1)}=\bM_{k,i}$. Since $\bU_{\ang{ma},i}=\bF_{\ang{ma}}(\bM_{\ang{ma},i}),$ we can write
\begin{equation}\label{eqn:joint_probability_1}
\begin{aligned}
&\mathph{=}\Prob(\bU_{\ang{mc},0}\leq \bv, \bU_{\ang{ma},\ang{m\xi}}\leq \bu)=\Prob(\bM_{0:(\ang{mc}-1)}\leq \bF_{\ang{mc}}^{\leftarrow}(\bv),\bM_{\ang{m\xi}:(\ang{m\xi}+\ang{ma}-1)}\leq \bF_{\ang{ma}}^{\leftarrow}(\bu))\\
&=\Prob(\cV_{1} \cap \cV_{2} \cap \cV_{3}),
\end{aligned} 
\end{equation}
where
\begin{align*}
\cV_{1}&=\{\bM_{\ang{m\xi}:(-1)}\leq\bF_{\ang{ma}}^{\leftarrow}(\bu)\}\\
\cV_{2}&=\{\bM_{0:(\ang{m\xi}+\ang{ma}-1)}\leq\bF_{\ang{mc}}^{\leftarrow}(\bv)\wedge\bF_{\ang{ma}}^{\leftarrow}(\bu)\}\\
\cV_{3}&=\{\bM_{(\ang{m\xi}+\ang{ma}):(\ang{mc}-1)}\leq\bF_{\ang{mc}}^{\leftarrow}(\bv)\}.  
\end{align*}
Now we seek to approximate $\Prob(\cV_{1} \cap \cV_{2} \cap \cV_{3})$ by $\Prob(\cV_{1})\Prob(\cV_{2})\Prob(\cV_{3})$ with a clipping technique; see also the proof of Lemma 5.1 of \cite{BucSeg18b}. Let $\cl=\cl_n=\ang{m^{2/3}}$ and define
\begin{align*}
\cW_{1}&=\{\bM_{\ang{m\xi}:(-\cl-1)}\leq\bF_{\ang{ma}}^{\leftarrow}(\bu)\}\\
\cW_{2}&=\cV_{2}\\
\cW_{3}&=\{\bM_{(\ang{m\xi}+\ang{ma}+\cl):(\ang{mc}-1)}\leq\bF_{\ang{mc}}^{\leftarrow}(\bv)\}
\end{align*}
be the clipped events; {note that $\cV_j \subset \cW_j$}. First we show that clipping `does not hurt'. Applying (ii) of Lemma \ref{lemma:scale} twice gives
\begin{equation}\label{eqn:joint_probability_2}
\begin{aligned}
\Prob(\cW_{1})-\Prob(\cV_{1})
&=\Prob(\bM_{|\ang{m\xi}|-\cl ,0}\leq\bF_{\ang{ma}}^{\leftarrow}(\bu))-\Prob(\bM_{|\ang{m\xi}| ,0}\leq\bF_{\ang{ma}}^{\leftarrow}(\bu))\\
&=F_{|\ang{m\xi}|-\cl}(\bF_{\ang{ma}}^{\leftarrow}(\bu))-F_{|\ang{m\xi}|}(\bF_{\ang{ma}}^{\leftarrow}(\bu))\\
&=(C_{\infty}(\bu))^{|\xi|/a}-(C_{\infty}(\bu))^{|\xi|/a}+o(1)=o(1).
\end{aligned}
\end{equation}
Similarly,
\begin{equation}\label{eqn:joint_probability_3}
\Prob(\cW_{3})-\Prob(\cV_{3})=o(1).
\end{equation}
Now we apply the clipping technique. First, by \eqref{eqn:joint_probability_2} and \eqref{eqn:joint_probability_3},
\begin{equation}\label{eqn:joint_probability_5}
\begin{aligned}
|\Prob(\cV_{1}\cap\cV_{2}\cap\cV_{3})-\Prob(\cW_{1}\cap\cW_{2}\cap\cW_{3})|
&=\Prob(\cW_{1}\cap\cW_{2}\cap\cW_{3}\cap(\cV_{1}\cap\cV_{2}\cap\cV_{3})^{c})\\&\leq \Prob (\cW_{1}\cap \cV_{1}^{c})+\Prob (\cW_{2}\cap \cV_{2}^{c})+\Prob (\cW_{3}\cap \cV_{3}^{c})\\
&=\Prob(\cW_{1})-\Prob(\cV_{1})+\Prob(\cW_{3})-\Prob(\cV_{3})=o(1).
\end{aligned}
\end{equation}
Second, by \eqref{eqn:joint_probability_2}, \eqref{eqn:joint_probability_3}, and since $\alpha(\cl)=o(1)$ by {Assumption \ref{assump:mixing}(ii}),
\begin{equation}\label{eqn:joint_probability_4}
\Prob(\cW_{1} \cap \cW_{2} \cap \cW_{3})=\Prob(\cW_{1})\Prob(\cW_{2})\Prob(\cW_{3})+o(1)=\Prob(\cV_{1})\Prob(\cV_{2})\Prob(\cV_{3})+o(1)
\end{equation}
Next, similarly as in \eqref{eq:ffi}, by (i) and (iii) of Lemma \ref{lemma:scale}, Assumption \ref{assump:copula_convergence} and continuity of $C_\infty$,
\begin{equation}\label{eqn:marginal_scale_2}
\begin{aligned}
\Prob(\cV_{2})&=\Prob\bigg(\bM_{\ang{m\xi}+\ang{ma},0}\leq\bF_{\ang{mc}}^{\leftarrow}(\bv)\wedge\bF_{\ang{ma}}^{\leftarrow}(\bu)\bigg)\\
&=C_{\ang{m\xi}+\ang{ma}}\bigg(\bF_{\ang{m\xi}+\ang{ma}}\big(\bF_{\ang{mc}}^{\leftarrow}(\bv)\wedge\bF_{\ang{ma}}^{\leftarrow}(\bu)\big)\bigg)\\
&=C_{\ang{m\xi}+\ang{ma}}\bigg(\bF_{\ang{m\xi}+\ang{ma}}\big(\bF_{\ang{mc}}^{\leftarrow}(\bv)\big)\wedge \bF_{\ang{m\xi}+\ang{ma}}\big(\bF_{\ang{ma}}^{\leftarrow}(\bu)\big)\bigg)\\
&\to C_{\infty}(\bv^{(\xi+a)/c}\wedge\bu^{(\xi+a)/a})=\bigg(C_{\infty}(\bv^{1/c}\wedge\bu^{1/a})\bigg)^{\xi+a}.
\end{aligned}
\end{equation}
Similarly,
\begin{equation}\label{eqn:marginal_scale_3}
\Prob(\cV_{1})\to (C_{\infty}(\bu^{1/a}))^{-\xi},\quad
\Prob(\cV_{3})\to (C_{\infty}(\bv^{1/c}))^{c-\xi-a}.
\end{equation}
By \eqref{eqn:joint_probability_1}, \eqref{eqn:joint_probability_5},  \eqref{eqn:joint_probability_4}, \eqref{eqn:marginal_scale_2}, and \eqref{eqn:marginal_scale_3}, for $\xi \in (-a,0)$,
\begin{equation}\label{eqn:cov_u_1}
\Prob(\bU_{\ang{mc},0}\leq \bv, \bU_{\ang{ma},\ang{m\xi}}\leq \bu)\to (C_{\infty}(\bu^{1/a}))^{-\xi}(C_{\infty}(\bv^{1/c}\wedge\bu^{1/a}))^{\xi+a}(C_{\infty}(\bv^{1/c}))^{c-\xi-a}.
\end{equation}
Similarly, for $\xi \in (0,c-a)$,
\begin{equation}\label{eqn:cov_u_2}
\begin{aligned}
\Prob(\bU_{\ang{mc},0}\leq \bv, \bU_{\ang{ma},\ang{m\xi}}\leq \bu)&=
\Prob(\bM_{0:(\ang{mc}-1)}\leq \bF_{\ang{mc}}^{\leftarrow}(\bv),\bM_{\ang{m\xi}:(\ang{m\xi}+\ang{ma}-1)}\leq \bF_{\ang{ma}}^{\leftarrow}(\bu)) \\
&=
\Prob(\bM_{0:(\ang{m\xi}-1)}\leq\bF_{\ang{mc}}^{\leftarrow}(\bv))\\
&\mathph{=}\quad \times\Prob(\bM_{\ang{m\xi}:(\ang{m\xi}+\ang{ma}-1)}\leq\bF_{\ang{mc}}^{\leftarrow}(\bv)\wedge\bF_{\ang{ma}}^{\leftarrow}(\bu))\\
&\mathph{=}\quad \times
\Prob(\bM_{(\ang{m\xi}+\ang{ma}):(\ang{mc}-1)}\leq\bF_{\ang{mc}}^{\leftarrow}(\bv))+o(1)\\
&=(C_{\infty}(\bv^{1/c}))^{\xi}(C_{\infty}(\bv^{1/c}\wedge\bu^{1/a}))^{a}(C_{\infty}(\bv^{1/c}))^{c-\xi-a}+o(1)\\
&=(C_{\infty}(\bv^{1/c}))^{c-a}(C_{\infty}(\bv^{1/c}\wedge\bu^{1/a}))^{a}+o(1),
\end{aligned}
\end{equation}
and for $\xi \in (c-a,c)$,
\begin{equation}\label{eqn:cov_u_3}
\begin{aligned}
\Prob(\bU_{\ang{mc},0}\leq \bv, \bU_{\ang{ma},\ang{m\xi}}\leq \bu)
&= \Prob(\bM_{0:(\ang{mc}-1)}\leq \bF_{\ang{mc}}^{\leftarrow}(\bv),\bM_{\ang{m\xi}:(\ang{m\xi}+\ang{ma}-1)}\leq \bF_{\ang{ma}}^{\leftarrow}(\bu)) \\
&=
\Prob(\bM_{0:(\ang{m\xi}-1)}\leq\bF_{\ang{mc}}^{\leftarrow}(\bv))\\
&\mathph{=}\quad \times\Prob(\bM_{\ang{m\xi}:(\ang{mc}-1)}\leq\bF_{\ang{mc}}^{\leftarrow}(\bv)\wedge\bF_{\ang{ma}}^{\leftarrow}(\bu))\\
&\mathph{=}\quad \times
\Prob(\bM_{\ang{mc}:(\ang{m\xi}+\ang{ma}-1)}\leq\bF_{\ang{ma}}^{\leftarrow}(\bu))+o(1)\\
&=(C_{\infty}(\bv^{1/c}))^{\xi}(C_{\infty}(\bv^{1/c}\wedge\bu^{1/a}))^{c-\xi}(C_{\infty}(\bu^{1/a}))^{\xi+a-c}+o(1).
\end{aligned}
\end{equation}
By Assumption \ref{assump:copula_convergence} and the stationarity of $(\bX_{t})_t$, 
\begin{equation}\label{eqn:cov_u_4}
\Prob(\bU_{\ang{mc},0}\leq\bv)\to C_{\infty}(\bv), \ \ \Prob(\bU_{\ang{ma},\ang{m\xi}}\leq\bu)\to C_{\infty}(\bu).
\end{equation}
Recall \eqref{eqn:integral_decomposition}. Now apply the Dominated Convergence Theorem to the right hand side of \eqref{eqn:integral_1}, \eqref{eqn:integral_2}, and \eqref{eqn:integral_3}. By \eqref{eqn:Cov_D}, \eqref{eqn:cov_u_1}, \eqref{eqn:cov_u_2},
\eqref{eqn:cov_u_3}, 
\eqref{eqn:cov_u_4}, and (iii) of Lemma \ref{lemma:scale}, we get 
\begin{align*}
    \Cov\big(\bbC_{n,m}^{\lozenge,b}(\bu,a),\bbC_{n,m}^{\lozenge,b}(\bv,c)\big))
&\to \gamma(\bv,\bu,c,a)
\end{align*}
with $\gamma(\bv,\bu,c,a)$ is defined in Theorem~\ref{theo:known}.
\end{proof}

\begin{Lemma}\label{lemma:fidi_convergence}
Under Assumption \ref{assump:copula_convergence} and \ref{assump:mixing} , the finite dimensional distributions of $\bbC_{n,m}^{\lozenge,b}$ converge weakly to the finite dimensional distributions of $\bbC^{\lozenge}$.
\end{Lemma}
\begin{proof}[Proof of Lemma \ref{lemma:fidi_convergence}]
We apply a big-blocks-small-blocks technique. Assume without loss of generality that $a_{\vee}$ is an integer (otherwise enlarge $a_{\vee}$, this does not change any of the arguments). For $\lambda>a_{\vee}+1$ chosen below, the block size of the big blocks {will be proportional to $(\lambda-a_{\vee}-1)m$, while the small blocks will have size proportional to $(a_{\vee}+1)m$}. By Assumption~\ref{assump:mixing}(i) and (iv), we have $b/m\to \infty$ and $\alpha(m)(b/m)=o\big((b/m)^{1/2-\zeta}\big)$ for some $\zeta\in (0,1/2)$. Hence, there exists $\lambda=\lambda_n$, which can be chosen to be integer-valued, such that
\begin{equation}\label{eqn:lambda_1}
\lambda=o\big((b/m)^{1/2-\zeta}\big)
\end{equation}
\begin{equation}\label{eqn:lambda_0}
\lambda\to \infty,
\end{equation}
\begin{equation}\label{eqn:lambda_2}
\alpha(m)(b/m)=o(\lambda).
\end{equation} 
Let $K=b/(\lambda m)$ and for simplicity assume that $K$ is an integer. For $k=1,\dots,K,$ consider big blocks
$$I_{k}=\{(k-1)\lambda m+1,\dots,(k\lambda-a_{\vee}-1)m\},$$
and small blocks
\[
J_{k}
=\{(k\lambda-a_{\vee}-1)m+1,\dots,k\lambda m\}.
\]
For $a \in [a_{\wedge},a_{\vee}]$, let 
\[
Y_{n,m}(\bu,a)={\sqrt{\frac{n}{mb^2}}} \sumb D_i(\bm u, a) (\bu)\ind (i\in \cup_{k=1}^{K}J_{k}).
\]
We may then write
\[
\bbC_{n,m}^{\lozenge,b}(\bu,a)
=
{\sqrt{\frac{n}{mb^2}}} \sumb  D_i(\bm u, a) 
=
{\sqrt{\frac{n}{mb^2}}} \sumb  D_i(\bm u, a) \ind (i\in \cup_{k=1}^{K}I_{k})+Y_{n,m}(\bu,a).
\]
We will now show that $Y_{n,m}$ is negligible. Indeed, let
\[
Y_{n,m,k}(\bu,a)
=
{\sqrt{\frac{n}{mb^2}}} \sumb D_i(\bm u, a)  \ind(i\in J_{k}),
\]
then, {by stationarity},
\begin{align*}
\Var(Y_{n,m}) 
&= 
\Var\Big(\sum_{k=1}^{K}Y_{n,m,k}\Big)=\sum_{|h|<K} (K-|h|)\Cov(Y_{n,m,1},Y_{n,m,|h|+1})\\
&\leq 
3K\Var(Y_{n,m,1})+\sum_{2\leq|h|<K}(K-|h|)\Cov(Y_{n,m,1},Y_{n,m,|h|+1}).
\end{align*}
Notice that {$|Y_{n,m,k}(\bu,a)|\leq (a_{\vee}+1)(mn/b^2)^{1/2}$}. Hence, since $\lambda=\lambda_n \to \infty$ by \eqref{eqn:lambda_0},
\[
{K \Var(Y_{n,m,1})\leq (a_{\vee}+1)^{2}Kmnb^{-2}=(a_{\vee}+1)^{2} \lambda^{-1}(1+o(1)) = o(1). }
\]
Moreover, for $|h|\geq 2$, by Lemma 3.9 in \cite{Deh02}, 
\begin{align*}
\Big|\Cov(Y_{n,m,1},Y_{n,m,|h|+1})\Big|
&\leq   
{4\Big((a_{\vee}+1)(mnb^{-2})^{1/2}\Big)^{2}} \alpha\Big(\sigma(Y_{n,m,1}),\sigma(Y_{n,m,|h|+1})\Big),\\
&{\leq 4(a_{\vee}+1)^{2}mnb^{-2}} \alpha((|h|-1)\lambda m).
\end{align*}
Hence, since $\lambda=\lambda_n \to \infty$ by \eqref{eqn:lambda_0}  and  by Assumption \ref{assump:mixing}~(ii),
\begin{align*}
\mathph{=}\sum_{2\leq|h|<K} (K-|h|)\Cov(Y_{n,m,1},Y_{n,m,|h|+1})
&\leq 4(a_{\vee}+1)^{2}{Kmnb^{-2}}\sum_{2\leq|h|<K} \alpha((|h|-1)\lambda m)\\
&=4(a_{\vee}+1)^{2} \lambda^{-1}(1+o(1))\sum_{2\leq|h|<K} \alpha((|h|-1)\lambda m)=o(1).
\end{align*}
Therefore,  
\begin{equation}\label{eqn:small_block}
\lim_{n\to\infty} \Var(Y_{n,m})= 0,
\end{equation}
and since $\Exp[Y_{n,m}]=0$, we obtain that $Y_{n,m}=\op{1}$ as asserted.

Next, we show finite-dimensional convergence of $\bbC_{n,m}^{\lozenge,b}$, i.e., for all $\bu_{1},\dots,\bu_{q}$ in $[0,1]^{d}$ and all $a_{1},\dots,a_{q}$ in $[a_{\wedge},a_{\vee}]$, 
\[
\{\bbC_{n,m}^{\lozenge,b}(\bu_{1},a_{1}),\dots,\bbC_{n,m}^{\lozenge,b}(\bu_{q},a_{q})\}
\dto 
\{\bbC^{\lozenge}(\bu_{1},a_{1}),\dots,\bbC^{\lozenge}(\bu_{q},a_{q})\}.
\]
By the Cram\'er-Wold device, it suffices to show that for all $(\theta_{1},\dots,\theta_{q}) \in \bbR^q$,
\[
 \sum_{j=1}^{q}\theta_{j}\bbC_{n,m}^{\lozenge,b}(\bu_{j},a_{j})\dto \sum_{j=1}^{q}\theta_{j}\bbC^{\lozenge,b}(\bu_{j},a_{j}).
\]
Recall
\begin{equation}\label{eqn:finite_dimension_1}
\sum_{j=1}^{q}\theta_{j}\bbC_{n,m}^{\lozenge,b}(\bu_{j},a_{j})
=
\sum_{j=1}^{q}\theta_{j} {\sqrt{\frac{n}{mb^2}}} \sumb D_{i}(\bu_{j}, a_j). 
\end{equation}
Let 
\begin{align*}
Z_{n}
&=
\sum_{j=1}^{q}\theta_{j}{\sqrt{\frac{n}{mb^2}}} \sumb D_{i}(\bu_{j}, a_j)\ind (i\in \cup_{k=1}^{K}I_{k}), \qquad
Z
=
\sum_{j=1}^{q}\theta_{j}\bbC^{\lozenge}(\bu_{j},a_{j}).
\end{align*}
By \eqref{eqn:small_block} and \eqref{eqn:finite_dimension_1}, to show the finite dimensional convergence, it suffices to show  that
\begin{equation}\label{eqn:finite_dimension_2}
Z_{n}\dto Z.
\end{equation}
Let 
\[
Z_{n,k}
=\sum_{j=1}^{q}\theta_{j}{\sqrt{\frac{n}{mb^2}}}\sumb D_{i}(\bu_{j}, a_j)\ind (i\in I_{k})
\]
and note that 
$
Z_{n}=\sum_{k=1}^{K}Z_{n,k}.
$
In addition, for $k\neq k'$, $Z_{n,k}$ and $Z_{n,k'}$ are based on observations that are at least $m$ observations apart. As a consequence, $Z_{n,k}$ and $Z_{n,k'}$ are asymptotically  independent. Next, let $\psi_{n}(\cdot)$ and $\psi_{n,k}(\cdot)$ denote the characteristic functions of $Z_{n}$ and $Z_{n,k}$, respectively. By the reasoning in p.\ 515 of \cite{BucSeg14} and \eqref{eqn:lambda_2}, for any fixed $t \in \bbR$,
\begin{equation}\label{eqn:char}
\Big|\psi_{n}(t)-\prod_{k=1}^{K}\psi_{n,k}(t)\Big| \leq K\alpha(m)=b/(\lambda m) \times \alpha(m) =o(1). 
\end{equation} 
Let $ \{\tZ_{n,k}\}_{k=1}^{K}$ denote row-wise independent random variables with $\tZ_{n,k}$ having the same distribution as $Z_{n,k}$ for $k=1,\dots,K$. Then $\prod_{k=1}^{K}\psi_{n,k}(t)$ is the characteristic function of $\sum_{k=1}^K \tZ_{n,k}$. If we can prove that $\tilde Z_n = \sum_{k=1}^K \tZ_{n,k} \dto Z$, then $\prod_{k=1}^{K}\psi_{n,k}(t)$ will converge the the characteristic function of $Z$, and then \eqref{eqn:char} will imply \eqref{eqn:finite_dimension_2}. 

Now we apply the Lyapunov Central Limit Theorem to $\{\tZ_{n,k}\}$. {By Lemma~\ref{lemma:cov}, applied with $n$ and $m$ replaced by $n' = (\lambda-a_{\vee})m-1$ and $m'=m$, respectively (note that $n'/m' \to \infty$ since $\lambda \to \infty$ and that $b' = n'-m'+1=(\lambda-a_\vee-1)m$) is the length of a big block)}, we obtain that
\begin{equation}\label{eqn:Lyapunov}
\begin{aligned}
s_n^2 = \Var(\tZ_n) &= \sum_{k=1}^{K}\Var(\tZ_{n,k}) =\sum_{k=1}^{K}\Var(Z_{n,k})=K\Var(Z_{n,1})
\\
&={K \frac{n}{mb^2} \frac{m'b'^2}{n'}} \Var\bigg(\sum_{j=1}^{q}\theta_{j}\bbC_{(\lambda-a_{\vee})m-1,m}^{\lozenge,b}(\bu_{j},a_{j})\bigg)
\\
&={(1+o(1))} \Var\bigg(\sum_{j=1}^{q}\theta_{j}\bbC_{(\lambda-a_{\vee})m-1,m}^{\lozenge,b}(\bu_{j},a_{j})\bigg) \to \Var(Z).
\end{aligned}    
\end{equation}
Now, if $\Var(Z)=0$, then $\tZ_{n}\dto 0=Z$ and we are left with the case $s_{n}^{-1} = O(1)$. Let $\delta=(1-2\zeta)/2\zeta$ with $\zeta \in (0,1/2)$ from Assumption~\ref{assump:mixing} (iv). By \eqref{eqn:lambda_1}, $\lambda^{1+\delta}(b/m)^{-\delta/2}\to 0$. By stationarity,
\begin{align*}
s_{n}^{-2-\delta}\sum_{k=1}^{K}\Exp\{|\tZ_{n,k}|^{2+\delta}\}
&=
s_{n}^{-2-\delta}K\Exp\{|\tZ_{n,1}|^{2+\delta}\}
\\
&\leq
s_{n}^{-2-\delta} K { \big\{ q \max_{j=1,\dots,q}|\theta_{j}| \sqrt{n/(mb^2)} \lambda m \big\}^{2+\delta}}
\\
&= 
s_{n}^{-2-\delta} q^{2+\delta} \max_{j=1,\dots,q}|\theta_{j}|^{2+\delta}
\lambda^{1+\delta}(b/m)^{-\delta/2} (n/b)^{1+\delta/2}\to 0.
\end{align*}
By the Lyapunov Central Limit Theorem, 
$$s_{n}^{-1}\sum_{k=1}^{K}\tZ_{n,k}\dto N(0,1).$$
Since $s_{n}^{2} \to \Var(Z)$ by \eqref{eqn:Lyapunov},  we obtain \eqref{eqn:finite_dimension_2} and the proof is finished.
\end{proof}

%\newpage

\subsubsection{Further technical Lemmas for the proof of Theorem~\ref{theo:known}}

Throughout this section assume that the conditions of Theorem~\ref{theo:known} hold. Recall that the functions $\bbl,\bbu$ are defined in~\eqref{eq:defl}, \eqref{eq:defu}.

\begin{Lemma}\label{lemma:Lipschitz_1}
{Let $A'=[a_\wedge', a_\vee']$ denote an arbitrary closed interval in $(0,\infty)$.} 
There exist $m_0, K < \infty $ and $\eta > 0$, depending on the mixing coefficients $\alpha(\cdot)$,  $\varrho$, the interval $A'$ and the dimension $d$ only, such that, for  all $m\ge m_0$, all $a,c\in A'$ with $ma,mc \in \bbN, 1 \geq |c-a| \geq m^{-1/2}$ and all $\bu,\bv\in [0,1]^d$, we have
\[
\Big\|\bbu_{\bv,a,c}-\bbl_{\bu,a,c}\Big\|_{P_m,2} \leq K \Big(\|\bu-\bv\|_{\infty}^{1/2}\vee|c-a|^{\eta}\Big).
\]
\end{Lemma}

\begin{proof}[Proof of Lemma \ref{lemma:Lipschitz_1}]
{Write $a_\wedge = a_\wedge', a_\vee = a_\vee'$ and $A=A'$.} 
It is sufficient to consider the case $a\le c$. Recall that  $\cK= b/(2\ell_m)$ and $\ell_m = (a_{\vee}+1)m$. Then, since $(a,c) \mapsto \ind(\tbM_{ma,i}\leq \bF_{mc}^{\leftarrow}(\bu))$ is decreasing in $a$ and increasing in $c$ {and since $b \ge n/2$ for sufficiently large $n$}, we have \begin{equation}\label{eq:Lipschitz}
\begin{aligned}
\Big\|\bbu_{\bv,a,c}-\bbl_{\bu,a,c}\Big\|_{P_m,2}^2
&= { \frac{n\cK}{b^2m}} \Exp\Bigg[\bigg\{
\sum_{i=1}^{\ell_{m}}\ind\Big((\uutbM_{i}^{(1)})_{1+ma-\ang{ma_{\wedge}},\cdot} \leq \bF_{mc}^{\leftarrow}(\bv)\Big) \\
& \hspace{3cm}-  \ind\Big((\uutbM_{i}^{(1)})_{1+mc-\ang{ma_{\wedge}},\cdot} \leq \bF_{ma}^{\leftarrow}(\bu)\Big)
\bigg\}^{2}\Bigg]
\\
&=\frac{n\cK}{b^2m} \Exp\Bigg[\bigg(\sum_{i=1}^{\ell_{m}}\ind(\tbM_{ma,i}\leq \bF_{mc}^{\leftarrow}(\bv)) -\ind(\tbM_{mc,i}\leq \bF_{ma}^{\leftarrow}(\bu)) \bigg)^{2}\Bigg]
\\
&\leq \frac{n}{2bm}\Exp\Bigg[\sum_{i=1}^{\ell_{m}} \ind( \tbM_{ma,i}\leq \bF_{mc}^{\leftarrow}(\bv)) -\ind(\tbM_{mc,i}\leq \bF_{ma}^{\leftarrow}(\bu))\Bigg]
\\
&\le (a_{\vee}+1) \Big(\Prob(\tbM_{ma,1}\leq \bF_{mc}^{\leftarrow}(\bv)) -\Prob(\tbM_{mc,1}\leq \bF_{ma}^{\leftarrow}(\bu))\Big)
\\
&=(a_{\vee}+1)(H_{1}+H_{2}+H_{3}),
\end{aligned}
\end{equation}
where 
\begin{align*}
    H_{1}&=\Prob(\tbM_{ma,1}\leq\bF_{mc}^{\leftarrow}(\bv))-\Prob(\tbM_{ma,1}\leq\bF_{ma}^{\leftarrow}(\bv)),\\
    H_{2}&=\Prob(\tbM_{ma,1}\leq\bF_{ma}^{\leftarrow}(\bv))-\Prob(\tbM_{ma,1}\leq \bF_{ma}^{\leftarrow}(\bu)),\\
    H_{3}&=\Prob(\tbM_{ma,1}\leq\bF_{ma}^{\leftarrow}(\bu))-\Prob(\tbM_{mc,1}\leq\bF_{ma}^{\leftarrow}(\bu)).
\end{align*}
For $H_{1}$ note that, by~\eqref{eq:coupling}~(i),
\begin{equation}\label{eq:Lipschitz_1_preliminary_1}
\begin{aligned}
H_{1}&=\Prob(\bM_{ma,i}\leq\bF_{mc}^{\leftarrow}(\bv), \bM_{ma,i}\nleq \bF_{ma}^{\leftarrow}(\bv))\\
&\le \sum_{j=1}^{d}[\Prob(M_{ma,1,j}\leq F_{mc,j}^{\leftarrow}(v_{j}))-\Prob(M_{ma,1,j} \leq F_{ma,j}^{\leftarrow}(v_{j}))]\\
&= \sum_{j=1}^{d}[\Prob(M_{ma,1,j}\leq F_{mc,j}^{\leftarrow}(v_{j}))-v_{j}]\\
&= \sum_{j=1}^{d}[\Prob(M_{ma,1,j}\leq F_{mc,j}^{\leftarrow}(v_{j}))-\Prob(M_{mc,1,j} \leq F_{mc,j}^{\leftarrow}(v_{j}))]\\
&= \sum_{j=1}^{d}[\Prob(M_{ma,1,j}\leq F_{mc,j}^{\leftarrow}(v_{j}), M_{mc,1,j}>F_{mc,j}^{\leftarrow}(v_{j}))]
\end{aligned}    
\end{equation}
In what follows, define $M_{i:i+k-1,j} := M_{k,i,j}$. Then
\begin{equation}\label{eq:Lipschitz_1_preliminary_2} 
\begin{aligned}
&\mathph{\leq}\Prob\bigg(M_{ma,1,j}\leq F_{mc,j}^{\leftarrow}(v_{j}), M_{mc,1,j}> F_{mc,j}^{\leftarrow}(v_{j})\bigg)\\
&=\Prob\bigg(M_{0:(ma-1),j}\leq F_{mc,j}^{\leftarrow}(v_{j}), \max(M_{0:(ma-1),j},M_{ma:(mc-1),j})> F_{mc,j}^{\leftarrow}(v_{j})\bigg)\\
&=\Prob\bigg(M_{0:(ma-1),j}\leq F_{mc,j}^{\leftarrow}(v_{j}),M_{ma:(mc-1),j}> F_{mc,j}^{\leftarrow}(v_{j})\bigg)
\end{aligned}
\end{equation}
Recall $\varrho>0$ defined in Assumption \ref{assump:mixing} (ii) and note that the assumption continues to hold with $\varrho\wedge(1/2)$ instead of $\varrho$. Throughout the remaining proof we can thus  assume without loss of generality that $\varrho \in (0,1/2)$. For each $j=1,\dots,d$, consider the three cases
\begin{align}
&v_{j}\in \Big((c-a)^{\varrho/2},1\Big]& & \text{and} &  &c-a\in  \Big(a_{\wedge}/8,a_{\vee}-a_{\wedge}\Big] \cap[0,1],\label{eq:very_large_c-a}\\
&v_{j}\in \Big((c-a)^{\varrho/2},1\Big]& & \text{and} & &c-a\in\Big[ m^{-1/2},a_{\wedge}/8\Big) \cap[0,1] , \label{eq:large_c-a}\\
&v_{j}\in \Big[0, (c-a)^{\varrho/2}\Big] . \label{eq:small_v}
\end{align}
Given \eqref{eq:very_large_c-a} or \eqref{eq:large_c-a},
\begin{equation}\label{eq:large_v}
\begin{aligned}
\mathph{\leq}\Prob\bigg(M_{0:(ma-1),j}\leq F_{mc,j}^{\leftarrow}(v_{j}),M_{ma:(mc-1),j}> F_{mc,j}^{\leftarrow}(v_{j})\bigg)& \leq \Prob\bigg(M_{ma:(mc-1),j}> F_{mc,j}^{\leftarrow}(v_{j})\bigg)\\
&=\Prob\bigg(M_{mc-ma,1,j}> F_{mc,j}^{\leftarrow}(v_{j})\bigg).    
\end{aligned}
\end{equation}
Now bound the right hand side of \eqref{eq:large_v}. First, focus on \eqref{eq:very_large_c-a}. Then 
\begin{equation}\label{eq:result_very_large}
\Prob\bigg(M_{mc-ma,1,j}> F_{mc,j}^{\leftarrow}(v_{j})\bigg)\leq 1 \leq (8/a_{\wedge})(c-a).
\end{equation}
Second, assume \eqref{eq:large_c-a}. Let $\cK_{1}=\ang{c/\{2(c-a)\}}$. Then
\begin{equation}\label{eq:blocking_large_c-a_large_u}
\begin{aligned}
F_{mc,j}(x)&=\Prob(\max_{t=1,\dots,mc}X_{t,j}\leq x) \\
&\leq\Prob(\max_{t=2(k-1)(mc-ma)+1,\dots,(2k-1)(mc-ma)}X_{t,j}\leq x,\ k=1,2,\dots,\cK_{1})\\
&\leq\prod_{k=1}^{\cK_{1}}\Prob(\max_{t=2(k-1)(mc-ma)+1,\dots,(2k-1)(mc-ma)}X_{t,j}\leq x)+\cK_{1}\alpha(m(c-a))\\
&\leq \Big(F_{mc-ma,j}(x)\Big)^{\cK_{1}}+\frac{a_{\vee}}{2(c-a)}\alpha(m(c-a)).
\end{aligned}
\end{equation}
By \eqref{eq:large_c-a} and Assumption \ref{assump:mixing} we have, for all sufficiently large $m$, say $m\geq m_{0,1}(\alpha(\cdot), \varrho, A, d)$, 
\begin{equation}\label{eq:error_large}
\begin{aligned}
\frac{a_{\vee}}{2(c-a)}\alpha(m(c-a))
&\leq \frac{a_{\vee}}{2(c-a)} \alpha(\ang{m^{1/2}})\leq (8/a_{\wedge})^{\varrho/2} \frac{1}{2(c-a)}m^{-(1+\varrho)/2}\\
&\leq (8/a_{\wedge})^{\varrho/2}\frac{1}{2}(c-a)^{\varrho}\leq \frac{1}{2}(c-a)^{\varrho/2}\leq \frac{v_{j}}{2},
\end{aligned}
\end{equation}
where the second inequality follows from $\alpha(\ang{m^{1/2}}) = o(m^{-(1+\varrho)/2})$ by Assumption~\ref{assump:mixing}(ii), which makes possible the values of the constants in this inequality.

By \eqref{eq:large_c-a}, \eqref{eq:blocking_large_c-a_large_u}, and \eqref{eq:error_large}, when $m \geq m_{0,1}(\alpha(\cdot), \varrho, A, d)$
\begin{equation}\label{eq:result_large}
\begin{aligned}
\Prob\bigg(M_{mc-ma,1,j}> F_{mc,j}^{\leftarrow}(v_{j})\bigg)&=\Prob\bigg(F_{mc,j}\big(M_{mc-ma,1,j}\big)> v_{j}\bigg)\\
&\leq \Prob\bigg(F_{mc-ma,j}\big(M_{mc-ma,1,j}\big)>(v_{j}/2)^{1/\cK_{1}}\bigg)\\
&=1-(v_{j}/2)^{1/\cK_{1}} \\
&\leq (1/\cK_{1})(-\log (v_{j}/2))\\
&\leq (4/a_{\wedge})(-\log (v_{j}/2))(c-a)\\
&\leq (4/a_{\wedge})(-\log ( (c-a)^{\varrho/2} /2))(c-a)\\
&\leq (4/a_{\wedge})(\varrho/2+\log 2)(c-a)^{1/2},
\end{aligned}
\end{equation} 
{where we used that $-x \log(x) \le x^{1/2}$ for $x>0$ in the last inequality.} Third, assume \eqref{eq:small_v}. Then, {the right-hand side of  \eqref{eq:Lipschitz_1_preliminary_2} can be bounded as follows:}
\begin{align} \label{eq:small_v_m}
\mathph{\leq}\Prob\Big(M_{0:(ma-1),j}\leq F_{mc,j}^{\leftarrow}(v_{j}),M_{ma:(mc-1),j}> F_{mc,j}^{\leftarrow}(v_{j})\Big)
&\leq \Prob(M_{0:(ma-1),j}\leq F_{mc,j}^{\leftarrow}(v_{j})) \nonumber \\
&=
\Prob(M_{ma,1,j}\leq F_{mc,j}^{\leftarrow}(v_{j})) \nonumber \\
&=
\Prob(F_{mc,j} (M_{ma,1,j}) \leq v_j)   \nonumber\\
&\le
\Prob(F_{mc,j} (M_{ma,1,j}) \leq (c-a)^{\varrho/2})  
\end{align}
Since $1 \ge c-a\geq m^{-1/2}$ and $0<\varrho<1$ we see that for
\[
\cK_{2} := \Big\langle \frac{1}{2} \min\Big\{m^{(1+\varrho)/(2+\varrho)}(c-a)^{1/4}, m(c-a)^{\varrho/(2+2\varrho)}\Big\}\Big\rangle
\] 
there exists $m_{0,2}(\alpha(\cdot), \varrho, A) < \infty$ such that for all $m \geq m_{0,2}$ and all $c-a\geq m^{-1/2}$
\begin{equation}\label{eq:cK_2}
\frac{1}{2}\ceil{a_{\vee}/a_{\wedge}} | \log(c-a) | (c-a)^{-\varrho/2}<\cK_{2}<\min\{m^{(1+\varrho)/(2+\varrho)}(c-a)^{1/4}, m(c-a)^{\varrho/(2+2\varrho)}\}.     
\end{equation}
{To see that the first inequality holds for sufficiently large $m$,  one may use that  $c-a \geq m^{-1/2}$ and the fact that $\varrho \in (0,1)$ to obtain the bound
\[
m^{(1+\varrho)/(2+\varrho)}(c-a)^{1/4} \geq m^{(1+\varrho)/(2+\varrho) -1/8} \geq m^{(1+\varrho) / 3  -1/8} \geq m^{\varrho/4+7/24}  \ge (c-a)^{-\varrho/2} m^{7/24}. 
\]
The lower bound for the first expression in the definition of $\cK_2$ then follows from $|\log(c-a)| \leq (\log m)/2$, which may be upper bounded by any constant multiple of $m^{7/24}$ for sufficiently large $m$.}
A similar argument can be used to bound $m(c-a)^{\varrho/(2+2\varrho)}$ from below for sufficiently large $m$ since
\[
m(c-a)^{\varrho/(2+2\varrho)} \geq m^{1-\varrho/(4+4\varrho)} \geq m^{1 - \varrho/4} \geq m^{3/4} \geq m^{\varrho/4} m^{1/2}  \ge (c-a)^{-\varrho/2} m^{1/2}. 
\]

Note that by construction $\inf_{|c-a|\geq m^{-1/2}} \cK_{2} \to \infty, m/\sup_{|c-a|\geq m^{-1/2}}\cK_{2}\to\infty $ as $m \to \infty$, so that we can assume (at the cost of potentially changing constants) that $\cK_{2}$ and $ma/\cK_{2}$ are integers. Observe that for any $l\in\bbN$ with $l \geq c/a$ {and any $x\in\R$}
\begin{equation*}%\label{eq:blocking_large_c-a_small_u_0}
\begin{aligned}
F_{mc,j}(x) \geq F_{mla,j}(x) & = \Prob\Big(\forall k=1,\dots, l : \max_{t=(k-1)ma+1,\dots,kma}X_{t,j}\leq x\Big)
\\
&= \Prob\Big(\forall k=1,\dots, l: \max_{t=(k-1)ma+1,\dots,kma-ma/\cK_{2}}X_{t,j}\leq x\Big)
\\
&\hspace{1cm}-\Prob\Big(\exists k=1,\dots, l: \max_{t=kma-ma/\cK_{2}+1,\dots,kma}X_{t,j}> x\Big)
\\
&\geq \Big(F_{ma-ma/\cK_{2},j}(x)\Big)^{l}-l\alpha(ma/\cK_{2})-l(1-F_{ma/\cK_{2},j}(x))
\\
&\geq \Big(F_{ma,j}(x)\Big)^{l}-l\alpha(ma/\cK_{2})-l(1-F_{ma/\cK_{2},j}(x)).
\end{aligned}
\end{equation*}
Hence, {for the particular choice $l = \lceil c/a \rceil$}
\begin{equation}\label{eq:blocking_large_c-a_small_u_1}
F_{mc,j}(M_{ma,1,j})\geq \Big(F_{ma,j}(M_{ma,1,j})\Big)^{\lceil c/a \rceil}- \lceil c/a \rceil \alpha(ma/\cK_{2})- \lceil c/a \rceil \Big(1-F_{ma/\cK_{2},j}(M_{ma,1,j})\Big). 
\end{equation}
Now we bound the second and the third term on the right hand side of \eqref{eq:blocking_large_c-a_small_u_1}. By \eqref{eq:cK_2} and Assumption \ref{assump:mixing}(ii) we have, for $m \geq m_{0,3}(\alpha(\cdot),A,\varrho)$,\begin{equation}\label{eq:error_large_c-a_small_u_1}
\begin{aligned}
\lceil c/a \rceil \alpha(ma/\cK_{2})&\leq (m/\cK_{2})^{-(1+\varrho)}\leq  m^{-(1+\varrho)}\Big(m(c-a)^{\varrho/(2+2\varrho)}\Big)^{1+\varrho}=(c-a)^{\varrho/2}.
\end{aligned}
\end{equation}
By blocking arguments similar to \eqref{eq:blocking_large_c-a_large_u},  {for any $x\in\R$},
\begin{equation}\label{eq:blocking_large_c-a_small_u_2}
    F_{ma,j}(x)\leq \Big(F_{ma/\cK_{2},j}(x)\Big)^{\cK_{2}/2}+(\cK_{2}/2)\alpha(ma/\cK_{2}). 
\end{equation}
By \eqref{eq:blocking_large_c-a_small_u_2}, Assumption \ref{assump:mixing}(ii) and \eqref{eq:cK_2}  we have, for $m \geq m_{0,4}(\alpha(\cdot),A,\varrho)$,
\begin{equation}\label{eq:error_large_c-a_small_u_2}
\begin{aligned}
\Prob\Big(\lceil c/a \rceil \Big(1-F_{ma/\cK_{2},j}(M_{ma,1,j})\Big)>(c-a)^{\varrho/2}\Big) 
&= 
F_{ma,j}\Big(F_{ma/\cK_{2},j}^{\leftarrow}\big(1-\ceil{c/a}^{-1}(c-a)^{\varrho/2}\big)\Big)
\\
&\leq\big(1-\ceil{c/a}^{-1}(c-a)^{\varrho/2}\big)^{\cK_{2}/2}+(\cK_{2}/2)\alpha(ma/\cK_{2}) \\
&\leq 
e^{-(\cK_{2}/2)\ceil{c/a}^{-1}(c-a)^{\varrho/2}}
+\cK_{2}^{2+\varrho}m^{-(1+\varrho)}
\\
&\leq 2(c-a)^{1/4}.
\end{aligned}
\end{equation}
By \eqref{eq:blocking_large_c-a_small_u_1}, \eqref{eq:error_large_c-a_small_u_1}, and \eqref{eq:error_large_c-a_small_u_2} {we can further bound the right-hand side of \eqref{eq:small_v_m}, for sufficiently large $m \geq m_{0,5}(\alpha(\cdot),A,\varrho)$, }
\begin{equation}\label{eq:result_large_c-a_small_u}
\begin{aligned}
&\mathph{\leq}\Prob\bigg(F_{mc,j}(M_{ma,1,j})\leq (c-a)^{\varrho/2}\bigg)
\\
&\leq \Prob\bigg(\Big(F_{ma,j}(M_{ma,1,j})\Big)^{\lceil c/a \rceil}-\lceil c/a \rceil\alpha(m/\cK_{2})-\lceil c/a \rceil\Big(1-F_{m/\cK_{2},j}(M_{ma,1,j})\Big)\leq (c-a)^{\varrho/2}\bigg)
\\
&\leq \Prob\bigg(\Big(F_{ma,j}(M_{ma,1,j})\Big)^{\lceil c/a \rceil}-\lceil c/a \rceil\Big(1-F_{m/\cK_{2},j}(M_{ma,1,j})\Big)\leq 2(c-a)^{\varrho/2}\bigg)
\\
&\leq \Prob\bigg(\Big(F_{ma,j}(M_{ma,1,j})\Big)^{\lceil c/a \rceil}\leq 3(c-a)^{\varrho/2}\bigg)+\Prob\bigg(\lceil c/a \rceil\Big(1-F_{m/\cK_{2},j}(M_{ma,1,j})\Big)>(c-a)^{\varrho/2}\bigg)
\\
&\leq { 3^{1/\lceil a_{\vee}/a_{\wedge}\rceil}(c-a)^{(a_{\vee}/a_{\wedge})\varrho/2} } +2(c-a)^{1/4},
\end{aligned}
\end{equation}
{where we used that $c-a \le 1$ for the last inequality.}
By \eqref{eq:Lipschitz_1_preliminary_1}, \eqref{eq:Lipschitz_1_preliminary_2},
\eqref{eq:result_very_large}, \eqref{eq:result_large}, and \eqref{eq:result_large_c-a_small_u} there exist $K_1 = K_1(\alpha(\cdot),A,\varrho,d) < \infty, \eta_1 = \eta_1(\alpha(\cdot),A,\varrho) > 0$ such that, for $m \geq m_{0,6}(\alpha(\cdot),A,\varrho,d)$, 
\begin{equation}\label{eq:Lipschitz_1_bound}
H_{1} \leq K_1(c-a)^{\eta_1}.
\end{equation}
For $H_{2}$, note that
\begin{equation}\label{eq:Lipschitz_2_bound}
\begin{aligned}
H_{2}&=\Prob(\tbU_{ma,1}\leq \bv, \tbU_{ma,1}\nleq \bu)\leq \sum_{j=1}^{d}\Prob(u_{j}<\tU_{ma,1,j}\leq v_{j})=\sum_{j=1}^{d}\Big(v_{j}-u_{j}\Big)\leq d \|\bu-\bv\|_{\infty}.
\end{aligned}
\end{equation}
For $H_{3}$, by \eqref{eq:coupling},
\begin{equation}\label{eq:Lipschitz_3_preliminary}
\begin{aligned}
H_{3}&=\Prob(\bM_{ma,1}\leq\bF_{ma}^{\leftarrow}(\bu), \bM_{mc,1}\nleq \bF_{ma}^{\leftarrow}(\bu))\\ 
&\leq \sum_{j=1}^{d}\Prob(M_{ma,1,j}\leq F_{ma,j}^{\leftarrow}(u_{j}), M_{mc,1,j}> F_{ma,j}^{\leftarrow}(u_{j})).
\end{aligned}
\end{equation}
To bound this term, consider the cases
\begin{align}
&u_{j}\in \Big((c-a)^{\varrho/2},1\Big]& & \text{and} &  &c-a\in  \Big(a_{\wedge}/8,a_{\vee}-a_{\wedge}\Big]  \cap[0,1] ,\label{eq:vlarge_very_large_c-a}\\
&u_{j}\in \Big((c-a)^{\varrho/2},1\Big]& & \text{and} &  &c-a\in\Big[ m^{-1/2},a_{\wedge}/8\Big)\cap[0,1] , \label{eq:vlarge_large_c-a}\\
&u_{j}\in \Big[0, (c-a)^{\varrho/2}\Big]. \label{eq:vsmall}
\end{align}
For the case~\eqref{eq:vsmall} we have
\[
\Prob\Big(M_{ma,1,j}\leq F_{ma,j}^{\leftarrow}(u_{j}), M_{mc,1,j}> F_{ma,j}^{\leftarrow}(u_{j})\Big) \leq \Prob\Big(M_{ma,1,j}\leq F_{ma,j}^{\leftarrow}(u_{j})\Big) = u_j \leq (c-a)^{\varrho/2}.
\]
In cases \eqref{eq:vlarge_very_large_c-a} and \eqref{eq:vlarge_large_c-a}, we have similarly to~\eqref{eq:large_v}
\[
\Prob\Big(M_{ma,1,j}\leq F_{ma,j}^{\leftarrow}(u_{j}), M_{mc,1,j}> F_{ma,j}^{\leftarrow}(u_{j})\Big) 
\leq 
\Prob\Big(M_{mc-ma,1,j}> F_{ma,j}^{\leftarrow}(u_{j})\Big),
\]
and the right-hand side above can be bounded exactly as before.
The right hand of \eqref{eq:Lipschitz_3_preliminary} has the same form as the right hand of \eqref{eq:Lipschitz_1_preliminary_1}. Hence there exist $K_2 = K_2(\alpha(\cdot),A,\varrho,d) < \infty, \eta_2 = \eta_2(\alpha(\cdot),A,\varrho) > 0$ such that for $m \geq m_{0,7}(\alpha(\cdot),A,\varrho,d)$ 
\begin{equation}\label{eq:Lipschitz_3_bound}
H_{3}\leq K_2(c-a)^{\eta_2}.
\end{equation}
Lemma \ref{lemma:Lipschitz_1} follows from \eqref{eq:Lipschitz}, \eqref{eq:Lipschitz_1_bound}, \eqref{eq:Lipschitz_2_bound}, and \eqref{eq:Lipschitz_3_bound}.
\end{proof}

\begin{Lemma}\label{lemma:Lipschitz_2}
{Let $A'=[a_\wedge', a_\vee']$ denote an arbitrary closed interval in $(0,\infty)$.} There exist $m_0 = m_0(\alpha(\cdot),d,A',\varrho) < \infty$ and $K= K(\alpha(\cdot),d,A',\varrho)$ such that, for all $m\ge m_0$, all $\bu,\bv\in [0,1]^d$ and all $a\in (\bbZ/m)\cap A'$, we have 
\[
\Big\|\bbu_{\bv,a,a}-\bbl_{\bu,a,a}\Big\|_{P_m,2} \leq K \|\bu-\bv\|_{\infty}^{1/2}.
\]
\end{Lemma}

\begin{proof}[Proof of Lemma \ref{lemma:Lipschitz_2}]
As in the previous proof, write $a_\vee = a_\vee'$. By \eqref{eq:Lipschitz} and \eqref{eq:Lipschitz_2_bound},
\begin{align*}
\Big\|\bbu_{\bv,a,a}-\bbl_{\bu,a,a}\Big\|_{P_m,2}^2
&\leq (a_{\vee}+1) \Prob(\tbM_{ma,1}\leq\bF_{ma}^{\leftarrow}(\bv))-\Prob(\tbM_{ma,1}\leq \bF_{ma}^{\leftarrow}(\bu))
\\
&\leq (a_{\vee}+1)d\|\bu-\bv\|_{\infty}. \qedhere
\end{align*}
\end{proof}

\subsection{Proofs for Section~\ref{sec:estmar}}

The following notation is taken from the proof of Theorem 3.5 in~\cite{BucSeg14}. Let
\begin{multline*}
\mathcal{D}_0 := \{f: [0,1]^d \to \R \mid f \mbox{ continuous and } f(\bu) = 0\mbox{ for } \bu = (1,\dots,1)
\\
\mbox{or if at least one coordinate of $\bu$ is equal to 0} \},
\end{multline*}
denote by $\mathcal{D}_\Phi$ the set of all cdfs on $[0,1]^d$ whose marginals put no mass at zero and define
\begin{align*}
\mathcal{D}_{k,n} := \{\alpha \in \ell^\infty([0,1]^d) \mid C_k + (n/k)^{-1/2}\alpha \in \mathcal{D}_\Phi\}, \qquad 1\le k \le n, n\in\N.
\end{align*}
Consider the copula mapping
\[
\Phi: \mathcal D_\Phi \to \ell^\infty([0,1]^d), \quad  H \mapsto H(H_1^-,\dots,H_d^-),
\]
where $H^-$ denotes the left-continuous generalized inverse function, and let
\begin{align*}
g_{k,n} : \cD_{k,n} \to \ell^\infty([0,1]^d); \quad 
&\alpha \mapsto \Big( \bm u \mapsto \sqrt{n/k}\{ \Phi(C_k + (n/k)^{-1/2}\alpha) - \Phi(C_k)\} (\bu) \Big),
\\
g : \cD_0 \to  \ell^\infty([0,1]^d); \quad 
&\alpha \mapsto \Big( \bu \mapsto \alpha(\bu) - \sum_{j=1}^d \dot{C}_{\infty,j}(\bu) \alpha(\bu^{(j)}) \Big).
\end{align*}
In the proof of their Theorem 3.2, \cite{BucSeg14} established the following result under conditions (i) and (ii), the proof under condition (iii) is new to this paper.

\begin{Prop} \label{prop:BucSeg14} 
Assume that Assumption~\ref{assump:copula_convergence} is met, with $C_\infty$ satisfying Assumption~\ref{assump:pd}. Let $k_n$ be a strictly increasing sequence of natural numbers with $k_n = o(n)$ such that one of the following conditions is met:
\begin{compactenum}
\item[(i)] $\SC_1(k_n)$ from Assumption~\ref{assump:BuSe3.4_new}(a) holds;
\item[(ii)] Assumption~\ref{assump:BuSe3.4_new}(b) holds; 
\item[(iii)] $\SC_2(k_n)$ from Assumption~\ref{assump:BuSe3.4_new}(c) holds.
\end{compactenum}
Then, for any sequence $\alpha_n$ in $\mathcal{D}_{k_n,n}$ with $\alpha_n \to \alpha$ in $(\ell^\infty([0,1]^d),\|\cdot\|_\infty)$ where $\alpha \in \mathcal{D}_0$, we have $g_{k_n,n}(\alpha_n) \to g(\alpha)$.
\end{Prop}

\begin{proof} The result under $(i)$ and $(ii)$ was given in the proof of Theorem 3.2 in \cite{BucSeg14} and it remains to establish the statement under $(iii)$. We recall some additional notation from \cite{BucSeg14}: let $\alpha_{nj}(u_{j})=\alpha_{n}(1,\dots,1,u_{j},1,\dots,1)$, $\text{id}_{[0,1]}$ be the identity function on $[0,1]$, $I_{nj}(u_{j})=(\text{id}_{[0,1]}+\sqrt{k_n/n}\alpha_{nj})^{-}(u_{j})$, and $\bI_{n}(\bu)=(I_{n1}(u_{1}),\dots, I_{nd}(u_{d}))$. 

Similar to the proof of Theorem 3.5 in \cite{BucSeg14}, it suffices to show that 
\begin{equation*} %\label{eq:BucSeg14_1}
\sup_{\bu\in [0,1]^{d}}\Big|\sqrt{n/k_n}\{C_{k_n}(\bI_{n}(\bu))-C_{k_n}(\bu)\}+\sum_{j=1}^{d}\dot{C}_{\infty,j}(\bu)\alpha(\bu^{(j)})\Big|=0, 
\end{equation*}
%Let $\Delta_{n}=\sqrt{n/k_n} \{ C_{k_n}-C_{\infty}-\ca(k_n)S \}$. 
%To show \eqref{eq:BucSeg14_1}, it suffices to show that 
which in turn follows if we show that 
\begin{align}
\sup_{\bu\in [0,1]^{d}}\Big|\sqrt{n/k_n}\{C_{\infty}(\bI_{n}(\bu))-C_{\infty}(\bu)\}+\sum_{j=1}^{d}\dot{C}_{\infty,j}(\bu)\alpha(\bu^{(j)})\Big|&\to 0,\label{eq:BucSeg14_2_1}
\\ 
\sup_{\bu\in [0,1]^{d}}|\Delta_{n}(\bI_{n}(\bu))-\Delta_{n}(\bu)|&\to 0, \label{eq:BucSeg14_2_2}
\\
\sup_{\bu\in [0,1]^{d}}\Big|\sqrt{n/k_n}\ca(k_n)[S(\bI_{n}(\bu))-S(\bu)]\Big|&\to 0, \label{eq:BucSeg14_2_3}
\end{align}
where $\Delta_{n}=\sqrt{n/k_n} \{ C_{k_n}-C_{\infty}-\ca(k_n)S \}$. 
Note that \eqref{eq:BucSeg14_2_1} follows by exactly the same arguments as (A.8) in \cite{BucSeg14}, while the proof of \eqref{eq:BucSeg14_2_2} is similar to the proof of (A.9) in \cite{BucSeg14}, using that
$\Delta_n$ is relatively compact in $\cC([0,1]^d)$.
For \eqref{eq:BucSeg14_2_3}, by H{\"o}lder-continuity of $S$,
\[
|S(\bI_{n}(\bu))-S(\bu)|\leq \big(\max_{j=1,\dots,d}|I_{nj}(u_{j})-u_{j}|\big)^{\delta}.
\]
By Vervaat's Lemma, see also formula (4.2) in \cite{BucVol13},
\[
\max_{j=1,\dots,d}|I_{nj}(u_{j})-u_{j}|=O(\sqrt{k_n/n}).
\]
Hence,
\[
\sup_{\bu\in [0,1]^{d}}\Big|\sqrt{n/k_n}\ca(k_n)[S(\bI_{n}(\bu))-S(\bu)]\Big|=O((n/k_n)^{(1-\delta)/2}\ca(k_n))=o(1)
\]
by assumption.
\end{proof}

We will generalize this proposition by including the additional parameter $a\in A$. Define 
\begin{align*}
\widetilde \cD_{k,n} &:= \{\alpha \in \ell^\infty([0,1]^d \times A): C_{\ang{ak}} + (n/k)^{-1/2}\alpha(\, \cdot \, , a) \in \cD_{\Phi}~ \forall a\in A\},
\\
\widetilde \cD_{0} &:= \{\alpha \in \mathcal{C}([0,1]^d \times A): \alpha(\cdot,a) \in \cD_{0}~ \forall a\in A\},
\\
\widetilde g_{k,n}&: \widetilde \cD_{k,n} \to \ell^\infty([0,1]^d \times A), 
\quad 
\alpha \mapsto \Big((\bu,a) \mapsto \sqrt{n/k}\big\{ \Phi(C_{\ang{ak}} + (n/k)^{-1/2}\alpha(\,\cdot\,,a)) - \Phi(C_{\ang{ak}})\big\}(\bu)  \Big),
\\
\widetilde g &: \widetilde \cD_{0} \to \ell^\infty([0,1]^d \times A), 
\quad 
\alpha \mapsto \Big((\bu,a) \mapsto  \alpha(\bu,a) - \sum_{j=1}^d \dot{C}_{\infty,j}(\bu) \alpha(\bu^{(j)},a) \Big).
\end{align*}

\begin{Prop}\label{prop:copblock} 
Assume that Assumption~\ref{assump:copula_convergence} is met, with $C_\infty$ satisfying Assumption~\ref{assump:pd}. Further, let $m=m_n$ be a strictly increasing integer sequence with $m_n=o(n)$ such that one of the following conditions is met:
\begin{compactenum}
\item[(i)]
$\SC_1(\ang{m_n a_n})$ from Assumption~\ref{assump:BuSe3.4_new}(a) holds for every converging sequence $a_n$ in $A$;
\item[(ii)] Assumption~\ref{assump:BuSe3.4_new}(b) holds;
\item[(iii)] $\SC_2(\ang{m_n a_n})$ from Assumption~\ref{assump:BuSe3.4_new}(c) holds for every converging sequence $a_n$ in $A$.
\end{compactenum}  
Then, for any sequence $\alpha_n$ in $\widetilde\cD_{m_n,n}$ with $\alpha_n \to \alpha$ in $(\ell^\infty([0,1]^d \times A),\|\cdot\|_\infty)$ where $\alpha \in \widetilde\cD_0$, we have $\widetilde g_{m_n,n}(\alpha_n) \to \widetilde g(\alpha)$.
\end{Prop}

\begin{proof}[Proof of Proposition~\ref{prop:copblock}] We will proceed by contradiction. Assume that $\widetilde g_{m_n,n}(\alpha_n)$ does not converge to $\widetilde g(\alpha)$. Then there exists $\varepsilon>0$, an increasing sequence $(n_j)_{j\in \N}$ of natural numbers and sequences $(a_n)_{n \in \N}$ in $A$, $(\bu_n)_{n \in \N}$ in $[0,1]^d$ such that 
\begin{equation} \label{eq:tildegn_eps}
\Big|[\widetilde g_{m_{n_j},n_j}(\alpha_{n_j})](\bu_{n_j},a_{n_j}) - [\widetilde g(\alpha)](\bu_{n_j},a_{n_j})\Big| \geq \varepsilon  \quad \forall j \in \N.
\end{equation} 
By compactness of $[0,1]^d\times A$ there further exists a sub-sequence $(n_{j_r})_{r\in \N}$ such that, as $r \to \infty$, $\bu_{n_{j_r}} \to \bu_0 \in [0,1]^d$, $a_{n_{j_r}} \to a_0 \in A$. To simplify notation we shall abbreviate $n_{j_r} = n$. Since $\alpha_n \to \alpha$ uniformly and since $\alpha$ is continuous (thus uniformly continuous as it is defined on a compact set) it follows that 
\[
\|\alpha_n(\cdot,a_n) - \alpha(\cdot, a_0)\|_{\infty} \leq \|\alpha_n(\cdot,a_n) - \alpha(\cdot, a_n)\|_{\infty} + \|\alpha(\cdot,a_n) - \alpha(\cdot, a_0)\|_{\infty} = o(1), 
\]
i.e., $(\beta_n)_{n \in \N}$ with $\beta_n = ( m_n/ \ang{a_nm_n})^{1/2} \alpha_{n}(\cdot, a_n)$
is a sequence in $\cD_{\ang{a_nm_n},n}$ with $\beta_n \to \beta = a_0^{-1/2} \alpha(\cdot, a_0) \in \cD_0$ uniformly. 
We may hence apply Proposition~\ref{prop:BucSeg14} to obtain that
\begin{align*} \label{eq:gbe}
g_{\ang{m_na_n},n}(\beta_n) \to g(\beta).
\end{align*}

Now, the left-hand side of this display equals
\[
\sqrt{n/\ang{a_nm_n}} \big\{ \Phi(C_{\ang{a_nm_n}} + (n/\ang{a_nm_n})^{-1/2} \beta_n {)} - \Phi(C_{\ang{a_nm_n}})\big\}
=
\sqrt{m_n/\ang{a_nm_n}}  \widetilde g_{m_n,n}(\alpha_n)(\,\cdot\,, a_n),
\]
while the right-hand side is equal to
\[
a_0^{-1/2} \widetilde g(\alpha)(\cdot, a_0).
\]
As a consequence, since $a_\wedge >0$, 
\[
\widetilde g_{m_n,n}(\alpha_n)(\,\cdot\,, a_n)
\to
\widetilde g(\alpha)(\cdot, a_0),
\]
uniformly on $[0,1]^d$. Finally, since the mapping $\widetilde g$ is continuous and since $\alpha(\cdot, a_n) \to \alpha(\cdot, a_0)$, it follows that $\widetilde g(\alpha)(\cdot, a_n) \to \widetilde g(\alpha)(\cdot, a_0)$ so that
\[
\lim_{n\to\infty}\Big\|[\widetilde g_{m_n,n}(\alpha_n)](\cdot,a_n) - [\widetilde g(\alpha)](\cdot,a_n)\Big\|_\infty = 0.
\]
This contradicts~\eqref{eq:tildegn_eps}, and thus $\widetilde g_{m_n,n}(\alpha_n) \to \widetilde g(\alpha)$ as asserted. 
\end{proof}

\begin{proof}[Proof of Theorem~\ref{theo:estmar}]
From Theorem~\ref{theo:known} (for the continuity of sample paths) and a careful calculation (for the other condition imposed in the definition of $\widetilde \cD_0$) it is easy to see that there exists a version of $\bbC^{\lozenge}$ with sample paths that are in $\widetilde \cD_0$ almost surely.  Hence, Proposition~\ref{prop:copblock} combined with the extended continuous mapping theorem (Theorem 1.11.1 in \citealp{VanWel96}) implies that
\begin{align*}
\sqrt{n/m}  \big( \hat C_{n,\ang{ma}}^{\text{alt}}- C_{\ang{ma}} \big)
=
\widetilde g_{m, n} ( \bbC_{n,m}^{\lozenge} )   
\dto
\widetilde g(\bbC^{\lozenge}) = \widehat \bbC^{\lozenge},
\end{align*}
where
\[
 \hat C_{n,\ang{ma}}^{\text{alt}}(\bm u) = \hat C_{n,\ang{ma}}^\circ( (\hat C_{n,\ang{ma},1}^\circ)^-(u_1), \dots, (\hat C_{n,\ang{ma},d}^\circ)^-(u_d)). 
\]
Following the lines of the proof of Lemma {A}.2 in \cite{BucSeg14} one can further show that
\[
\sup_{(\bm u,a) \in [0,1]^d \times A}  \Big| \hat C_{n,\ang{ma}}^{\text{alt}}(\bm u)  -  \hat C_{n,\ang{ma}}(\bm u) \Big| = o_\Prob(\sqrt{m/n}),
\]
which implies the assertion.
\end{proof}

\subsection{Proofs for Section~\ref{sec:comp}} \label{sec:proofcomp}

Before proving Theorem~\ref{theo:varcomp}, we state the following lemma which provides a crucial technical ingredient.

\begin{Lemma}\label{lem:varcomp}
Suppose that, for any $n\in\N$, $(X_{n,i})_{i=1,2,\dots,n}$ is an excerpt from a univariate strictly stationary time series. % and that $X_{n,i}$ is uniformly bounded (in $i$ and $n$), almost surely. 
Suppose $m=m_{n}$ is a sequence of positive integers such that $m\to\infty$ and $m/n\to 0$ as $n\to \infty$. For $h\in\N$, let $\Gamma_{n}(h)=\Cov(X_{n,i},X_{n,i+h})$ and assume that $\sup_{h \in \Z, n \in \N} |\Gamma_{n}(h)| < \infty$. If there exists a sequence $\alpha(\cdot)$ such that $\sum_{h=1}^{\infty}\alpha(h)<\infty$ and $|\Gamma_{n}(h+m_n)|\leq\alpha(h)$ for all $h=1,2,\dots$ and $n=1,2,\dots$, then 
\[
 \Var\Big(\sqrt{n/m} \frac{1}{n}\sum_{i=1}^{n} X_{n,i}\Big) \le  \Var\Big( \sqrt{n/m} \frac{m}{n}\sum_{1 \leq h \leq \ang{n/m}} X_{n,1 + m(h-1)} \Big) + o(1), \quad n \to \infty.
\]
\end{Lemma}

\begin{proof}[Proof of Lemma~\ref{lem:varcomp}]
Without loss of generality, we may assume that $\Exp[X_{n,i}]=0$. Observe that
\[
\Var\Big(\sqrt{n/m} \frac{1}{n}\sum_{i=1}^{n} X_{n,i}\Big) = \frac{1}{n} \sum_{i,j = 1}^{n} \frac{\Gamma_n(|i-j|)}{m} = \frac{1}{m} \sum_{|h|<n} (1-|h|/n) \Gamma_n(h).
\]
%Uniform boundedness of $X_{n,i}$ implies that $\Gamma_n(h)$ is uniformly bounded (in $h$ and $n$) and thus, 
By uniform boundedness and symmetry of $\Gamma_n(\cdot)$,
\begin{align*}
\frac{1}{m} \sum_{|h|<n} \frac{|h| |\Gamma_n(h)|}{n} 
&= \frac{1}{n} \sum_{|h|\leq m} \frac{|h|}{m}|\Gamma_n(h)| 
+ \frac{2}{m} \sum_{k=1}^{n-m-1} \frac{k+m}{n} |\Gamma_n(k+m)|
\\
& \leq O(m/n) + \frac{1}{m} O(1) \sum_{k=1}^{n-m-1} \alpha(k) = o(1).  
\end{align*}
Hence 
\[
\Var\Big(\sqrt{n/m} \frac{1}{n}\sum_{i=1}^{n} X_{n,i}\Big) = \frac{1}{m} \sum_{|h|<n} \Gamma_n(h) + o(1).
\]
Similarly
\begin{align*}
\Var\Big( \sqrt{n/m} \frac{m}{n}\sum_{1 \leq h \leq \ang{n/m}} X_{n,1 + m(h-1)} \Big) 
&= \Var\Big( \frac{1}{\sqrt{\ang{n/m}}}\sum_{1 \leq h \leq \ang{n/m}} X_{n,1 + m(h-1)} \Big)(1+o(1)) 
\\
&= \sum_{|h| < \ang{n/m}} \Gamma_n(mh) + o(1).
\end{align*}
From now on assume without loss of generality that $n/m$ is an integer and that $m\ge 2$. It suffices to prove that
\begin{align} \label{eq:tbs}
\sum_{|h| < \ang{n/m}} \Gamma_n(mh) - \frac{1}{m} \sum_{|h|<n} \Gamma_n(h)  = 
\sum_{|h|<n} \gamma_n(h) \Gamma_n(h) \geq o(1),
\end{align}
where
\[
    \gamma_{n}(h)  = 
    \begin{cases}
    \frac{m-1}{m} & \text{if}\ h=0    \pmod {m}, \\
    -\frac{1}{m} & \text{if}\ h\neq 0 \pmod {m}.
    \end{cases}
\]
To this end, for $n\in\N$, let $U_{n}$ be a random variable uniformly distributed on $\{0,1,\dots,m-1\}$ and independent of $X_{n,i}$, $i=1,2,\dots,n$. For $t\in\Z$, let
\[
\phi_{n,t}= 
\begin{cases} 
a_{n} & \text{if}\ t=U_{n}     \pmod {m}, \\
b_{n} & \text{if}\ t\neq U_{n} \pmod {m},
\end{cases}
\]
where $a_{n}=\frac{(m-1)}{\sqrt{m}}$, $b_{n}=-\frac{1}{\sqrt{m}}$. Note that, for any $n\in\N$, $(\phi_{n,t})_{t\in \Z}$ is a second order stationary time series. More precisely, when $h=0 \pmod {m}$, 
\[
\Exp[\phi_{n,t}\phi_{n,t+h}]=\frac{1}{m}a_{n}^{2}+\frac{m-1}{m}b_{n}^{2}=\frac{m-1}{m},
\]
while, for $h\neq 0 \pmod {m}$,
\[
\Exp[\phi_{n,t}\phi_{n,t+h}]=\frac{2}{m}a_{n}b_{n}+\frac{m-2}{m}b_{n}^{2}=-\frac{1}{m},
\]
whence $\Exp[\phi_{n,t}\phi_{n,t+h}] = \gamma_{n}(h)$.
Next, since $X_{n,t}$ is centered, we have
\begin{equation}\label{eq:avg_1}
\begin{aligned}
    0&\leq \Var\Big(\frac{1}{\sqrt{n}}\sum_{t=1}^{n}\phi_{n,t}X_{n,t}\Big)=\sum_{|h|<n}\Big(1-\frac{|h|}{n}\Big)\gamma_{n}(h)\Gamma_{n}(h)=\sum_{|h|<n}\gamma_{n}(h)\Gamma_{n}(h)-\sum_{|h|<n}\frac{|h|}{n}\gamma_{n}(h)\Gamma_{n}(h).
\end{aligned}
\end{equation}
By the Dominated Convergence Theorem,
\begin{equation}\label{eq:avg_2}
    \begin{aligned}
 \bigg|\sum_{|h|<n}\frac{|h|}{n}\gamma_{n}(h)\Gamma_{n}(h)\bigg|
        &\leq
        \bigg|\sum_{|h|<m}\frac{|h|}{n}\gamma_{n}(h)\Gamma_{n}(h)\bigg|
        +
        \bigg|2 \sum_{m\leq h<n}\frac{h}{n}\gamma_{n}(h)\Gamma_{n}(h)\bigg|    \\     
        &\leq
        \sum_{|h|<m}\frac{|h|}{mn}\Big|\Gamma_{n}(h)\Big|
        +        
        2 \sum_{1 \le h <n-m}\frac{h+m}{n}\Big|\Gamma_{n}(h+m)\Big| 
        =O\Big(\frac{m}{n}\Big) + o(1).
    \end{aligned}
\end{equation}
By \eqref{eq:avg_1} and \eqref{eq:avg_2}, 
\[
\sum_{|h|<n}\gamma_{n}(h)\Gamma_{n}(h)+o(1)\geq 0,
\]
which is \eqref{eq:tbs}.
\end{proof}

\begin{proof}[Proof of Theorem~\ref{theo:varcomp}]
Since the proofs for the version with known and estimated margins are similar, we restrict our attention to the case of estimated margins. We need to show that, for any $k\in\N$, $w_1,\dots,w_k$ and $\bu_1,\dots,\bu_k$,
\[
\Var\Big( \sum_{j=1}^k w_j\widehat\bbC^{\lozenge}(\bu_j,1) \Big) \le \Var\Big(\sum_{j=1}^k w_j \widehat\bbC^{D}(\bu_j)\Big).
\] 
For the ease of writing, we only consider the case $k=1$ and $w_1=1$; the general case follows along similar lines.
Let $(\bm X_t)_t$ denote an i.i.d.\ sequence with cdf $F_{1}=C_\infty$. In that case, Assumptions~\ref{assump:copula_convergence}, \ref{assump:mixing} and \ref{assump:BuSe3.4_new} are trivially met, provided $m=m_n\to\infty$ and $m=o(n)$, whence we may apply all results from the proof of Theorem~\ref{theo:estmar}.  
Let us next show that
\[
\Var\Big( \widehat\bbC^{\lozenge}(\bu,1) \Big) = \lim_{n \to \infty} \Var\Big( \sqrt{n/m} \frac{1}{n}\sum_{i=1}^{n} V_{n,i}(\bu) \Big)
\]
where 
\[
V_{n,i}(\bu) := \ind(\bU_{m,i}\leq \bu) - C_m(\bu) - \sum_{j=1}^d \dot C_{\infty,j}(\bu) \{\ind(U_{m,i,j}\leq u_j) - u_j\}.
\] 
Indeed,
\[
V_{n,i}(\bu) = W(\bu)^\top \Big(\ind(\bU_{m,i}\leq \bu),\ind(U_{m,i,1}\leq u_1),...,\ind(U_{m,i,d}\leq u_d) \Big)^\top =: W(\bu)c_{n,i}(\bu) 
\] 
where $W(\bu)^\top = (1,-\dot C_{\infty,1}(\bu),\dots ,-\dot C_{\infty,d}(\bu))$.  Now, the covariance matrix of $\frac{1}{b}\sum_{i=1}^b c_{n,i}(\bu)$ converges by Lemma~\ref{lemma:cov}, and the same is true for $\frac{1}{n}\sum_{i=1}^n c_{n,i}(\bu)$, since $b=n(1+o(1))$ and since the $c_{n,i}$ are bounded by $1$ coordinate-wise. The claim then follows by simple linear algebra.

A similar argument can be used to show that 
\[
\Var\Big( \widehat\bbC^{D}(\bu) \Big) = \lim_{n \to \infty} \Var\Big( \sqrt{n/m} \frac{m}{n}\sum_{1 \leq h \leq \ang{n/m}} V_{1 + m(h-1),n}(\bu) \Big);
\]
note that the summands on the right-hand side are independent due to the independence of $(\bm X_t)_t$

Finally, apply Lemma~\ref{lem:varcomp} with $X_{n,i} := V_{n,i}(\bu)$ (note that the condition on the autocovariances is trivially satisfied with $\alpha\equiv0$). 
\end{proof}

\subsection{Analytical expressions for the variances in Figure~\ref{fig:compvarth}} \label{sec:asyvar}
A tedious but straightforward calculation shows that one may alternatively express $\gamma(\bv,\bu,c,a)$ from Theorem~\ref{theo:known}, for $a_\wedge\leq a\leq c\leq a_\vee$, as
\[
\gamma(\bv,\bu,c,a) = 
\begin{cases}
0, & C_{\infty}(\bv^{a/c}\wedge\bu)= C_{\infty}(\bv^{a/c})C_{\infty}(\bu)
\\
-(c+a)C_{\infty}(\bv)C_{\infty}(\bu) & C_{\infty}(\bv^{a/c}\wedge\bu) = 0 \neq C_{\infty}(\bv^{a/c})C_{\infty}(\bu)
\end{cases}
\]
and otherwise
\begin{multline*}
\gamma(\bv,\bu,c,a) = C_{\infty}(\bv^{1-a/c})\Big[2a\frac {C_{\infty}(\bv^{a/c}\wedge\bu)-C_{\infty}(\bv^{a/c})C_{\infty}(\bu)}{\ln C_{\infty}(\bv^{a/c}\wedge\bu)-\ln \big(C_{\infty}(\bv^{a/c})C_{\infty}(\bu)\big)}
\\
+(c-a)C_{\infty}(\bv^{a/c}\wedge\bu)-(c+a)C_{\infty}(\bv^{a/c})C_{\infty}(\bu)\Big].
\end{multline*}
For $a=c=1$, the expression further simplifies to
\[
\gamma(\bu,\bv,1,1)=\begin{cases}
0,      & \text{if}\ C_{\infty}(\bu\wedge\bv)= C_{\infty}(\bu)C_{\infty}(\bv), 
\\
-2C_{\infty}(\bu)C_{\infty}(\bv),      & \text{if}\ C_{\infty}(\bu\wedge\bv) = 0 \neq C_{\infty}(\bu)C_{\infty}(\bv), 
\\
2\Big(\frac{C_{\infty}(\bu\wedge\bv)-C_{\infty}(\bu)C_{\infty}(\bv)}
{\ln C_{\infty}(\bu\wedge\bv)-\ln(C_{\infty}(\bu)C_{\infty}(\bv))}
-C_{\infty}(\bu)C_{\infty}(\bv)\Big),      & \text{else}. 
\end{cases}
\]

For the bivariate Gumbel--Hougaard copula $C_\infty$ with shape parameter $\beta\geq1$, it is then straightforward to show that, for $u\in (0,1),$
\begin{align*}
    C_{\infty}(u,u)&=u^{(2^{1/\beta})},\\
    \dC_{\infty,j}(u,u)&=u^{(2^{1/\beta}-1)}2^{1/\beta-1}, \qquad (j\in\{1,2\}),\\
    \Var(\widehat\bbC^{\lozenge}(u,u,1))&=2\bigg(\frac{C_{\infty}(u,u)-(C_{\infty}(u,u))^{2}}{-\ln C_{\infty}(u,u)}-(C_{\infty}(u,u))^{2}\bigg)\\
    &\mathph{=}+\ 4(\dC_{\infty,1}(u,u))^{2}\bigg[\bigg(\frac{u-u^{2}}{-\ln u}-u^{2}\bigg)+\ind\{\beta>1\}\bigg(\frac{C_{\infty}(u,u)-u^{2}}{\ln C_{\infty}(u,u)- 2\ln u}- u^{2}\bigg)\bigg]\\
    &\mathph{=}-\ 8\dC_{\infty,1}(u,u)\bigg(\frac{C_{\infty}(u,u)-C_{\infty}(u,u)u}{-\ln u}-C_{\infty}(u,u)u\bigg),\\
    \Var(\widehat\bbC^{D}(u,u))&=C_{\infty}(u,u)-(C_{\infty}(u,u))^{2}+2(\dC_{\infty,1}(u,u))^{2}\Big(u-u^{2}+C_{\infty}(u,u)-u^{2}\Big)\\
    &\mathph{=} -\ 4\dC_{\infty,1}(u,u)\Big(C_{\infty}(u,u)-C_{\infty}(u,u)u\Big)
\end{align*}
In particular, when $\beta=1$, the Gumbel--Hougaard copula degenerates to the independent copula, and 
\begin{align*}
    \Var(\widehat\bbC^{\lozenge}(u,u,1))&=2\bigg(\frac{u^2-u^4}{\ln u^2 -\ln u^4}-u^4\bigg)-4\bigg(\frac{u^3-u^4}{\ln u^3 -\ln u^4}-u^4\bigg),\\
    \Var(\widehat\bbC^{D}(u,u))&=u^2-2u^3+u^4.
\end{align*}

\section{Proofs for Section~3}
\label{sec:proof2}

\subsection{Proof of Lemma~\ref{lem:regvarab} and Lemma~\ref{lem:bounddiffa}}

%\begin{proof}[Proof of Lemma~\ref{lem:regvarab} and Lemma~\ref{lem:bounddiffa}] 
Elementary calculations as in Lemma 2.2 in \cite{BucVolZou19} show that, if Assumption~\ref{assump:high_order} and the expansion in \eqref{eq:unifCk} is met, then
\begin{equation}\label{eq:homS}
\frac{S(\bm u^s)}{C_\infty(\bm u^s)} = s^{1-\rho_{\ca}} \frac{S(\bm u)}{C_\infty(\bm u)}
\end{equation}
for all $s>0, \bm u\in(0,1]^d$.

The following consequence of \eqref{eqn:thirdor}, holding for any $x \in (0,\infty)$ with $x_k := \ang{xk}/k$ and for any $\bm u \in (0,1]^d$, will be used repeatedly:
\begin{align} 
&C_k(\bu^{1/x_k})^{x_k} - C_\infty(\bu^{1/x_k})^{x_k}\nonumber
\\ 
=\ & x_k C_\infty(\bu^{1/x_k})^{x_k-1} \big\{ C_k(\bu^{1/x_k}) - C_\infty(\bu^{1/x_k}) \big\} \nonumber
\\
& + \frac{x_k(x_k-1)}{2}C_\infty(\bu^{1/x_k})^{x_k-2} \big\{C_k(\bu^{1/x_k}) - C_\infty(\bu^{1/x_k}) \big\}^2 + (x_k-1)O(\ca(k)^3 ) \nonumber
\\
=\ & x_k C_\infty(\bu^{1/x_k})^{x_k-1}\Big\{ \ca(k)S(\bu^{1/x_k}) + \ca(k)\cb(k)T(\bu^{1/x_k}) + o(\ca(k)\cb(k))\Big\} \nonumber
\\
&+ \frac{x_k(x_k-1)}{2}C_\infty(\bu^{1/x_k})^{x_k-2}\Big\{ \ca(k)S(\bu^{1/x_k}) + \ca(k)\cb(k)T(\bu^{1/x_k}) + o(\ca(k)\cb(k))\Big\}^2 \nonumber
\\
&+(x_k-1)O(\ca(k)^3) \nonumber 
\\
=\ & x_k C_\infty(\bu)\Big\{ \ca(k)\frac{S(\bu^{1/x_k})}{C_\infty(\bu^{1/x_k})} + \ca(k)\cb(k)\frac{T(\bu^{1/x_k})}{C_\infty(\bu^{1/x_k})}
+ \ca(k)^2 \frac{(x_k-1)}{2}\Big[\frac{S(\bu^{1/x_k})}{C_\infty(\bu^{1/x_k})}\Big]^2 \Big\} \nonumber
\\
& + (x_k-1)O(\ca(k)^3) + o(\ca(k)\cb(k)) \label{eqn_cmxexp1}
\\
=\ & x_k C_\infty(\bu)\Big\{ \ca(k)\frac{S(\bu)}{C_\infty(\bu)}x_k^{\rho_\ca-1} + \ca(k)\cb(k)\frac{T(\bu^{1/x_k})}{C_\infty(\bu^{1/x_k})}
+ (1+o(1))\ca(k)^2 \frac{(x_k-1)x_k^{2\rho_\ca-2} }{2}\Big[\frac{S(\bu)}{C_\infty(\bu)}\Big]^2 \Big\}\nonumber
\\
& + o(\ca(k)\cb(k)). \label{eqn_cmxexp2}
\end{align} 

For the proof of Lemma~\ref{lem:bounddiffa}, 
note that the $o$-terms in~\eqref{eqn_cmxexp2} are uniform in $x\in \cX$.
Next, observe that 
\begin{align*}
 C_k(\bu^{1/x_k})^{x_k} - C_\infty(\bu^{1/x_k})^{x_k} 
 &= 
 C_{kx_k}(\bu) - C_\infty(\bu) + O(\delta(k)) \\
&=
 \ca(kx_k) S(\bu) + \ca(kx_k)\cb(kx_k) T(\bu) + o\Big(\ca(kx_k)\cb(kx_k)\Big)  + O(\delta(k))
\end{align*}
uniformly  in $x\in \cX$, by \eqref{eq:unifCk} and \eqref{eqn:thirdor}.
Take the difference of this expansion with~\eqref{eqn_cmxexp2} and note that by regular variation of $\cb$ we have, uniformly in $x \in \cX$, $\cb(\ang{kx}) = O(\cb(k))$. We obtain 
\begin{align*}
\Big\{\ca(\ang{kx}) - \big(\tfrac{\ang{kx}}{k}\big)^{\rho_\ca}\ca(k)\Big\}S(\bu) = O\Big(\ca(k)\cb(k) + \ca^2(k) + \delta(k) \Big)
\end{align*}
uniformly in $x \in \cX$, which implies Lemma~\ref{lem:bounddiffa}.

Next let us prove Lemma~\ref{lem:regvarab}. Note that under the i.i.d. assumption on $(\bX_t)_{t \in \Z}$ we have 
\[
C_{kx_k}(\bu) - C_\infty(\bu) = C_k(\bu^{1/x_k})^{x_k} - C_\infty(\bu^{1/x_k})^{x_k}.
\]
By linear independence of  $S, S^2/C_\infty, T$ there exist points $\bu_1,\bu_2,\bu_3 \in (0,1]^d$ such that the vectors $(S,S^2/C_\infty,T)(\bu_i), i=1,\dots,3$ are linearly independent. By continuity of $S,C_\infty,T$ this implies the existence of an open interval $\cX \subset (1/2,2)$ with $1 \in \cX$ such that the determinant of the matrix with columns $(S,S^2/C_\infty,T)(\bu_i^{1/x}), i=1,\dots,3$ is bounded away from zero uniformly on $x \in \cX$.

A combination of~\eqref{eq:homS} and \eqref{eqn_cmxexp1} shows that (recall that $x_k \to x$)
\begin{align*}
&\ca(kx_k)x_k^{1-\rho_\ca}  C_\infty(\bu)\frac{S(\bu^{1/x_k})}{C_\infty(\bu^{1/x_k})} + \ca(kx_k)\cb(kx_k) T(\bu) + o\Big(\ca(kx_k)\cb(kx_k)\Big)
\\
=\ & C_{kx_k}(\bu) - C_\infty(\bu) 
\\
=\ & x_k \frac{C_\infty(\bu)}{C_\infty(\bu^{1/x_k})}\Big\{ \ca(k)S(\bu^{1/x_k}) + \ca(k)\cb(k)T(\bu^{1/x_k})
+ \ca(k)^2 \frac{(x_k-1)}{2}\frac{S^2(\bu^{1/x_k})}{C_\infty(\bu^{1/x_k})}(1+o(1)) \Big\} 
\\
& \hspace{9cm} + o(\ca(k)\cb(k)).
\end{align*}
Re-arrange terms, recall that $x_k \to x$ and use regular variation of $\ca$ to obtain %for any $(\bu,x) \in U\times\cX$ (which implies $T(\bu^{1/x}) \neq 0$)
\begin{align}
\frac{T(\bu)C_\infty(\bu^{1/x})}{xC_\infty(\bu)} 
=\ & \frac{\frac{\ca(k)}{\ca(kx_k)} - x_k^{-\rho_\ca} }{\cb(kx_k)}\Big[S(\bu^{1/x})+r_{k,1}(\bu,x)\Big] + \frac{\cb(k)}{\cb(kx_k)}\Big[x^{-\rho_\ca}T(\bu^{1/x})+r_{k,2}(\bu,x)\Big]  \nonumber
\\
& + \frac{\ca(k)}{\cb(kx_k)}\frac{x_k-1}{2}\Big[\frac{S^2(\bu^{1/x})}{C_\infty(\bu^{1/x})}x^{-\rho_\ca} + r_{k,3} (\bu,x)\Big] + o(1), \label{eq:linrep}
\end{align}
where the remainder terms satisfy $r_{k,i}(\bu,x) = o(1), i=1,2,3$ uniformly in $x \in \cX$. 

Define the vectors
\[
v_k(\bu,x) := \Big(S(\bu^{1/x})+r_{k,1}(\bu,x),x^{-\rho_\ca}T(\bu^{1/x})+r_{k,2}(\bu,x),\frac{S^2(\bu^{1/x})}{C_\infty(\bu^{1/x})}x^{-\rho_\ca} + r_{k,3}(\bu,x)\Big)^\top. 
\]
By construction of the points $\bu_i$ and the interval $\cX$, there exists $M\in\N$ such that the determinant of the matrix $V_k(x)$ with rows $v_k(\bu_1,x),v_k(\bu_2,x),v_k(\bu_3,x)$ is bounded away from zero uniformly in $x \in \cX$ for all $k > M$. Define the vector 
\[
\bw_k(x) := \Big(\frac{ \frac{\ca(k)}{\ca(kx_k)}- x_k^{-\rho_\ca} }{\cb(kx_k)},\frac{\cb(k)}{\cb(kx_k)},\frac{x_k-1}{2}\frac{\ca(k)}{\cb(kx_k)}\Big)^\top.
\] 
Applying~\eqref{eq:linrep} with $\bu = \bu_i, i=1,2,3$ we find that, 
\[
V_k(x)\bw_k(x) \to \Big(\frac{T(\bu_1)C_\infty(\bu_1^{1/x})}{xC_\infty(\bu_1)},\frac{T(\bu_2)C_\infty(\bu_2^{1/x})}{xC_\infty(\bu_2)},\frac{T(\bu_3)C_\infty(\bu_3^{1/x})}{xC_\infty(\bu_3)} \Big)^\top \quad \text{as} \quad k \to \infty.
\]
Since $V_k(x)$ converges to an invertible matrix, we obtain that $\bw_k(x)$ converges to a vector $\bw(x)$ with entries $w_1(x),w_2(x),w_3(x)$ where each entry is a continuous function of $x$. Since $\bw_k(1) = (0,1,0)$ for all $k$, it follows that $w_2(1) = 1$. This implies that the limit of $\cb(\ang{tx})/\cb(\ang{t}) = \cb(x_t \ang{t})/\cb(\ang{t})$ with $x_t := {\ang{tx}}/{\ang{t}} \to x$ exists and is positive for all $x$ in an open set containing $1$. Regular variation of the function $t \mapsto \cb(\ang{t})$ follows by an application of Theorem B.1.3 in~\cite{DehFer06}. 
\qed
%\end{proof}

\subsection{Proof of Lemma~\ref{rem:raiid}}

We begin by proving a couple of preliminary results. All results hold under Assumptions~\ref{assump:copula_convergence} and~\ref{assump:mixing} and all convergences are for $m\to\infty$, if not mentioned otherwise. 

\begin{Lemma}\label{lemma:trans_block_length}
Let $v\in (0,1]$ and $y, z\in \bbN$ with $2y\leq z$. If
\begin{equation}\label{eq:assump_trans_block_length}
(z/(2y))\alpha(y) \leq v/2,
\end{equation}
then, for any $j=1,\dots,d$, 
\[
\Prob\Big(M_{y,1,j} \ge  F_{z,j}^{\leftarrow}(v) \Big) \leq \ang{z/(2y)}^{-1}[-\log (v/2)].
\]
\end{Lemma}
\begin{proof}[Proof of Lemma \ref{lemma:trans_block_length}]
By Assumption~\ref{assump:mixing} we have, for any $x \in \R$ and $j=1,\dots,d$,
\begin{equation}\label{eq:trans_block_length}
\begin{aligned}
F_{z,j}(x)&=\Prob(\max_{t=1,\dots,z}X_{t,j}\leq x) 
\\
&\leq\Prob(\max_{t=2(k-1)y+1,\dots,(2k-1)y}X_{t,j}\leq x,\ k=1,2,\dots,\ang{z/(2y)})
\\
&\leq\prod_{k=1}^{\ang{z/(2y)}}\Prob\Big(\max_{t=2(k-1)y+1,\dots,(2k-1)y}X_{t,j}\leq x\Big)+\ang{z/(2y)}\alpha(y)
\\
&= \Big(F_{y,j}(x)\Big)^{\ang{z/(2y)}}+\ang{z/(2y)}\alpha(y).
\end{aligned}
\end{equation}
By \eqref{eq:assump_trans_block_length} and \eqref{eq:trans_block_length},
\begin{equation*}
\begin{aligned}
\Prob\bigg(M_{y,1,j} \ge  F_{z,j}^{\leftarrow}(v)\bigg)&=\Prob\bigg(F_{z,j}\big(M_{y,1,j}\big) \ge  v\bigg)\leq \Prob\bigg(\Big(F_{y,j}\big(M_{y,1,j}\big)\Big)^{\ang{z/(2y)}}+\ang{z/(2y)}\alpha(y) \ge v\bigg)
\\
&\leq\Prob\bigg(F_{y,j}\big(M_{y,1,j}\big) \ge (v/2)^{1/\ang{z/(2y)}}\bigg)
=1-(v/2)^{1/\ang{z/(2y)}}
\\
&\leq \ang{z/(2y)}^{-1}[-\log (v/2)].
\end{aligned}
\end{equation*}
where the last line follows since $ 1-x \le - \log x $ for $x \in [0,1]$. 
\end{proof}

\begin{Lemma}\label{lemma:trans_block_length_limit}
Let $\bv\in (0,1]^{d}$. Suppose $y=y_{m}$ and $z=z_{m}$ are $\bbN$-valued sequences with $y=o(z)$ and $(z/y)\alpha(y)\to 0$,
then
\[
\Prob\Big(\bM_{y,1}\nleq \bF_{z}^{\leftarrow}(\bv)\Big) =O(y/z).
\]
\end{Lemma}
\begin{proof}[Proof of Lemma \ref{lemma:trans_block_length_limit}]
Notice
\begin{equation}\label{eq:trans_block_length_limit}
\Prob\Big(\bM_{y,1}\nleq \bF_{z}^{\leftarrow}(\bv)\Big)
\leq \sum_{j=1}^{d}\Prob\Big(M_{y,1,j}> F_{z,j}^{\leftarrow}(v_{j})\Big).  
\end{equation}
Lemma \ref{lemma:trans_block_length_limit} follows from Lemma \ref{lemma:trans_block_length} and \eqref{eq:trans_block_length_limit}. 
\end{proof}

\begin{Lemma}\label{lemma:separate_block}
Suppose $\Delta=\Delta_{m}$ is an $\bbN$-valued sequence satisfying $\Delta=o(m)$ and $(m/\Delta) \alpha(\Delta)=o(1)$. Let
\begin{align*}
\cS=\bigg\{\Big(\{J_{m}\}_{m\in\bbN},\{b_{m,j}\}_{m\in\bbN,j=0,\dots,J_{m}}\Big): 
&J_{m}\in \bbN,\ b_{m,j}\in \bbN_0,\ 0=b_{m,0}\leq b_{m,1}\leq \dots \leq b_{m,J_{m}}, \\
&\Delta_{m}<\min_{j=1,\dots,J_{m}}(b_{m,j}-b_{m,j-1}) \bigg\}.
\end{align*}
Each element of $\cS$ should be interpreted as a sequence of `numbers of blocks' and a triangular array of `end points of blocks' in a block-wise partition of the first $b_{m, J_m}$ integers. 
Then
\begin{align*}
\sup_{(J,b)\in \cS}J_{m}^{-1}\bigg|&\Prob\Big(\bM_{b_{m,J_{m}},1}\leq \bF_{m}^{\leftarrow}(\bv)\Big) - \prod_{j=1}^{J_m} \Prob\Big(\bM_{b_{m,j}-b_{m,j-1},1}\leq \bF_{m}^{\leftarrow}(\bv)\Big)
\bigg| = O\Big(\alpha(\Delta)+\Delta/m\Big). 
\end{align*}
\end{Lemma}
\begin{proof}[Proof of Lemma \ref{lemma:separate_block}]
Observe the decomposition
\begin{align*}
&\mathph{=}J_{m}^{-1}\bigg\{\Prob\Big(\bM_{b_{m,J_{m}},1}\leq \bF_{m}^{\leftarrow}(\bv)\Big)
- \prod_{j=1}^{J_m} \Prob\Big(\bM_{b_{m,j}-b_{m,0},j-1}\leq \bF_{m}^{\leftarrow}(\bv)\Big) \bigg\}
%\\
%&\mathph{=}-\Prob\Big(\bM_{b_{m,1}-b_{m,0},1}\leq \bF_{m}^{\leftarrow}(\bv)\Big)\dots\Prob\Big(\bM_{b_{m,J_{m}-2}-b_{m,J_{m}-1},1}\leq \bF_{m}^{\leftarrow}(\bv)\Big)\Prob\Big(\bM_{b_{m,J_{m}}-b_{m,J_{m}-1},1}\leq \bF_{m}^{\leftarrow}(\bv)\Big)\bigg)
%\\
%&
=L_{1}+L_{2}+L_{3},
\end{align*}
where
\begin{align*}
L_{1} &= J_{m}^{-1}\bigg\{ \Prob\Big(\bM_{(b_{m,j-1}+1):b_{m,j}}\leq \bF_{m}^{\leftarrow}(\bv), j=1,...,J_m\Big) - \Prob\Big(\bM_{(b_{m,j-1}+1):(b_{m,j}-\Delta)}\leq \bF_{m}^{\leftarrow}(\bv), j=1,...,J_m\Big)\bigg\},
\\
L_{2} &= J_{m}^{-1}\bigg\{\Prob\Big(\bM_{(b_{m,j-1}+1):(b_{m,j}-\Delta)}\leq \bF_{m}^{\leftarrow}(\bv), j=1,...,J_m\Big) - \prod_{j=1}^{J_m} \Prob\Big(\bM_{(b_{m,j-1}+1):(b_{m,j}-\Delta)}\leq \bF_{m}^{\leftarrow}(\bv)\Big)   \bigg\},
\\
L_{3} &= J_{m}^{-1}\bigg\{ \prod_{j=1}^{J_m} \Prob\Big(\bM_{(b_{m,j-1}+1):(b_{m,j}-\Delta)}\leq \bF_{m}^{\leftarrow}(\bv)\Big) - \prod_{j=1}^{J_m} \Prob\Big(\bM_{b_{m,j}-b_{m,j-1},1}\leq \bF_{m}^{\leftarrow}(\bv)\Big) \bigg\}.
\end{align*}
By Assumption \ref{assump:mixing} 
\begin{equation}\label{eq:separate_block_2}
|L_{2}|\leq \alpha(\Delta).
\end{equation}
In addition, note that for arbitrary finite events $V_k \subset W_k, k=1,\dots,K$ we have 
\[
\Prob(\cap_k W_k) - \Prob(\cap_k V_k) \leq \sum_{k=1}^K \Prob(W_k) - \Prob(V_k).
\] 
Hence 
\begin{equation}\label{eq:separate_block_1_preliminary}
\begin{aligned}
    |L_{1}|&\leq
J_{m}^{-1}\sum_{j=1}^{J_{m}}\bigg|\Prob\Big(\bM_{(b_{m,j-1}+1):b_{m,j}}\leq \bF_{m}^{\leftarrow}(\bv)\Big)-\Prob\Big(\bM_{(b_{m,j-1}+1):b_{m,j}-\Delta}\leq \bF_{m}^{\leftarrow}(\bv)\Big)\bigg|.
\end{aligned}
\end{equation}
Since $\Delta=o(m)$ and $(m/\Delta) \alpha(\Delta)=o(1)$, by Lemma \ref{lemma:trans_block_length_limit}, 
\begin{equation}\label{eq:separate_block_1_preliminary_2}
\sup_{(J,b)\in \cS} \bigg|\Prob\Big(\bM_{(b_{m,j-1}+1):b_{m,j}}\leq \bF_{m}^{\leftarrow}(\bv)\Big)-\Prob\Big(\bM_{(b_{m,j-1}+1):b_{m,j}-\Delta}\leq \bF_{m}^{\leftarrow}(\bv)\Big)\bigg|=O\Big(\Delta/m\Big).
\end{equation}
Hence, by \eqref{eq:separate_block_1_preliminary} and \eqref{eq:separate_block_1_preliminary_2},
\begin{equation}\label{eq:separate_block_1}
\sup_{(J,b)\in \cS}|L_{1}|=O\Big(\Delta/m\Big).
\end{equation}
By~\eqref{eq:separate_block_1_preliminary_2} and since $|\prod_i a_i - \prod_i b_i| \leq \sum_i |a_i - b_i|$ if $a_i,b_i \in [0,1]$ we have
\begin{align}\label{eq:separate_block_3}
\sup_{(J,b)\in \cS}|L_{3}| \nonumber
&\leq 
{\sup_{(J,b)\in \cS}} J_{m}^{-1}\sum_{j=1}^{J_{m}}\bigg|\Prob\Big(\bM_{(b_{m,j-1}+1):b_{m,j}}\leq \bF_{m}^{\leftarrow}(\bv)\Big)-\Prob\Big(\bM_{(b_{m,j-1}+1):b_{m,j}-\Delta}\leq \bF_{m}^{\leftarrow}(\bv)\Big)\bigg| \\
&=O\Big( \Delta/m\Big).
\end{align}
Lemma \ref{lemma:separate_block} follows from \eqref{eq:separate_block_2}, \eqref{eq:separate_block_1}, and \eqref{eq:separate_block_3}.
\end{proof}

\begin{Lemma}\label{lemma:integer_multiplier}
For fixed $\bv\in (0,1]^{d}$ (hence $C_{\infty}(\bv)>0$) and $\ell \in \bbN$, we have 
\[
\bigg|\Prob\big(\bM_{ml,1}\leq \bF_{m}^{\leftarrow}(\bv)\big)-\Big(\Prob\big(\bM_{m,1}\leq \bF_{m}^{\leftarrow}(\bv)\big)\Big)^{\ell }\bigg|=O\Big( m^{-(1+\varrho)/(2+\varrho)}\Big),
\]
where $\varrho>0$ is defined in Assumption~\ref{assump:mixing}(ii).
\end{Lemma}

\begin{proof}[Proof of Lemma \ref{lemma:integer_multiplier}]
Applying Lemma \ref{lemma:separate_block} with $\Delta=\ang{m^{1/(2+\varrho)}}$, $J=J_m=\ell$, and $b_{m,j}=j m$ for $j=0, \dots \ell$, we get 
\[
\bigg|\Prob\big(\bM_{ml,1}\leq \bF_{m}^{\leftarrow}(\bv)\big)-\Big(\Prob\big(\bM_{m,1}\leq \bF_{m}^{\leftarrow}(\bv)\big)\Big)^{\ell}\bigg|
=O\Big(\alpha(\Delta)+\Delta/m\Big)
=O\Big(m^{-(1+\varrho)/(2+\varrho)}\Big).   \qedhere
\]
\end{proof}

\begin{Lemma}\label{lemma:one_over_multiplier}
Let $\bv\in (0,1]^{d}$ and $L=L_{m}$ be an $\bbN$-valued sequence with $L_{m}=o(m^{(1+\varrho)/(2+\varrho)})$.
Then
\[
\max_{\ell=1,\dots,L}\bigg|\Prob\big(\bM_{\ang{m/\ell},1}\leq \bF_{m}^{\leftarrow}(\bv)\big)-\Big(\Prob\big(\bM_{m,1}\leq \bF_{m}^{\leftarrow}(\bv)\big)\Big)^{1/\ell}\bigg|=O\Big( m^{-(1+\varrho)/(2+\varrho)}\Big). 
\]
\end{Lemma}

\begin{proof}[Proof of Lemma \ref{lemma:one_over_multiplier}]
For fixed $\ell\in\N$, we have
\[
\ell^{-1}\bigg(\Prob\big(\bM_{m,1}\leq \bF_{m}^{\leftarrow}(\bv)\big)-\Big(\Prob\big(\bM_{\ang{m/\ell},1}\leq \bF_{m}^{\leftarrow}(\bv)\big)\Big)^{\ell}\bigg)=N_{1,\ell}+N_{2,\ell},
\]
where
\begin{align*}
    N_{1,\ell}&=\ell^{-1}\bigg(\Prob\big(\bM_{m,1}\leq \bF_{m}^{\leftarrow}(\bv)\big)-\Prob\big(\bM_{\ell\ang{m/\ell},1}\leq \bF_{m}^{\leftarrow}(\bv)\big)\bigg),\\
    N_{2,\ell}&=\ell^{-1}\bigg(\Prob\big(\bM_{\ell\ang{m/\ell},1}\leq \bF_{m}^{\leftarrow}(\bv)\big)-\Big(\Prob\big(\bM_{\ang{m/\ell},1}\leq \bF_{m}^{\leftarrow}(\bv)\big)\Big)^{\ell}\bigg).
\end{align*}
Applying Lemma \ref{lemma:separate_block} with $\Delta=\ang{m^{1/(2+\varrho)}}$, $J_m\equiv\ell$, and $b_{m,j}=j \ang{m/\ell}$, we get 
\begin{equation}\label{eq:pre_Taylor_1}
\sup_{\ell=1,\dots,L}|N_{2,\ell}|= O\Big(m^{-(1+\varrho)/(2+\varrho)}\Big).
\end{equation}
Further, if $\ell\leq \ang{m^{1/(2+\varrho)}}$, 
\begin{equation*} %\label{eq:pre_Taylor_2_1}
    |N_{1,\ell}|\leq \ell^{-1}\Prob\big(\bM_{\ell}\nleq \bF_{m}^{\leftarrow}(\bv) \big)\leq \Prob\big(\bM_{\ell}\nleq \bF_{m}^{\leftarrow}(\bv) \big)\leq \Prob\big(\bM_{\ang{m^{1/(2+\varrho)}}}\nleq \bF_{m}^{\leftarrow}(\bv) \big).
\end{equation*}
Hence, by Lemma \ref{lemma:trans_block_length} and Lemma \ref{lemma:trans_block_length_limit},
\begin{equation}\label{eq:pre_Taylor_2}
\begin{aligned}
\max_{\ell=1,\dots,L}|N_{1,\ell}|
&\leq 
\max\Big\{\max_{\ell=\ang{m^{1/(2+\varrho)}}+1,\dots,L}|N_{1,\ell}|,\max_{\ell=1,\dots,\ang{m^{1/(2+\varrho)}}}|N_{1,\ell}|\Big\}
\\
&\leq \max_{\ell=\ang{m^{1/(2+\varrho)}}+1,\dots,L}|N_{1,\ell}|+\max_{\ell=1,\dots,\ang{m^{1/(2+\varrho)}}}|N_{1,\ell}|
\\
&=\max_{\ell=\ang{m^{1/(2+\varrho)}}+1,\dots,L}\ell^{-1}\Prob\big(\bM_{\ell}\nleq \bF_{m}^{\leftarrow}(\bv) \big)+\Prob\big(\bM_{\ang{m^{1/(2+\varrho)}}}\nleq \bF_{m}^{\leftarrow}(\bv)\big)
\\
&=O(m^{-1})+O(m^{-(1+\varrho)/(2+\varrho)}) = O(m^{-(1+\varrho)/(2+\varrho)}).
\end{aligned}
\end{equation}
By \eqref{eq:pre_Taylor_1} and \eqref{eq:pre_Taylor_2}, 
\begin{equation}\label{eq:pre_Taylor}
\max_{\ell=1,\dots,L}\ell^{-1}\bigg|\Prob\big(\bM_{m,1}\leq \bF_{m}^{\leftarrow}(\bv)\big)-\Big(\Prob\big(\bM_{\ang{m/\ell},1}\leq \bF_{m}^{\leftarrow}(\bv)\big)\Big)^{\ell}\bigg|=O\Big(m^{-(1+\varrho)/(2+\varrho)}\Big).
\end{equation}

Next, by the Mean Value Theorem, 
\begin{equation}\label{eq:Taylor}
\begin{aligned}
&\mathph{=}\bigg|\Big(\Prob\big(\bM_{m,1}\leq \bF_{m}^{\leftarrow}(\bv)\big)\Big)^{1/\ell}-\Prob\big(\bM_{\ang{m/\ell},1}\leq \bF_{m}^{\leftarrow}(\bv)\big)\bigg|\\
&=\bigg|\Big(\Prob\big(\bM_{m,1}\leq \bF_{m}^{\leftarrow}(\bv)\big)\Big)^{1/\ell}-\Big(\Big(\Prob\big(\bM_{\ang{m/\ell},1}\leq \bF_{m}^{\leftarrow}(\bv)\big)\Big)^{\ell}\Big)^{1/\ell}\bigg|\\
&=\xi^{1/\ell-1}\ell^{-1}\bigg|\Prob\big(\bM_{m,1}\leq \bF_{m}^{\leftarrow}(\bv)\big)-\Big(\Prob\big(\bM_{\ang{m/\ell},1}\leq \bF_{m}^{\leftarrow}(\bv)\big)\Big)^{\ell}\bigg|,
\end{aligned}
\end{equation}
where $\xi=\xi_{m, \ell}$ is a quantity between $\Prob\big(\bM_{m,1}\leq \bF_{m}^{\leftarrow}(\bv)\big)$ and $(\Prob\big(\bM_{\ang{m/\ell},1}\leq \bF_{m}^{\leftarrow}(\bv)\big))^{\ell}$.
Since $\xi\geq C_{m}(\bv)-|C_{m}(\bv)-\xi|$, by Assumption \ref{assump:copula_convergence} and \eqref{eq:pre_Taylor}, 
\begin{align*}
    \min_{\ell=1,\dots,L}\xi&\geq \min_{\ell=1,\dots,L}\Big(C_{m}(\bv)-|C_{m}(\bv)-\xi|\Big)
    = C_{m}(\bv)+\min_{\ell=1,\dots,L}\Big(-|C_{m}(\bv)-\xi|\Big)\\
    &\geq
    C_{m}(\bv)-\max_{\ell=1,\dots,L}\bigg|\Prob\big(\bM_{m,1}\leq \bF_{m}^{\leftarrow}(\bv)\big)-\Big(\Prob\big(\bM_{\ang{m/\ell},1}\leq \bF_{m}^{\leftarrow}(\bv)\big)\Big)^{\ell}\bigg|=C_{\infty}(\bv)+o(1).
\end{align*}
Hence, noting that $C_{\infty}(\bv)>0$, we have $\max_{\ell=1,\dots,L}\xi^{1/\ell-1}\leq \max_{\ell=1,\dots,L}\xi^{-1}=(\min_{\ell=1,\dots,L}\xi)^{-1}=O(1)$. Lemma \ref{lemma:one_over_multiplier} follows from \eqref{eq:pre_Taylor} and \eqref{eq:Taylor}.  
\end{proof}

\begin{Lemma}\label{lemma:fraction_multiplier}
Let $\bv\in (0,1]^{d}$, $\cX=[x_\wedge, x_{\vee}] \subset (0,\infty)$ a compact interval
%$\tA=[a_{\wedge}/2,a_{\vee}]$, 
and 
\begin{equation}\label{eq:mesh_width}
J=J_{m}=\ang{\frac{1+\varrho}{2+\varrho}\log_{2}m-\log_{2}\log m}. 
\end{equation}
Then 
\[
\max_{x\in \bbN/2^{J}\cap \cX}\bigg|\Prob\big(\bM_{\ang{mx},1}\leq \bF_{m}^{\leftarrow}(\bv)\big)-\Big(\Prob\big(\bM_{m,1}\leq \bF_{m}^{\leftarrow}(\bv)\big)\Big)^{x}\bigg|=O\Big(\log(m)m^{-(1+\varrho)/(2+\varrho)}\Big). 
\]
\end{Lemma}
\begin{proof}[Proof of Lemma \ref{lemma:fraction_multiplier}]
Let $J_{\vee} > 1$ be an integer such that $x_{\vee}<2^{J_{\vee}}$ and note that, for all sufficiently large $m$, $J_{m} > 0$ so that $2^{J_{\vee}-1} > 2^{-J_m}$.  Without loss of generality we only consider such $m$ from now on. Each $x \in \bbN/2^{J}\cap \cX$ has a dyadic expansion of the form $x=\sum_{j=-J}^{J_{\vee}-1}2^{j}r_x(j)$ with $r_x(j)\in \{0,1\}$. Let $x_{i}=\sum_{j=i}^{J_{\vee}-1}2^{j}r_x(j), i=-J,\dots,J_{\vee}-1$ and define $x_{J_{\vee}} = 0$.
Then
\[
\Prob\big(\bM_{\ang{mx},1}\leq \bF_{m}^{\leftarrow}(\bv)\big)-\Big(\Prob\big(\bM_{m,1}\leq \bF_{m}^{\leftarrow}(\bv)\big)\Big)^{x}=Q_{1}+Q_{2}+Q_{3}, 
\]
where
\begin{align*}
Q_{1} &= \Prob\big(\bM_{\ang{mx},1}\leq \bF_{m}^{\leftarrow}(\bv)\big) 
- \prod_{ {-J \leq j \leq J_{\vee}-1 \atop r_x(j)=1}} \Prob\Big(\bM_{\ang{mx_{j}}-\ang{mx_{j+1}},1}\leq \bF_{m}^{\leftarrow}(\bv)\Big),
\\
Q_{2}& = \prod_{ {-J \leq j \leq J_{\vee}-1 \atop r_x(j)=1}} \Prob\Big(\bM_{\ang{mx_{j}}-\ang{mx_{j+1}},1}\leq \bF_{m}^{\leftarrow}(\bv)\Big) - \prod_{ {-J \leq j \leq J_{\vee}-1 \atop r_x(j)=1}} \Prob\Big(\bM_{\ang{mx_{j}-mx_{j+1}},1}\leq \bF_{m}^{\leftarrow}(\bv)\Big),
\\
Q_{3}&= \prod_{ {-J \leq j \leq J_{\vee}-1 \atop r_x(j)=1}} \Prob\Big(\bM_{\ang{mx_{j}-mx_{j+1}},1}\leq \bF_{m}^{\leftarrow}(\bv)\Big) -\Big(\Prob\big(\bM_{m,1}\leq \bF_{m}^{\leftarrow}(\bv)\big)\Big)^{x}.
\end{align*}
Notice that $2^{J}$ has an order of $(\log m)^{-1}m^{(1+\varrho)/(2+\varrho)}$. Applying Lemma \ref{lemma:separate_block} with $\Delta=\ang{m^{1/(2+\varrho)}}$ and  $b_{m,j}=\ang{mx_{J_{\vee}-j}}$, we get 
\[
    \max_{x\in \bbN/2^{J}\cap \cX}|Q_{1}|=O\Big( (J_{\vee}+J)m^{-(1+\varrho)/(2+\varrho)}\Big)=O\Big(\log(m) m^{-(1+\varrho)/(2+\varrho)}\Big).
\]
Next, since $|\prod_i a_i - \prod_i b_i| \leq \sum_i |a_i - b_i|$ for $a_i,b_i \in [0,1]$, an application of Lemma \ref{lemma:trans_block_length_limit} yields
\begin{multline*}
\max_{x\in \bbN/2^{J}\cap \cX}|Q_{2}|\leq (J_{\vee}+J)\Prob\Big(\bM_{1,1} \nleq F_{m}^{\leftarrow}(\bv)\Big)
\\
\leq (J_{\vee}+J) \Prob\Big(\bM_{\ang{m^{1/(2+\varrho)}},1}\nleq F_{m}^{\leftarrow}(\bv)\Big)=O\Big(  \log(m)m^{-(1+\varrho)/(2+\varrho)}\Big).
\end{multline*}
Finally, note that
\begin{multline*}
\Big(\Prob\big(\bM_{m,1}\leq \bF_{m}^{\leftarrow}(\bv)\big)\Big)^{x}=\Big(\Prob\big(\bM_{m,1}\leq \bF_{m}^{\leftarrow}(\bv)\big)\Big)^{x_{-J}}
= \prod_{ {-J \leq j \leq J_{\vee}-1 \atop r_x(j)=1}} \Big\{ \Big(\Prob\big(\bM_{m,1}\leq \bF_{m}^{\leftarrow}(\bv)\big)\Big)^{x_{j}-x_{j+1}}  \Big\}
\end{multline*}
Since$|\prod_i a_i - \prod_i b_i| \leq \sum_i |a_i - b_i|$ for $a_i,b_i \in [0,1]$, applying Lemma \ref{lemma:integer_multiplier} and Lemma \ref{lemma:one_over_multiplier} (recall that $2^J$ is of the order $(\log m)^{-1}m^{(1+\varrho)/(2+\varrho)}$),  we obtain
\[
\max_{x\in \bbN/2^{J}\cap \cX}|Q_{3}|=O\Big( (J_{\vee}+J)m^{-(1+\varrho)/(2+\varrho)}\Big)=O\Big(\log(m) m^{-(1+\varrho)/(2+\varrho)}\Big).
\]
This implies Lemma \ref{lemma:fraction_multiplier}.
\end{proof}

\begin{proof}[Proof of Lemma~\ref{rem:raiid}]
Choose $x_\wedge, x_\vee \in (0,\infty)$ such that $\cX \subset [x_\wedge, x_\vee]$. Observe that, uniformly in $x \in \cX$, we have $\ang{mx}/m - x = O(1/m)$ and $m/\ang{mx} - 1/x = O(1/m)$. Hence, for any fixed $\bu \in (0,1]^d$, it holds that $\sup_{x \in \cX} \|\bu^{1/x} - \bu^{m/\ang{mx}}\|_1 = O(1/m)$. By Lipschitz continuity of $C_m$ for every $m$ with Lipschitz constant $1$ (note that every $C_m$ is a copula), we thus have $\sup_{x \in \cX} |C_m(\bu^{1/x}) - C_m(\bu^{m/\ang{mx}})| = O(1/m)$. Finally, since $\bu \in (0,1]^d$, it follows that $\inf_{x \in \cX} \bu^{1/x} \ge \bu^{1/x_\vee} \in (0,\infty]^d$ where the infimum and inequality are understood component-wise. Hence 
\[
\lim_{m \to \infty} \inf_{x \in \cX} C_m(\bu^{1/x}) \geq \lim_{m \to \infty}  C_m(\bu^{1/x_{\vee}}) = C_\infty(\bu^{1/x_{\vee}}) > 0.
\]
Given this bound, a Taylor expansion shows that
\[
\sup_{x \in \cX} \Big|C_m(\bu^{1/x})^x - C_m(\bu^{m/\ang{mx}})^{\ang{mx}/m}\Big| = O(1/m) = o(m^{-(1+\varrho)/(2+\varrho)}\log m).
\]
Hence, in order to prove Lemma~\ref{rem:raiid}, it suffices to show that
\[
\sup_{x \in \cX} \Big| C_{\ang{mx}}(\bu)-\Big(C_{m}(\bu^{1/x})\Big)^{x} \Big| = O(m^{-(1+\varrho)/(2+\varrho)}\log m)
\]

For that purpose, write
\[
C_{\ang{mx}}(\bu)-\Big(C_{m}(\bu^{1/x})\Big)^{x}=G_{1}+G_{2},
\]
where 
\begin{align*}
    G_{1}&=\Prob\big(\bM_{\ang{mx},1}\leq \bF_{\ang{mx}}^{\leftarrow}(\bu)\big)-\Prob\big(\bM_{\ang{mx},1}\leq \bF_{m}^{\leftarrow}(\bu^{1/x})\big),\\
    G_{2}&=\Prob\big(\bM_{\ang{mx},1}\leq \bF_{m}^{\leftarrow}(\bu^{1/x})\big)-\Big(\Prob\big(\bM_{m,1}\leq \bF_{m}^{\leftarrow}(\bu^{1/x})\big)\Big)^{x}.
\end{align*}
We start with $G_{2}$; and note that all subsequent $O$-terms are uniform in $x\in \cX$. Define $J=J_m$ as in \eqref{eq:mesh_width}. Notice $G_{2}=G_{2,1}+G_{2,2}+G_{2,3}$, where
\begin{align*}
    G_{2,1}&=\Prob\big(\bM_{\ang{mx},1}\leq \bF_{m}^{\leftarrow}(\bu^{1/x})\big)-\Prob\big(\bM_{\ang{m\ang{x2^{J}}2^{-J}},1}\leq \bF_{m}^{\leftarrow}(\bu^{1/x})\big)\\
    G_{2,2}&=\Prob\big(\bM_{\ang{m\ang{x2^{J}}2^{-J}},1}\leq \bF_{m}^{\leftarrow}(\bu^{1/x})\big)-\Big(\Prob\big(\bM_{m,1}\leq \bF_{m}^{\leftarrow}(\bu^{1/x})\big)\Big)^{\ang{x2^{J}}2^{-J}}\\
    G_{2,3}&=\Big(\Prob\big(\bM_{m,1}\leq \bF_{m}^{\leftarrow}(\bu^{1/x})\big)\Big)^{\ang{x2^{J}}2^{-J}}-\Big(\Prob\big(\bM_{m,1}\leq \bF_{m}^{\leftarrow}(\bu^{1/x})\big)\Big)^{x}.
\end{align*}
Notice $2^{J}$ has an order of $(\log m)^{-1}m^{(1+\varrho)/(2+\varrho)}$. By Lemma \ref{lemma:trans_block_length_limit},
\[
    |G_{2,1}|=O (2^{-J})=O\Big((\log m) m^{-(1+\varrho)/(2+\varrho)}\Big).
\]
For large enough $m$, $\ang{x2^{J}}2^{-J}\geq x_{\wedge}/2$. By Lemma \ref{lemma:fraction_multiplier}, applied with $\tilde \cX = [x_{\wedge}/2, x_{\vee}]$,
\[
    |G_{2,2}|=O\Big((\log m) m^{-(1+\varrho)/(2+\varrho)}\Big).
\]
By the Mean Value Theorem and Assumption \ref{assump:copula_convergence},
\[
    |G_{2,3}|\leq 2^{-J}\log \Big(\Prob\big(\bM_{m,1}\leq \bF_{m}^{\leftarrow}(\bu^{1/x})\big)\Big)=2^{-J}\log \Big(C_{m}(\bu^{1/x})\Big)=O\Big((\log m) m^{-(1+\varrho)/(2+\varrho)}\Big).
\]
Hence, 
\begin{equation}\label{eq:approx_to_iid_2}
    |G_{2}|=O\Big((\log m) m^{-(1+\varrho)/(2+\varrho)}\Big).
\end{equation}

We next turn to $G_{1}$. Note that plugging-in $(1,\dots, 1,u_j,1,\dots,1)$ instead of $\bu$ into the definition of $G_2$, \eqref{eq:approx_to_iid_2} yields
\begin{equation}\label{eq:approx_to_iid_2_d=1}
    \bigg|\Prob\big(M_{\ang{mx},1,j}\leq F_{m,j}^{\leftarrow}(u_{j}^{1/x})\big)-\Big(\Prob\big(M_{m,1,j}\leq F_{m,j}^{\leftarrow}(u_{j}^{1/x})\big)\Big)^{x}\bigg|=O\Big((\log m) m^{-(1+\varrho)/(2+\varrho)}\Big),
\end{equation}
uniformly in  $j=1, \dots, d$.
We further have
\begin{equation}\label{eq:approx_to_iid_1_preliminary}
\begin{aligned}
|G_{1}|&=\Prob\Big(\bM_{\ang{ma},1}\leq\bF_{m}^{\leftarrow}(\bu^{1/x}) \vee \bF_{\ang{ma}}^{\leftarrow}(\bu), \bM_{\ang{mx},1}\nleq \bF_{m}^{\leftarrow}(\bu^{1/x}) \wedge \bF_{\ang{mx}}^{\leftarrow}(\bu)\Big)
\\
&\leq \sum_{j=1}^{d}\Prob\big(M_{\ang{mx},1,j}\leq F_{m,j}^{\leftarrow}(u_{j}^{1/x})\vee F_{\ang{mx},j}^{\leftarrow}(u_{j}), M_{\ang{mx},1,j}>F_{m,j}^{\leftarrow}(u_{j}^{1/x}) \wedge F_{\ang{mx},j}^{\leftarrow}(u_{j})\big)
\\
&= \sum_{j=1}^{d}\bigg|\Prob\big(M_{\ang{mx},1,j}\leq F_{m,j}^{\leftarrow}(u_{j}^{1/x})\big)-\Prob\big(M_{\ang{mx},1,j} \leq F_{\ang{mx},j}^{\leftarrow}(u_{j})\big)\bigg|
\\
&= \sum_{j=1}^{d}\bigg|\Prob\big(M_{\ang{mx},1,j}\leq F_{m,j}^{\leftarrow}(u_{j}^{1/x})\big)-u_{j}\bigg|
\\
&= \sum_{j=1}^{d}\bigg|\Prob\big(M_{\ang{mx},1,j}\leq F_{m,j}^{\leftarrow}(u_{j}^{1/x})\big)-\Big(\Prob\big(M_{m,1,j} \leq F_{m,j}^{\leftarrow}(u_{j}^{1/x})\big)\Big)^{x}\bigg|.
\end{aligned}    
\end{equation}
By \eqref{eq:approx_to_iid_2_d=1} and \eqref{eq:approx_to_iid_1_preliminary},
\begin{equation}\label{eq:approx_to_iid_1}
    |G_{1}|=O\Big((\log m) m^{-(1+\varrho)/(2+\varrho)}\Big).
\end{equation}
Lemma~\ref{rem:raiid} follows from \eqref{eq:approx_to_iid_2} and \eqref{eq:approx_to_iid_1}.
\end{proof}

\subsection{Proofs for Sections~\ref{sec:simpagg}--\ref{sec:estrho}}

\begin{proof}[Proof of Proposition~\ref{prop:simpleagg}]
Observe the decomposition 
\begin{align*}
&\hat C_{n,(M,w)}^{\mathrm{agg}}(\bu) - C_\infty(\bu) - \sum_{k \in M_n} w_{n,k} \{C_{k}(\bu) - C_\infty(\bu)\} 
= 
\sum_{k \in M_n} w_{n,k} \{\hat C_{n,k}(\bu) - C_k(\bu)\} 
\end{align*}
which can be rewritten as
\begin{multline*}
%\frac{1}{m}\sum_{k \in M_n} w_{n,k} \{\hat C_{n,k}(\bu) - C_k(\bu)\} 
%&= 
\sqrt{\frac{m}n}\frac{1}{m} \sum_{k: k/m \in A} \{ f(k/m) + o(1) \} \widehat \bbC^{\lozenge}_{n,m}(\bu,k/m)\\
= 
\sqrt{\frac{m}n}\frac{1}{m} \sum_{k: k/m \in A} f(k/m)\widehat \bbC^{\lozenge}_{n,m}(\bu,k/m) + O_\Prob(\sqrt{m/n}).
\end{multline*}
Next, by continuity of $f$ and since additionally
\begin{equation*} %\label{eq:unifequi}
\sup_{\bu \in [0,1]^d} \sup_{k: k/m \in A} \sup_{a \in [k/m,(k+1)/m]} |\widehat \bbC^{\lozenge}_{n,m}(\bu,k/m)-\widehat \bbC^{\lozenge}_{n,m}(\bu,a)| = O_\Prob(1)
\end{equation*}
by asymptotic equicontinuity of $\bbC^{\lozenge}_{n,m}$,
it follows that 
\[
\sqrt{m/n}\frac{1}{m} \sum_{k: k/m \in A} f(k/m)\widehat \bbC^{\lozenge}_{n,m}(\bu,k/m) = \sqrt{m/n} \int_A f(a)\widehat \bbC^{\lozenge}_{n,m}(\bu,a) \diff a + O_\Prob(\sqrt{m/n}) 
\]
uniformly in $\bu \in [0,1]^d$. By similar but simpler arguments we also have  
\[
\frac{1}{m}\sum_{k \in M_n} w_{n,k} = \int_A f(a) \diff a + o(1)
\]
and the claim follows by an application of the continuous mapping theorem.
\end{proof}
 
\begin{proof}[Proof of Proposition~\ref{prop:bcnai} and~\ref{prop:bcagg}]
Begin by observing that 
\begin{multline*}
\sqrt{\frac{n}{m}}\Big(\hat C_{n,m}(\bu) - \frac{\hat C_{n,\ang{ma}}(\bu) - \hat C_{n,m}(\bu)}{(\ang{ma}/m)^{\rho_\ca} - 1} - \Big\{C_m(\bu)  - \frac{C_{\ang{ma}}(\bu) - C_m(\bu)}{(\ang{ma}/m)^{\rho_\ca} - 1} \Big\} \Big)
\\
= \widehat \bbC^{\lozenge}_{n,m}(\bu,1) - \frac{\widehat \bbC^{\lozenge}_{n,m}(\bu,a) - \widehat \bbC^{\lozenge}_{n,m}(\bu,1)}{(\ang{ma}/m)^{\rho_\ca} - 1}. 
\end{multline*}
Since $\ang{ma}/m \to a$ uniformly in $a \in A$ and by a simple application of the continuous mapping theorem, this implies the following process convergence result
\begin{multline*} %\label{eq:bcnaiproc}
\Big[ \sqrt{\frac{n}{m}}\Big(\hat C_m(\bu) - \frac{\hat C_{\ang{ma}}(\bu) - \hat C_m(\bu)}{(\ang{ma}/m)^{\rho_\ca} - 1} - \Big\{C_m(\bu)  - \frac{C_{\ang{ma}}(\bu) - C_m(\bu)}{(\ang{ma}/m)^{\rho_\ca} - 1} \Big\} \Big) \Big]_{a \in A, \bu \in [0,1]^d}
\\
\dto \Big[\bbC^{\lozenge}(\bu,1) - \frac{\bbC^{\lozenge}(\bu,a) - \bbC^{\lozenge}(\bu,1)}{a^{\rho_\ca} - 1}\Big]_{a \in A, \bu \in [0,1]^d} \quad in~~\ell^\infty(A\times [0,1]^d).
\end{multline*}
The expansion for 
\begin{multline*}
%B_{n,m}^{bc,nai}(\bu) := 
C_m(\bu)  - \frac{C_{\ang{ma}}(\bu) - C_m(\bu)}{(\ang{ma}/m)^{\rho_\ca} - 1} - C_\infty(\bu)
\\
= \Big\{\ca(m) r_a(m) \frac{S(\bm u)}{a^{\rho_\ca}-1} + \ca(m) \cb(m) \frac{1-a^{\rho_\cb}}{1-a^{-\rho_\ca}} T(\bm u) \Big\} + \ca(m)o\Big(\cb(m)+|r_a(m)|\Big)
\end{multline*}
with remainder terms holding uniformly in $a \in A$ follows from regular variation of $\ca, \cb$ and~\eqref{eqn:thirdor} after some tedious but straightforward computations.
The claim regarding $\hat{C}_{n,(m,m')}^{\mathrm{bc,nai}}$ in Proposition~\ref{prop:bcnai} follows immediately upon setting $A = \{a\}$ in the previous display. The proof for $\hat{C}_{n,(M,w)}^{\mathrm{bc,agg}}$ in Proposition~\ref{prop:bcagg} is very similar to the proof of Proposition~\ref{prop:simpleagg}; details are omitted for the sake of brevity.

Next, concerning the claim regarding $\check{C}_{n,(m,m')}^{\mathrm{bc,nai}}$ in Proposition~\ref{prop:bcnai}, observe the decomposition
\begin{align*}
\check{C}_{n,(m,m')}^{\mathrm{bc,nai}}(\bu)  - \hat{C}_{n,(m,m')}^{\mathrm{bc,nai}} (\bu) 
&=
\frac{\hat C_{n,m'}(\bu) - \hat C_{n,m}(\bu)}{(m'/m)^{\rho_\ca} - 1} - \frac{\hat C_{n,m'}(\bu) - \hat C_{n,m}(\bu)}{(m'/m)^{\hat \rho_\ca} - 1}  \\
&= 
\frac{\hat C_{n,m'}(\bu) - \hat C_{n,m}(\bu)}{((m'/m)^{\rho_\ca} - 1)((m'/m)^{\hat\rho_\ca} - 1)}\Big((m'/m)^{\hat \rho_\ca} - (m'/m)^{\rho_\ca} \Big). 
\end{align*}
We have 
\begin{align*}
\hat C_{n,m'}(\bu) - \hat C_{n,m}(\bu) &= C_{m'}(\bu) - C_{m}(\bu) + \sqrt{m/n}\{\widehat \bbC^{\lozenge}_{n,m}(\bu,m'/m) - \widehat \bbC^{\lozenge}_{n,m}(\bu,1)\} 
\\
&= O(\ca(m)) + O_\Prob(\sqrt{m/n}).
\end{align*}
Now the final assertion of Proposition~\ref{prop:bcnai} follows directly since $m'/m \to a \neq 0$ and hence $(m'/m)^{\hat \rho_\ca} - (m'/m)^{\rho_\ca} = O_\Prob(|\hat \rho_\ca-\rho_\ca|)$. The final assertion of Proposition~\ref{prop:bcagg} follows after some simple algebra upon observing that the assertion regarding $\check{C}_{n,(m,m')}^{\mathrm{bc,nai}}$  holds uniformly over $m' \in M_n$. 
\end{proof}

\begin{proof}[Proof of Proposition~\ref{prop:bcreg}]
Begin by observing that, for $v \in \{0,1,2\}$,
\[
\mu_{v,n} = \sum_{k \in M_n} w_{n,k} (k/m)^{v \rho_\ca} = \int_A f(a) a^{v\rho_\ca} da + o(1) = \kappa_v + o(1).
\] 
Moreover, for $v = 0,1$
\[
\sum_{k \in M_n} w_{n,k} (k/m)^{v \rho_\ca} \hat C_{n,k}(\bu)  = 
\sum_{k \in M_n} w_{n,k} (k/m)^{v \rho_\ca} C_k(\bu) + \sum_{k \in M_n} w_{n,k} (k/m)^{v \rho_\ca} \{\hat C_{n,k}(\bu) - C_k(\bu)\}.
\]
Now, under the imposed assumptions of Theorem~\ref{theo:estmar},
\[
\sqrt{\frac{n}{m}}\left(
\begin{array}{c}
\sum_{k \in M_n} w_{n,k} \{\hat C_{n,k}(\cdot) - C_k(\cdot)\}
\\
\sum_{k \in M_n} w_{n,k} (k/m)^{\rho_\ca} \{\hat C_{n,k}(\cdot) - C_k(\cdot)\}
\end{array}
\right)
\dto 
\left(
\begin{array}{c}
\int_A f(a) \widehat\bbC^{\lozenge}(\cdot,a) \diff a
\\
\int_A f(a) a^{\rho_\ca}\widehat\bbC^{\lozenge}(\cdot,a) \diff a
\end{array}
\right) 
= \left(
\begin{array}{c}
T_0(\cdot)
\\
T_1(\cdot)
\end{array}
\right),
\]
which follows by arguments that are similar to those given in the proof of Proposition~\ref{prop:simpleagg}. Next note that for $k \in M_n$ we have 
\[
\ca(k) = \ca(m)\frac{\ca(k)}{\ca(m)} 
= 
\ca(m)(k/m)^{\rho_\ca} + \ca(m)\Big(\frac{\ca(k)}{\ca(m)} - \frac{k^{\rho_\ca}}{m^{\rho_\ca}} \Big) 
= 
\ca(m)(k/m)^{\rho_\ca} + O(\ca(m)d_{m,n}(A))
\]
where 
\[
d_{m,n}(A) := \sup_{a \in A}\Big|\frac{\ca(\ang{ma})}{\ca(m)} - (\ang{ma}/m)^{\rho_\ca}\Big|.
\]
From this we obtain the expansion (holding uniformly in $\bu$)
\begin{align*}
\sum_{k \in M_n} w_{n,k} (k/m)^{v \rho_\ca} C_k(\bu) 
&= 
\sum_{k \in M_n} w_{n,k} (k/m)^{v \rho_\ca} \{C_\infty(\bu) + \ca(k)S(\bu) + \ca(k)\cb(k)T(\bu)+o(\ca(k)\cb(k))\}
\\
&= \mu_{v,n}C_\infty(\bu) + \mu_{v+1,n} \ca(m) S(\bu) \\
&\mathph{=}+ 
\sum_{k \in M_n} w_{n,k} (k/m)^{v \rho_\ca}\ca(k)\cb(k)T(\bu)\\
&\mathph{=}+
\sum_{k \in M_n} w_{n,k} (k/m)^{v \rho_\ca}\ca(m)\Big(\frac{\ca(k)}{\ca(m)} - \frac{k^{\rho_\ca}}{m^{\rho_\ca}} \Big)S(\bu)+o(r(m))\\
&=\mu_{v,n}C_\infty(\bu) + \mu_{v+1,n} \ca(m) S(\bu) + \cT_{m,v}(\bu)+o(r(m));
\end{align*}
recall that 
\begin{equation*} %\label{eq:rmn}
r(m) = \ca(m)O\Big(\cb(m) + d_{m,n}(A)\Big)
\end{equation*}
by \eqref{eq:rmn}. 

Next observe the representation
\begin{align*}
\Bigg(
\begin{array}{cc} 
\mu_{0,n} & \mu_{1,n} \\ 
\mu_{1,n} & \mu_{2,n}  
\end{array}
\Bigg)^{-1}
 =& \frac{1}{\mu_{2,n}\mu_{0,n} - \mu_{1,n}^2} \Bigg(
\begin{array}{cc} 
\mu_{2,n} & -\mu_{1,n} \\ 
-\mu_{1,n} & \mu_{0,n}  
\end{array}
\Bigg)
\\
=& \frac{1}{\kappa_{2}\kappa_{0} - \kappa_{1}^2 + o(1)}
\Bigg(
\begin{array}{cc} 
\kappa_{2} + o(1) & -\kappa_{1} + o(1) \\ 
-\kappa_{1}+ o(1) & \kappa_{0} + o(1)  
\end{array}
\Bigg).
\end{align*}
Using the latter two representations, we obtain
\begin{align*}
&\Bigg(
\begin{array}{cc} 
\mu_{0,n} & \mu_{1,n} \\ 
\mu_{1,n} & \mu_{2,n}  
\end{array}
\Bigg)^{-1}
\Bigg(
\begin{array}{c} 
\sum_{k \in M_n} w_{n,k} C_k(\bu) \\ 
\sum_{k \in M_n} w_{n,k} (k/m)^{\rho_\ca} C_k(\bu) 
\end{array}
\Bigg)
\\
=& \Bigg(
\begin{array}{cc} 
\mu_{0,n} & \mu_{1,n} \\ 
\mu_{1,n} & \mu_{2,n}  
\end{array}
\Bigg)^{-1}
\Bigg\{
\Bigg(
\begin{array}{c} 
\mu_{0,n} \\ 
\mu_{1,n}
\end{array}
\Bigg) C_\infty(\bu)
+ \Bigg(
\begin{array}{c} 
\mu_{1,n}\\ 
\mu_{2,n}  
\end{array}
\Bigg) \ca(m) S(\bu)
+ 
\Bigg(
\begin{array}{c} 
\cT_{m,0}(\bu) \\ 
\cT_{m,1}(\bu)
\end{array}
\Bigg) 
+ 
\Bigg(
\begin{array}{c} 
o(r(m)) \\ 
o(r(m))
\end{array}
\Bigg) 
\Bigg\} 
\\
 =&\Bigg(\begin{array}{c} 
C_\infty(\bu)  \\ 
B_m(\bu)
\end{array}\Bigg)
+
\Bigg(
\begin{array}{cc} 
\mu_{0,n} & \mu_{1,n} \\ 
\mu_{1,n} & \mu_{2,n}  
\end{array}
\Bigg)^{-1}
\Bigg(\begin{array}{c} 
\cT_{m,0}(\bu) \\ 
\cT_{m,1}(\bu) 
\end{array}\Bigg)
+
\Bigg(\begin{array}{c} 
o(r(m)) \\ 
o(r(m))
\end{array}\Bigg)
.
\end{align*}
Putting together all of the above results, all claims in Proposition~\ref{prop:bcreg} except for the one involving $\check C_{n,(M,w)}^{\mathrm{bc,reg}}$ follow. 

For the missing proof regarding $\check C_{n,(M,w)}^{\mathrm{bc,reg}}$, we begin by noting that, similarly as above,   
\[
\hat \mu_{v,n} := \sum_{k \in M_n} w_{n,k} (k/m)^{v \hat\rho_\ca} = \int_A f(a) a^{v\rho_\ca} da + o_\Prob(1)
\] 
(also note that $\hat \mu_{0,n}=\mu_{0,n}$) and
\begin{multline*}
\sqrt{\frac{n}{m}}\left(
\begin{array}{c}
\sum_{k \in M_n} w_{n,k} \{\hat C_{n,k}(\cdot) - C_k(\cdot)\}
\\
\sum_{k \in M_n} w_{n,k} (k/m)^{\hat \rho_\ca} \{\hat C_{n,k}(\cdot) - C_k(\cdot)\}
\end{array}
\right)
\\
= \sqrt{\frac{n}{m}}\left(
\begin{array}{c}
\sum_{k \in M_n} w_{n,k} \{\hat C_{n,k}(\cdot) - C_k(\cdot)\}
\\
\sum_{k \in M_n} w_{n,k} (k/m)^{\rho_\ca} \{\hat C_{n,k}(\cdot) - C_k(\cdot)\}
\end{array}
\right)
+ O_\Prob(|\hat\rho_\ca - \rho_\ca|).
\end{multline*}
We also have
\[
\sum_{k \in M_n} w_{n,k} (k/m)^{v \hat \rho_\ca}(k/m)^{\rho_\ca} = \hat\mu_{v+1,n} + O_\Prob(|\hat \rho_\ca - \rho_\ca|).
\]
This implies
\begin{align*}
\sum_{k \in M_n} w_{n,k} (k/m)^{v \hat  \rho_\ca} C_k(\bu) 
&= 
\sum_{k \in M_n} w_{n,k} (k/m)^{v \hat \rho_\ca} \{C_\infty(\bu) + \ca(k)S(\bu) + O(\ca(k)\cb(k))\}
\\
&= \hat \mu_{v,n}C_\infty(\bu) + \sum_{k \in M_n} w_{n,k} (k/m)^{v \hat \rho_\ca}(k/m)^{\rho_\ca} \ca(m) S(\bu) + O_\Prob(r(m)) 
\\
&= \hat \mu_{v,n}C_\infty(\bu) +  \hat \mu_{v+1,n} \ca(m) S(\bu) + O_\Prob(r(m)+\ca(m)|\hat \rho_\ca - \rho_\ca|).
\end{align*}
where $r(m)$ is defined in~\eqref{eq:rmn}. Moreover
\begin{align*}
\Bigg(
\begin{array}{cc} 
\hat \mu_{0,n} & \hat \mu_{1,n} \\ 
\hat \mu_{1,n} & \hat \mu_{2,n}  
\end{array}
\Bigg)^{-1}
&= \frac{1}{\hat \mu_{2,n} \hat \mu_{0,n} -  \hat \mu_{1,n}^2} \Bigg(
\begin{array}{cc} 
\hat \mu_{2,n} & - \hat \mu_{1,n} \\ 
-\hat \mu_{1,n} & \hat \mu_{0,n}  
\end{array}
\Bigg)
\\
&= \frac{1}{\kappa_{2}\kappa_{0} - \kappa_{1}^2 }
\Bigg(
\begin{array}{cc} 
\kappa_{2} & -\kappa_{1}  \\ 
-\kappa_{1} & \kappa_{0}  
\end{array}
\Bigg) + o_\Prob(1).
\end{align*}
Using this representation yields
\begin{align*}
&\Bigg(
\begin{array}{cc} 
\hat \mu_{0,n} & \hat\mu_{1,n} \\ 
\hat\mu_{1,n} & \hat\mu_{2,n}  
\end{array}
\Bigg)^{-1}
\Bigg(
\begin{array}{c} 
\sum_{k \in M_n} w_{n,k} C_k(\bu) \\ 
\sum_{k \in M_n} w_{n,k} (k/m)^{\hat\rho_\ca} C_k(\bu) 
\end{array}
\Bigg)
\\
=& \Bigg(
\begin{array}{cc} 
\hat \mu_{0,n} & \hat\mu_{1,n} \\ 
\hat\mu_{1,n} & \hat\mu_{2,n}  
\end{array}
\Bigg)^{-1}
\Bigg\{
\Bigg(
\begin{array}{c} 
\hat\mu_{0,n} \\ 
\hat\mu_{1,n}
\end{array}
\Bigg) C_\infty(\bu)
+ \Bigg(
\begin{array}{c} 
\hat\mu_{1,n}\\ 
\hat\mu_{2,n}  
\end{array}
\Bigg) \ca(m) S(\bu)
+ 
\Bigg(
\begin{array}{c} 
O_\Prob(r(m)+\ca(m)|\hat \rho_\ca - \rho_\ca|) \\ 
O_\Prob(r(m)+\ca(m)|\hat \rho_\ca - \rho_\ca|)
\end{array}
\Bigg) 
\Bigg\} 
\\
 =& \Big(\begin{array}{c} 
C_\infty(\bu) + O_\Prob(r(m)+\ca(m)|\hat \rho_\ca - \rho_\ca|) \\ 
\ca(m) S(\bu) + O_\Prob(r(m)+\ca(m)|\hat \rho_\ca - \rho_\ca|))
\end{array}\Big).
\end{align*}
Combining the derivations so far and the corresponding expansions for $\hat C_{n,(M,w)}^{\mathrm{bc,reg}}$ the result follows.
\end{proof}

\begin{proof}[Proof of Proposition~\ref{prop:hatrhonai}]
Define
\[
\Delta_{m_\rho,n}(\bu,a) := 
\sqrt{\frac{m_\rho}{n}}\Big(\widehat \bbC^{\lozenge}_{n,m_\rho}(\bu,a) - \widehat \bbC^{\lozenge}_{n,m_\rho}(\bu,1)\Big).
\]
Recall the definition of $r_a$ in \eqref{eq:ra}. We have the following expansion
{\small\begin{align*}
&\frac{\hat C_{n,\ang{m_\rho a^2}}(\bu) - \hat C_{n,m_\rho}(\bu)}{\hat C_{n,\ang{{m_\rho a}}}(\bu) - \hat C_{n,m_\rho}(\bu)}
\\
=\ & \frac{\Delta_{m_\rho,n}(\bu,a^2) + C_{\ang{m_\rho a^2}}(\bu) - C_{m_\rho}(\bu)}{\Delta_{m_\rho,n}(\bu,a) + C_{\ang{m_\rho a}}(\bu) - C_{m_\rho}(\bu)}
\\
=\ & 
\frac{\frac{1}{\ca(m_\rho)}\Delta_{m_\rho,n}(\bu,a^2) + \Big(\frac{\ca(\ang{m_\rho a^2})}{\ca(m_\rho)}-1\Big)S(\bu) + \cb(m_\rho)(a^{2\rho_\ca + 2\rho_\cb}-1)T(\bu) + o(\cb(m_\rho))}{\frac{1}{\ca(m_\rho)}\Delta_{m_\rho,n}(\bu,a) + \Big(\frac{\ca(\ang{m_\rho a})}{\ca(m_\rho)}-1\Big)S(\bu) + \cb(m_\rho)(a^{\rho_\ca + \rho_\cb}-1)T(\bu) + o(\cb(m_\rho))}
\\
=\ & 
\frac
{\frac{1}{\ca(m_\rho)}\Delta_{m_\rho,n}(\bu,a^2) + \Big(a^{2\rho_\ca}-1\Big)S(\bu) + \cb(m_\rho)(a^{2\rho_\ca + 2\rho_\cb}-1)T(\bu) + O(r_{a^2}(m_\rho)+m_\rho^{-1}) + o(\cb(m_\rho))}
{\frac{1}{\ca(m_\rho)}\Delta_{m_\rho,n}(\bu,a) + \Big(a^{\rho_\ca}-1\Big)S(\bu) + \cb(m_\rho)(a^{\rho_\ca + \rho_\cb}-1)T(\bu) + O(r_a(m_\rho)+m_\rho^{-1}) + o(\cb(m_\rho))}
\\
=\ & 
\frac
{a^{2\rho_\ca}-1 +\frac{1}{\ca(m_\rho)}\frac{\Delta_{m_\rho,n}(\bu,a^2)}{S(\bu)} +  \cb(m_\rho)(a^{2\rho_\ca + 2\rho_\cb}-1)\frac{T(\bu)}{S(\bu)} + O(r_{a^2}(m_\rho)+m_\rho^{-1}) + o(\cb(m_\rho))}{a^{\rho_\ca}-1 +\frac{1}{\ca(m_\rho)}\frac{\Delta_{m_\rho,n}(\bu,a)}{S(\bu)} +  \cb(m_\rho)(a^{\rho_\ca + \rho_\cb}-1)\frac{T(\bu)}{S(\bu)} + O(r_a(m_\rho)+m_\rho^{-1}) + o(\cb(m_\rho))}
\\
=\ & 
a^{\rho_\ca} + 1 + \frac{1}{a^{\rho_\ca}-1}\Big\{\frac{1}{\ca(m_\rho)}\frac{\Delta_{m_\rho,n}(\bu,a^2)}{S(\bu)} +  \cb(m_\rho)(a^{2\rho_\ca + 2\rho_\cb}-1)\frac{T(\bu)}{S(\bu)} \Big\}
\\
& - \frac{a^{\rho_\ca} + 1}{a^{\rho_\ca} - 1}\Big\{\frac{1}{\ca(m_\rho)}\frac{\Delta_{m_\rho,n}(\bu,a)}{S(\bu)} +  \cb(m_\rho)(a^{\rho_\ca + \rho_\cb}-1)\frac{T(\bu)}{S(\bu)}\Big\}
\\
& + O_\Prob\Big( \Big\{\ca(m_\rho) \sqrt{m_\rho/n}\Big\}^2 \Big) + O(r_{a^2}(m_\rho) + r_{a}(m_\rho)+m_\rho^{-1}) + o(\cb(m_\rho)).
\end{align*}
A Taylor expansion of $x\mapsto \log_a(x)$ in the point $x = a^{\rho_\ca}$ yields
\begin{align*}
\hat\rho_{\ca}^{\mathrm{nai}}(a,\bu)  - \rho_\ca = &\log_a\Big(\frac{\hat C_{n,\ang{m_\rho a^2}}(\bu) - \hat C_{n,m_\rho}(\bu)}{\hat C_{n,\ang{m_\rho a^2}}(\bu) - \hat C_{n,m_\rho}(\bu)} - 1 \Big) - \log_a(a^{\rho_\ca})
\\
=\ & 
\frac{1}{a^{\rho_\ca} \log a}\Big[\frac{1}{a^{\rho_\ca}-1}\Big\{\frac{1}{\ca(m_\rho)}\frac{\Delta_{m_\rho,n}(\bu,a^2)}{S(\bu)} +  \cb(m_\rho)(a^{2\rho_\ca + 2\rho_\cb}-1)\frac{T(\bu)}{S(\bu)} \Big\}
\\
& - \frac{a^{\rho_\ca} + 1}{a^{\rho_\ca} - 1}\Big\{\frac{1}{\ca(m_\rho)}\frac{\Delta_{m_\rho,n}(\bu,a)}{S(\bu)} +  \cb(m_\rho)(a^{\rho_\ca + \rho_\cb}-1)\frac{T(\bu)}{S(\bu)}\Big\} \Big]
\\
& +  O_\Prob\Big( \Big\{\ca(m_\rho) \sqrt{m_\rho/n}\Big\}^2 \Big)  + O(r_{a^2}(m_\rho) + r_{a}(m_\rho)+m_\rho^{-1}) + o(\cb(m_\rho))
\end{align*}}
The claim follows from this expansion after rearranging terms. 
\end{proof}

\begin{Lemma}\label{lemma:M_est} Let $U$ be an arbitrary set and let $(\Theta, d)$ be a metric space. Suppose that $\cL \in \ell^\infty(\Theta)$ is a deterministic function and $(\ccL_{n})_{n\in\N}$ is a sequence of random elements in $\ell^\infty(\Theta\times U)$. Assume that $\cL(\cdot)$ has a unique maximizer $\theta_0$ and let
\[
\ctheta_{n}(\bu) \in \argmax_{\theta\in \Theta}\ccL_{n}(\theta;\bu) \quad \forall \, n \in \N,\bu \in U;
\]
note in particular that we do not assume that $\ccL_{n}(\cdot;\bu)$ has a unique maximizer. If
\begin{equation}\label{eq:M_cond_1}
\sup_{\bu\in U}\sup_{\theta \in \Theta}\Big|\ccL_{n}(\theta;\bu)-\cL(\theta)\Big|=\op{1},
\end{equation}
then 
\[
\sup_{\bu\in U}\Big[\cL(\theta_{0})-\ccL_{n}(\ctheta_{n}(\bu);\bu)\Big]\leq \op{1}.
\]
If, additionally, for all $\ep>0$, 
\begin{equation}\label{eq:M_cond_2}
\sup_{\theta: d(\theta,\theta_{0})\geq \ep }\cL(\theta)<\cL(\theta_{0}),    
\end{equation}
then
\[
\sup_{\bu\in U}d(\ctheta_{n}(\bu),\theta_{0}) = \op{1}.
\]
\end{Lemma}
\begin{proof}[Proof of Lemma \ref{lemma:M_est}]
By the definition of $\ctheta_{n}(\bu)$ and \eqref{eq:M_cond_1},
\begin{equation*} 
\begin{aligned}
\mathph{\leq}
\sup_{\bu\in U}\Big[\cL(\theta_{0})-\ccL_{n}(\ctheta_{n}(\bu);\bu)\Big] 
&= \sup_{\bu\in U}\Big[\cL(\theta_{0}) - \ccL_{n}(\theta_{0};\bu) + \ccL_{n}(\theta_{0};\bu) - \ccL_{n}(\ctheta_{n}(\bu);\bu)\Big]\\
&\leq \sup_{\bu\in U}\Big[\cL(\theta_{0})-\ccL_{n}(\theta_{0};\bu)\Big]+\sup_{\bu\in U}\Big[\ccL_{n}(\theta_{0};\bu)-\ccL_{n}(\ctheta_{n}(\bu);\bu)\Big] \\
&\leq \sup_{\bu\in U}\Big[\cL(\theta_{0})-\ccL_{n}(\theta_{0};\bu)\Big]=\op{1}.
\end{aligned}
\end{equation*}
By \eqref{eq:M_cond_2}, for all $\ep>0$, there exists $\eta_{\ep}>0$, such that
\begin{equation}\label{eq:M_result_2}
\inf_{\theta:d(\theta,\theta_{0}) \geq \ep} \Big[\cL(\theta_{0})-\cL(\theta)\Big]>\eta_{\ep}.
\end{equation}
By \eqref{eq:M_result_2} and the definition of $\ctheta_{n}(\cbu)$, for all $\ep>0$,
\begin{align*}
    \Big\{\sup_{\bu\in U}d(\ctheta_{n}(\bu),\theta_{0})>\ep\Big\}
    &\subset\Big\{\exists\, \cbu\in U: d(\ctheta_{n}(\cbu),\theta_{0}) \geq\ep\Big\}\\
    &\subset\Big\{\exists\, \cbu\in U: \cL(\theta_{0})-\cL(\ctheta_{n}(\cbu))>\eta_{\ep}\Big\}\\
    &\subset\Big\{\exists\, \cbu\in U: \ccL_{n}(\theta_{0};\cbu)-\ccL_{n}(\ctheta_{n}(\cbu),\cbu)+2\sup_{\bu\in U}\sup_{\theta\in \Theta}|\ccL_{n}(\theta;\bu)-\cL(\theta)|>\eta_{\ep}\Big\}\\
    &\subset\Big\{2\sup_{\bu\in U}\sup_{\theta\in \Theta}|\ccL_{n}(\theta;\bu)-\cL(\theta)|>\eta_{\ep}\Big\}.
\end{align*}
Hence, by \eqref{eq:M_cond_1},
\[
\Prob\Big(\sup_{\bu\in U}d(\ctheta_{n}(\bu),\theta_{0})>\ep\Big)\leq \Prob\Big(2\sup_{\bu\in U}\sup_{\theta\in \Theta}|\ccL_{n}(\theta;\bu)-\cL(\theta)|>\eta_{\ep}\Big)\to 0. \qedhere
\]
\end{proof}

\begin{proof}[Proof of Proposition \ref{prop:reg_est_rho}]
Let
\[
\Big(\tb_{0}(\bu),\tb_{1}(\bu),\trho(\bu)\Big) \in \argmin_{b_{0} , b_{1} \in \R, K'\leq \rho\leq K''} \tRSS(b_{0},b_{1},\rho;\bu).
\]
Define
\[
\Big(\check b_0(\rho;\bu),\check b_1(\rho;\bu)\Big) := \argmin_{b_{0}, b_{1}\in\R}\tRSS(b_{0},b_{1},\rho;\bu).
\]
Note that in the minimization problem above $\rho$ is fixed, whence $\check b_0(\rho;\bu)$ and $\check b_1(\rho;\bu)$ can be computed explicitly as the solution of a weighted simple linear regression problem. Recalling that by assumption $\sum_{k \in M_n} w_{n,k} = 1$
standard results (or a tedious computation) show that
\[
\tRSS(\check b_0(\rho;\bu),\check b_1(\rho;\bu),\rho;\bu) = \tilde S_{yy}(\bu)(1 - \tcL_{n}(\rho;\bu))
\]  
where 
\begin{align*}
\tilde S_{yy}(\bu) 
&:= 
\sum_{k \in M_n} w_{n,k}\Big\{ \hat C_{n,k}(\bu) - \sum_{i \in M_n} w_{n,i}\hat C_{n,i}(\bu) \Big\}^2,
\\
\tilde S_{xx}(\rho;\bu) &:=  \sum_{k \in M_n} w_{n,k}\Big \{ (k/m)^\rho - \sum_{i \in M_n} w_{n,i}(i/m)^\rho \Big\}^2,
\\
\tilde S_{xy}(\rho;\bu) 
&:= 
\sum_{k \in M_n} w_{n,k} \Big\{ \hat C_{n,k}(\bu) - \sum_{i \in M_n} w_{n,i}\hat C_{n,i}(\bu)\Big\} \Big\{ (k/m)^\rho - \sum_{i \in M_n} w_{n,i} (i/m)^\rho \Big\},
\\
\tcL_{n}(\rho;\bu) &:= \frac{\tilde S_{xy}^2(\rho;\bu)}{\tilde S_{yy}(\bu)\tilde S_{xx}(\rho;\bu)}.
\end{align*} 
From the definitions above it is clear that
\begin{align*}
& \min_{b_{0}, b_{1} \in \R, K'\leq \rho\leq K''} \hRSS_\eta(b_{0},b_{1},\rho;\bu)
\\
&= \min_{ K'\leq \rho\leq K''}\Big( \min_{b_{0}, b_{1} \in \R} \tRSS(b_0,b_1,\rho;\bu) + \frac{\eta}{|\rho|}\tRSS(\tb_{0}(\bu),\tb_{1}(\bu),\trho(\bu);\bu)\Big) 
\\
&= \tilde S_{yy}(\bu) \min_{ K'\leq \rho\leq K''}\Big\{ 1 - \tcL_{n}(\rho;\bu)  + \frac{\eta}{|\rho|}\Big(1 - \tcL_{n}(\tilde\rho(\bu);\bu) \Big)\Big\}. 
\end{align*}
Hence,
\[
\hat\rho(\bu) \in \argmax_{K'\leq \rho\leq K''} \Big\{\tcL_{n}(\rho;\bu) - \frac{\eta}{|\rho|}\Big(1 - \tcL_{n}(\tilde\rho(\bu);\bu) \Big)\Big\}
\]
and by similar but simpler arguments
\[
\trho(\bu) \in \argmin_{K'\leq \rho \leq K''}\tRSS\Big(\check b_{0}(\rho;\bu),\check b_{1}(\rho;\bu),\rho;\bu\Big) = \argmax_{K'\leq \rho\leq K''} \tcL_{n}(\rho;\bu).
\]
Next observe that by an application of Theorem~\ref{theo:estmar}, regular variation of $\ca(\cdot)$, and straightforward calculations
\begin{align*}
\sup_{\bu \in U} \Big|\frac{\tilde S_{yy}(\bu)}{\{\ca(m)S(\bu)\}^2} - \int_A f(a)(a^{\rho_\ca} - \mu_{\rho_\ca})^2   \diff a\Big| =\op{1}.
\\
\sup_{\bu \in U} \Big| \tilde S_{xx}(\rho;\bu)  - \int_A f(a)(a^\rho - \mu_\rho )^2  \diff a\Big| = \op{1}.
\\
\sup_{\bu \in U} \Big| \frac{\tilde S_{xy}(\rho;\bu)}{\ca(m)S(\bu)} - \int_A f(a)(a^{\rho_\ca} - \mu_{\rho_\ca})(a^\rho - \mu_\rho)   \diff a\Big| = \op{1}.
\end{align*}
where $\mu_\rho := \int_A a^\rho f(a)  \diff a$. Next define 
\[
\cL(\rho) = \frac{\Big\{\int_{A}(a^{\rho}-\mu_\rho)(a^{\rho_\ca}-\mu_{\rho_\ca} ) f(a) \diff a \Big\}^{2}}{\int_{A}(a^{\rho}-\mu_\rho)^{2} f(a) \diff a \int_{A}(a^{\rho_\ca}-\mu_{\rho_\ca})^{2}f(a) \diff a }.
\]
The arguments given above show that
\begin{equation}\label{eq:unifL}
\sup_{K'\leq \rho\leq K''}\sup_{\bu \in U} |\tcL_{n}(\rho;\bu) - \cL(\rho)| = o_P(1).
\end{equation}
Next we show that $\cL$ satisfies~\eqref{eq:M_cond_2} with $\theta_0 = \rho_\ca$ and $\Theta = [K',K'']$. Since $\cL$ is continuous, $\{\rho\in [K',K'']:|\rho-\rho_\ca|\geq \ep\}$ is compact, $\cL(\rho)\leq 1$ by Cauchy-Schwarz, and $\cL(\rho_\ca)=1$, it suffices to show that 
\begin{equation*}
    \cL(\rho)=1, \quad \text{only if} \ \rho=\rho_\ca.
\end{equation*} 
This, however, follows again from Cauchy-Schwarz and linear independence of the functions $A \ni a \mapsto a^\rho$ and $A\ni a \mapsto a^\rho$ for $\rho \neq \rho_\ca$.

Next, apply Lemma \ref{lemma:M_est} with
\begin{equation*}
\Theta=[K',K''], \quad U=U,\quad  \ccL_{n}=\tcL_{n},\quad \cL=\cL, \quad\ctheta_{n}=\trho, \quad\theta_{0}=\rho_\ca.
\end{equation*}
Condition \eqref{eq:M_cond_1} follows directly from~\eqref{eq:unifL} and hence, by the first part of Lemma \ref{lemma:M_est} and by the fact that $\tcL_{n}(\tilde\rho(\bu);\bu) \le 1$ by Cauchy-Schwarz, we have
\[
\tcL_{n}(\tilde\rho(\bu);\bu) = \cL(\rho_\ca) + \op{1} = 1 + \op{1}
\]
uniformly in $\bu \in U$. This implies
\[
\hcL_{n}(\bu,\rho) := \tcL_{n}(\rho;\bu) - \frac{\eta}{|\rho|}\Big(1 - \tcL_{n}(\tilde\rho(\bu);\bu) \Big) = \cL(\rho) + \op{1}
\]
uniformly in $\rho \in [K',K''], \bu \in U$. Hence, we may apply Lemma \ref{lemma:M_est} again, with
\begin{equation*}
\Theta=[K',K''], \quad U=U, \quad \ccL_{n}=\hcL_{n}, \quad \cL=\cL, \quad \ctheta_{n}=\hat \rho, \quad \theta_{0}=\rho_\ca,
\end{equation*}
and the result follows by the definition of $\hat \rho(\bu)$.
\end{proof} 

\section{Derivations for Section~4}
\label{sec:proof3}

\begin{proof}[Proof  of the claims regarding Example~\ref{ex:movmax}]
By a straight-forward calculation, see also \cite{BucSeg14}, the copula $C_m$ is given by
\[
C_m(\bu) = \prod_{s=1-p}^m D( (u_j^{\beta_{mjs}})_{j=1}^d) ,
\]
where 
\begin{align*}
\beta_{mjs} = \frac{\alpha_{mjs}}{\alpha_{mj\bullet}}, \quad
\alpha_{mjs} =  \max(a_{ij}: i=\max(1-s,0), \dots, \min(m-s,p)\}, \quad
\alpha_{mj\bullet} = \sum_{s=1-p}^m \alpha_{mjs}.
\end{align*}
Without loss of generality assume that $m > p$. For all $s=1, \dots, m-p$, we have
\[
\alpha_{mjs} = A_{j} := \max(a_{0j}, \dots, a_{pj}),
\]
whence we may rewrite
\begin{align} \label{eq:cmmov}
C_m(\bu) = \prod_{s=1-p}^0 D\big( (u_j^{\beta_{mjs}})_{j=1}^d\big) \times  \big\{ D\big( (u_j^{B_{mj}})_{j=1}^d\big) \big\}^{m-p} \times  \prod_{s=m-p+1}^m D\big( (u_j^{\beta_{mjs}})_{j=1}^d\big),
\end{align}
where
\[
B_{mj} = \frac{A_j}{\alpha_{mj\bullet}} = \frac{A_j}{(m-p) A_j +  c_{j}}, \qquad  
c_{j} = \sum_{s=1-p}^0  \alpha_{mjs} + \sum_{s=m-p+1}^m  \alpha_{mjs}.
\]
Note that $c_{j}$ does not depend on $m$.

Now, using the Taylor expansion $u^x= \exp(x \log u) = 1+x\log u + O(x^2)$ for $x\to0$ (which is uniform in $u$ bounded away from zero), each factor in the two products on the right-hand side of \eqref{eq:cmmov} can be written as
\[
D\big( (u_j^{\beta_{mjs}})_{j=1}^d \big) = D\big( (1+ \tfrac{\alpha_{mjs}}{(m-p)A_j +  c_{j} } \log u_j + O(m^{-2}) )_{j=1}^d\big) = 1+O(1/m),
\]
where we have used Lipschitz-continuity of $D$ and where the $O$-term is uniform on $[\delta,1]^d$ for any $\delta>0$. The factor in the middle of the right-hand side of \eqref{eq:cmmov} can be rewritten as
\begin{align*}
D\big( (u_j^{B_{mj}})_{j=1}^d\big)^{m-p}
&=
D_{m-p} \big( (u_j^{(m-p) B_{mj}})_{j=1}^d\big) \\
&=
D_\infty\big( (u_j^{(m-p) B_{mj}})_{j=1}^d\big) + \ca_D(m-p)  S_D\big( (u_j^{(m-p) B_{mj}})_{j=1}^d\big) + o(a(m)).
\end{align*}
Since we have, again uniform on $[\delta,1]^d$ for any $\delta>0$,
\[
u_j^{(m-p) B_{mj}} 
=
u_j \Big\{ 1 - \frac{c_{j}}{(m-p)A_j + c_{j}}  \log u_j  + O(m^{-2}) \Big\} = u_j + O(1/m),
\]
we may use Lipschitz-continuity of $D_\infty$ and continuity of $S_D$ to obtain that
\begin{align*}
D\big( (u_j^{B_{mj}})_{j=1}^d\big)^{m-p}
&=
D_\infty( \bu ) + \ca_D(m)  S_D(\bu) + o(\ca_D(m)) + O(1/m).
\end{align*}

Assembling terms, and additionally assuming that $1/m = o(\ca_D(m))$, we finally obtain that
\begin{align*}
C_m(\bu) 
&= 
(1+o(\ca_D(m)))^{2p} \Big\{ D_\infty( \bu ) + \ca_D(m)  S_D(\bu) + o(a(m))\Big\} \\
&= D_\infty( \bu ) + \ca_D(m)  S_D(\bu) + o(a(m)),
\end{align*}
uniformly on $[\delta, 1]^d$.  In order to obtain the result uniformly on $[0,1]^d$, one may use similar arguments as in the proof of Lemma 2.4 in \cite{BucVolZou19}, explicitly making use of the special structure of $C_m$ in \eqref{eq:cmmov} with the middle factor being equal to $D_{m-p} \big( (u_j^{(m-p) B_{mj}})_{j=1}^d\big)$. 
\end{proof}

\section{Additional simulation results in higher dimensions}\label{sec:addsim}

\subsection{Comparison of bias-corrected estimators in higher dimensions} \label{sec:highdim}

In this section we compare the performance of the bias-corrected estimators when $d=4,8$. The simulation settings are similar to those in Section \ref{sec:finiteprop}; specifically, each estimator is computed for all values $\bu \in \cU$, where $\cU = \{.25,.50,.75\}^4$ when $d=4$ and $\cU = \{.25,.75\}^8$ when $d=8$, and block size $m \in \{1,\dots,20\}$ (except for the aggregated versions, for which we specify the set of block length parameters below). Squared bias, variance and MSE of each estimator and in each point $\bu \in \cU$ for sample size $n = 1000$ was estimated based on $1000$ Monte Carlo replications. For the sake of brevity we only report summary results which correspond to taking averages of the squared bias, MSE and variance over all values $\bu \in \cU$. 

When $d=4,8$, as in Section \ref{sec:finiteprop}, we generate data from the $t$-Copula and the outer-power transformation of Clayton Copula; see Examples \ref{example:thighdim} and \ref{example:outerhighdim} below. In particular, we present results on the following models:

\begin{enumerate}
\item[(M6)] iid realizations from an outer Power Clayton Copula with $d = 4, \theta = 1, \beta = \log(2)/\log(2-0.25)$. 
\item[(M7)] A moving maximum process based on the outer Power Clayton Copula with $d=4, \theta = 1, \beta = \log(2)/\log(2-0.25)$ and  $a_{1(2k-1)} = 0.25, a_{1(2k)} = 0.75$, $k = 1, 2$.   
\item[(M8)] iid realizations from a $t$-Copula with $d = 4, \nu = 5, \theta = 0.5$.  
\item[(M9)] A moving maximum process based on a $t$-Copula with $d = 4, \nu = 5, \theta = 0.5$ and $a_{1(2k-1)} = 0.25, a_{1(2k)} = 0.75$, $k = 1, 2$.  
\item[(M10)] A moving maximum process based on a $t$-Copula with $d = 4, \nu = 3, \theta = 0.25$ and $a_{1(2k-1)} = 0.25, a_{1(2k)} = 0.75$, $k = 1, 2$.  
\item[(M11)] iid realizations from an outer Power Clayton Copula with $d = 8, \theta = 1, \beta = \log(2)/\log(2-0.25)$. 
\item[(M12)] A moving maximum process based on the outer Power Clayton Copula with $d=8, \theta = 1, \beta = \log(2)/\log(2-0.25)$ and $a_{1(2k-1)} = 0.25, a_{1(2k)} = 0.75$, $k = 1, 2, 3, 4$.  
\item[(M13)] iid realizations from a $t$-Copula with $d = 8, \nu = 5, \theta = 0.5$.  
\item[(M14)] A moving maximum process based on a $t$-Copula with $d = 8, \nu = 5, \theta = 0.5$ and $a_{1(2k-1)} = 0.25, a_{1(2k)} = 0.75$, $k = 1, 2, 3, 4$.  
\item[(M15)] A moving maximum process based on a $t$-Copula with $d = 8, \nu = 3, \theta = 0.25$ and $a_{1(2k-1)} = 0.25, a_{1(2k)} = 0.75$, $k = 1, 2, 3, 4$.
\end{enumerate}
For the sake of brevity, we do not include an iid version of Model (M10) and (M15) because the findings are very similar to the time series case. Further note that we also investigated other parameter combinations, but chose to only present results for the above models as they provide, to a large extent, a representative subset of the results.

We implement the same collection of bias-corrected estimators and choose the same tuning parameters as in Section \ref{sec:finiteprop}. The results are presented in Figure \ref{fig:bc4} and Figure \ref{fig:bc8}. Notice that when $d=4, 8$, for $\bu\in \cU$, $C_{\infty}(\bu)$ in general gets closer to zero, and as a result the variances of the estimators become smaller compared to their squared biases. Other than that, when $d=4, 8$, the patterns of the bias-corrected estimators are very similar to the patterns when $d=2$; in particular, when $d=4, 8$, the superiority of the aggregated naive and the regression-based corrected estimators is preserved. Overall, for a generic dimension $d$, we would recommend using the aggregated bias-corrected estimator among all bias-corrected estimators since it leads to better results than the naive estimator, is reasonably fast to compute (see Section \ref{sec:comptime}), and is simpler to implement than the regression-based estimator.

\begin{figure}[t!]
\begin{center}
%\vspace{-.1cm}
\includegraphics[width=0.82\textwidth]{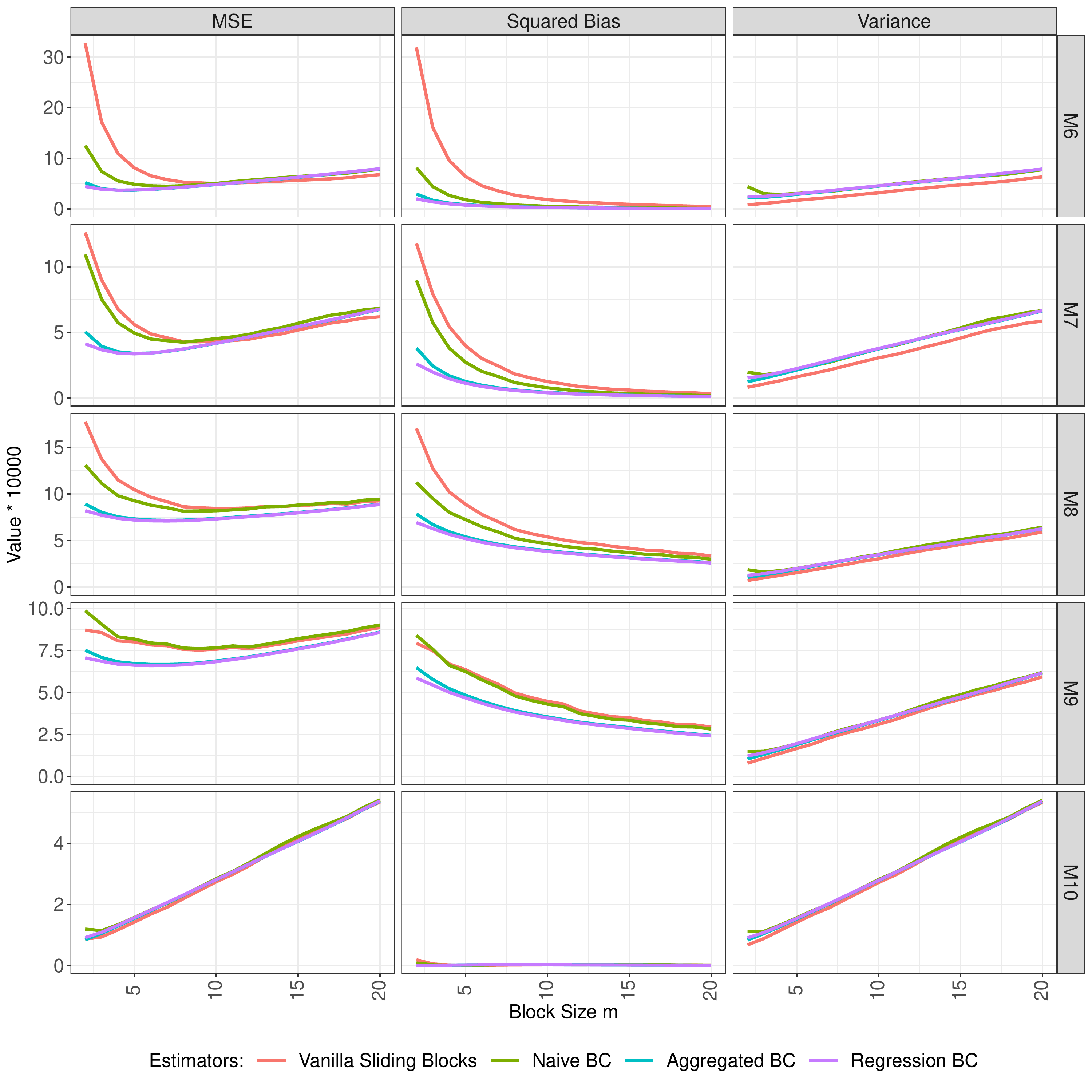}
\vspace{-.6cm}
\end{center}
\caption{\label{fig:bc4} 
$10^{4} \times $ average MSE, average squared bias and average variance of sliding blocks estimator,  naive bias corrected estimator, and aggregated naive bias corrected estimator.} 
\vspace{-.3cm}
\end{figure}

\begin{figure}[t!]
\begin{center}
%\vspace{-.1cm}
\includegraphics[width=0.82\textwidth]{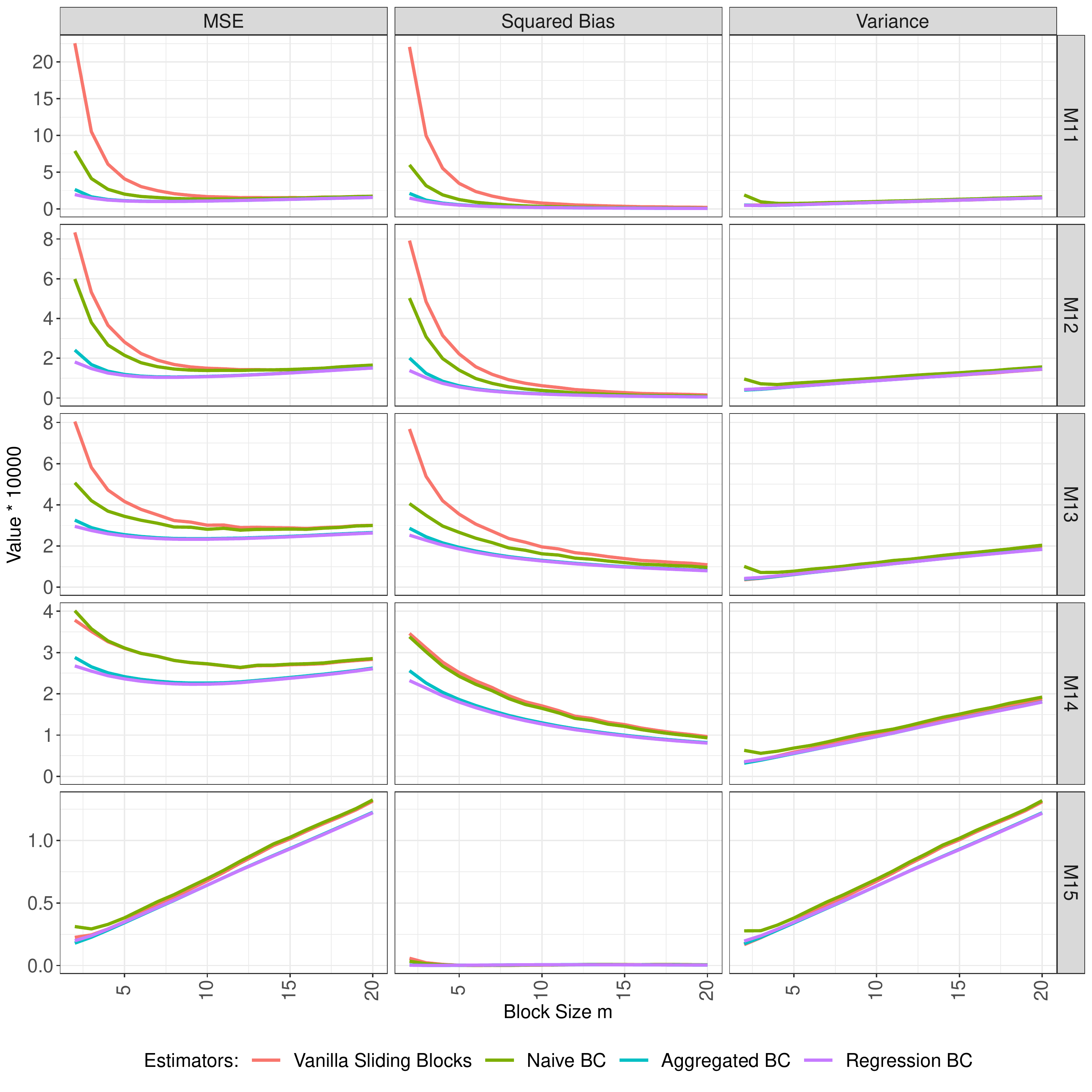}
\vspace{-.6cm}
\end{center}
\caption{\label{fig:bc8} 
$10^{4} \times $ average MSE, average squared bias and average variance of sliding blocks estimator,  naive bias corrected estimator, and aggregated naive bias corrected estimator.} 
\vspace{-.3cm}
\end{figure}

\begin{example}[$t$-Copula, iid case, $d\geq 2$]\label{example:thighdim}
When $d\geq2$, the $t$-copula is defined as
\[
D(\bu;\nu,\theta)=\int_{-\infty}^{t_{\nu}^{-1}(u_{1})}\cdots\int_{-\infty}^{t_{\nu}^{-1}(u_{d})}\frac{\Gamma\Big(\frac{\nu+d}{2}\Big)}{\Gamma\big(\frac{\nu}{2}\big)\sqrt{(\pi\nu)^{d} |P|}}\bigg(1+\frac{\bx'P^{-1}\bx}{\nu}\bigg)^{-\frac{\nu+d}{2}}\diff \bx, \quad \bu=(u_{1},\dots,u_{d})\in [0,1]^{d},
\]
where $P$ is a $d\times d$ correlation matrix and  $t_{\nu}$ is the cumulative distribution function of a standard univariate $t$-distribution with degrees of freedom $\nu$. When the off-diagonal entries of $P$ all equal $\theta$, with straightforward calculations, the stable tail dependence function of the $d$-dimensional $t$-copula is given by
\begin{align*}
L(\bx)=\sum_{j=1}^{d}x_{j}t_{(\bm{0},\tilde{P}-\theta^{2}\bm{1}\bm{1}',\nu+1)}\Bigg(&\sqrt{\nu+1}\bigg[\Big(\frac{x_{j}}{x_{1}}\Big)^{1/\nu}-\theta\bigg],\dots,\sqrt{\nu+1}\bigg[\Big(\frac{x_{j}}{x_{j-1}}\Big)^{1/\nu}-\theta\bigg],\\
&\sqrt{\nu+1}\bigg[\Big(\frac{x_{j}}{x_{j+1}}\Big)^{1/\nu}-\theta\bigg],\dots, \sqrt{\nu+1}\bigg[\Big(\frac{x_{j}}{x_{d}}\Big)^{1/\nu}-\theta\bigg]\Bigg)
\end{align*}
where $t_{(\bm{0},\tilde{P}-\theta^{2}\bm{1}\bm{1}',\nu+1)}$ is the cumulative distribution function of a $(d-1)$-dimensional $t$-distribution with mean $\bm{0}$, shape matrix $\tilde{P}-\theta^{2}\bm{1}\bm{1}'$, and degrees of freedom $\nu-1$; $\bm{0}$ is a $(d-1)$-dimensional column vector with all entries 0, $\bm{1}$ is a $(d-1)$-dimensional column vector with all entries 1, and $\tilde{P}$ is a $(d-1)\times(d-1)$-dimensional matrix with all diagonal entries 1 and all off-diagononal entries $\theta$.
\end{example}

\begin{example}[Outer-power transformation of Clayton Copula, iid case, $d\geq 2$]\label{example:outerhighdim}
When dimension $d\geq 2$, the outer-power transformation of a Clayton Copula is defined as 
\[
D(\bu;\theta, \beta)=\bigg[1+\Big\{\sum_{j=1}^{d}(u_{j}^{-\theta}-1)^{\beta}\Big\}^{1/\beta}\bigg]^{-1/\theta},\quad \bu=(u_{1},\dots,u_{d})\in [0,1]^{d}.
\]
Notice that $D$ is in the copula domain of attraction of the Gumbel--Hougaard Copula with shape parameter $\beta$, defined by
\[
D_{\infty}(\bu)=D(\bu;\beta)\coloneqq \exp\bigg[-\Big\{\sum_{j=1}^{d}(-\log u_{j})^{\beta}\Big\}^{1/\beta}\bigg], \quad \bu=(u_{1},\dots,u_{d})\in [0,1]^{d}. 
\]
\end{example}

\subsection{Computation time of bias-corrected estimators} \label{sec:comptime}

Computation times for evaluating each bias-corrected estimator at a single point $\bm u \in [0,1]^d$ for dimensions $d = 2,4,8$ and different sample sizes and all values of $m \in \{1,\dots,100\}$ are collected in Table~\ref{tab1} with the \texttt{R} implementation used in our simulations on a standard laptop. Notice that for even the largest sample size ($n = 5000$) and dimension $d=8$, the average computation time for the bias-corrected estimators does not exceed $5$ seconds. Hence, in even moderate dimensions and with reasonably large sample sizes, the bias-corrected estimators can be computed fairly quickly. Further, notice that a substantial proportion of the computation time for all bias-corrected estimators is used for computing $\hat \rho$. Since $\hat \rho$ needs to be computed only once, the evaluation of estimators at additional points will be much cheaper.

\begin{table}[H]
\begin{center}
\begin{tabular}{|l|c|c|c|c|c|c|}
\hline
& \multicolumn{2}{c|}{$d=2$} & \multicolumn{2}{c|}{$d=4$} & \multicolumn{2}{c|}{$d=8$} \\ 
\hline
& $n=1000$ & $n=5000$ & $n=1000$ & $n=5000$ & $n=1000$ & $n=5000$ \\ 
\hline
$\hat C_{n,m}$ & 6 & 28 & 10 & 45 & 20 & 116 \\ 
\hline
$\hat C_{n,(m,m')}^{bc,nai}$ & 95 & 171 & 107 & 271 & 140 & 448 \\ 
\hline
$\hat C_{n,(m',M,w)}^{bc,agg}$ & 98 & 181 & 111 & 275 & 145 & 450 \\ 
\hline
$\hat C_{n,(M,w)}^{bc,reg}$ & 96 & 182 & 117 & 263 & 127 & 452 \\ 
\hline
$\hat \rho$ & 85 & 148 & 98 & 203 & 117 & 302 \\ 
\hline
\end{tabular}
\end{center} 
\caption{\label{tab1}
Run time (in seconds) for evaluating a single estimator 100 times at a single point $\bu = (0.5,\dots,0.5)$ and for all values $m \in \{1,\dots,100\}$ (for $\hat C_{n,m}$), or for all $m' =1, m \in \{1,\dots,60\}$ (for $\hat C_{n,(m,m')}^{bc,nai}$), or for all $m' = 1, M = \{m,\dots,m+9\}, m \in \{2,\dots,50\}$ (for $\hat C_{n,(m',M,w)}^{bc,agg}$) or for all $M = \{1,m,\dots,m+9\}, m \in \{2,\dots,50\}$ (for $\hat C_{n,(M,w)}^{bc,reg}$). Computation times for bias-corrected estimators include the time to compute $\hat \rho$.} 
\end{table}

\bibliographystyle{apalike}
\bibliography{oevc}

\end{document}